\documentclass[reqno,10pt,centertags]{amsart} 
\usepackage{amsmath,amsthm,amscd,amssymb,latexsym,esint,upref,stmaryrd,
enumerate,color,verbatim,yfonts}
\usepackage{hyperref} 
\newcommand*{\mailto}[1]{\href{mailto:#1}{\nolinkurl{#1}}}
\newcommand{\arxiv}[1]{\href{http://arxiv.org/abs/#1}{arXiv:#1}}

\newcommand{\R}{{\mathbb R}}
\newcommand{\N}{{\mathbb N}}
\newcommand{\Z}{{\mathbb Z}}
\newcommand{\C}{{\mathbb C}}

\newcommand{\bbC}{{\mathbb{C}}}

\newcommand{\bbN}{{\mathbb{N}}}

\newcommand{\bbR}{{\mathbb{R}}}

\newcommand{\bbZ}{{\mathbb{Z}}}

\newcommand{\beq}{\begin{equation}}
\newcommand{\enq}{\end{equation}}

\renewcommand{\a}{\alpha}
\renewcommand{\b}{\beta}
\newcommand{\g}{\gamma}
\renewcommand{\d}{\delta}

\renewcommand{\l}{\lambda}

\newcommand{\s}{\sigma}

\newcommand{\G}{\Gamma}

\DeclareMathOperator{\supp}{supp}

\DeclareMathOperator{\dom}{dom}

\renewcommand{\Re}{\text{\rm Re}}
\renewcommand{\Im}{\text{\rm Im}}
\renewcommand{\ln}{\text{\rm ln}}

\newcommand{\no}{\notag}
\newcommand{\lb}{\label}
\newcommand{\f}{\frac}

\newcommand{\ol}{\overline}
\newcommand{\bs}{\backslash}

\newcommand{\wti}{\widetilde}
\newcommand{\Oh}{O}
\newcommand{\oh}{o}
\newcommand{\hatt}{\widehat} 
\newcommand{\dott}{\,\cdot\,}

\renewcommand{\dot}{\overset{\textbf{\Large.}}}

\newcommand{\bi}{\bibitem}

\let\geq\geqslant
\let\leq\leqslant

\newcommand{\la}{\lambda}
\newcommand{\lam}{\lambda}

\newcommand{\al}{\alpha}
\newcommand{\be}{\beta}

\newcommand{\Lr}{{L^2((a,b);rdx)}} 

\newcommand{\ACl}{{AC_{loc}((a,b))}}
\newcommand{\Ll}{{L^1_{loc}((a,b);dx)}}
\newcommand{\Pn}{P_n^{\alpha,\beta}(x)}

\makeatletter
\def\theequation{\@arabic\c@equation}

\allowdisplaybreaks 
\numberwithin{equation}{section}

\newtheorem{theorem}{Theorem}[section]
\newtheorem{proposition}[theorem]{Proposition}
\newtheorem{lemma}[theorem]{Lemma}
\newtheorem{corollary}[theorem]{Corollary}
\newtheorem{definition}[theorem]{Definition}
\newtheorem{hypothesis}[theorem]{Hypothesis}
\newtheorem{example}[theorem]{Example}

\theoremstyle{remark}
\newtheorem{remark}[theorem]{Remark}

\begin{document}

\title[Jacobi $m$-functions]{The Jacobi Operator on $(-1,1)$ and \\ its Various $m$-Functions} 

\author[F.\ Gesztesy]{Fritz Gesztesy}
\address{Department of Mathematics, 
Baylor University, Sid Richardson Bldg., 1410 S.\,4th Street, Waco, TX 76706, USA}
\email{\mailto{Fritz\_Gesztesy@baylor.edu}}
\urladdr{\url{http://www.baylor.edu/math/index.php?id=935340}}

\author[L.\ L.\ Littlejohn]{Lance L. Littlejohn}
\address{Department of Mathematics, 
Baylor University, Sid Richardson Bldg., 1410 S.\,4th Street, Waco, TX 76706, USA}
\email{\mailto{Lance\_Littlejohn@baylor.edu}}
\urladdr{\url{http://www.baylor.edu/math/index.php?id=53980}}

\author[M. Piorkowski]{Mateusz Piorkowski}
\address{Department of Mathematics, KU Leuven \\
Celestijnenlaan 200B, 3001 Leuven, Belgium}
\email{\href{mailto:Mateusz.Piorkowski@kuleuven.be}{Mateusz.Piorkowski@kuleuven.be}}
\urladdr{\url{https://www.kuleuven.be/wieiswie/en/person/00153535}}

\author[J.\ Stanfill]{Jonathan Stanfill}
\address{Department of Mathematics, The Ohio State University \\
100 Math Tower, 231 West 18th Avenue, Columbus, OH 43210, USA}
\email{\mailto{stanfill.13@osu.edu}}
\urladdr{\url{https://u.osu.edu/stanfill-13/}}


\date{\today}
\@namedef{subjclassname@2020}{\textup{2020} Mathematics Subject Classification}

\begin{abstract} 
We offer a detailed treatment of spectral and Weyl--Titchmarsh--Kodaira theory for all self-adjoint Jacobi operator realizations of the differential expression 
\begin{align*} 
\tau_{\a,\b} = - (1-x)^{-\a} (1+x)^{-\b}(d/dx) \big((1-x)^{\a+1}(1+x)^{\b+1}\big) (d/dx),&     \\ 
\a, \b \in \bbR, \; x \in (-1,1),&  
\end{align*}
in $L^2\big((-1,1); (1-x)^{\a} (1+x)^{\b} dx\big)$, $\a, \b \in \bbR$. In addition to discussing the separated boundary conditions that lead to Jacobi orthogonal polynomials as eigenfunctions in detail, we exhaustively treat the case of coupled boundary conditions and illustrate the latter with the help of the general $\eta$-periodic and Krein--von Neumann extensions. In particular, we treat all underlying Weyl--Titchmarsh--Kodaira and Green's function induced $m$-functions and revisit their Nevanlinna--Herglotz property. We also consider connections to other differential operators associated with orthogonal polynomials such as Laguerre, Gegenbauer, and Chebyshev.
\end{abstract}

\maketitle

{\scriptsize{\tableofcontents}}

\section{Introduction} \lb{s1} 

Despite the prominence of the Jacobi differential expression 
\begin{align}
\begin{split}  
\tau_{\a,\b} = - (1-x)^{-\a} (1+x)^{-\b}(d/dx) \big((1-x)^{\a+1}(1+x)^{\b+1}\big) (d/dx),&     \\ 
\a, \b \in \bbR, \; x \in (-1,1),&    \lb{1.1} 
\end{split}     
\end{align}
in the orthogonal polynomial community (see, e.g., \cite[Ch.~22]{AS72}, \cite{AA85}, 
\cite[Ch.~6]{AAR99}, \cite[Ch.~4]{BW10}, 
\cite[Ch.~X]{EMOT53a}, \cite[Ch.~4]{Is05}, \cite[Ch.~V]{MOS66}, \cite[Ch.~18]{OLBC10}, \cite[Ch.~IV]{Sz75}), the actual spectral and Weyl--Titchmarsh--Kodaira theory for self-adjoint realizations associated with \eqref{1.1} has received relatively little attention, with the notable exception of the papers by 
Bush, Frymark, and Liaw \cite{BFL20} discussing the case $\a,\b\in(0,1)$, as well as Frymark and Liaw \cite{FL20} and Frymark \cite{Fr20} discussing the case $\a,\b\in[0,1)$. This paper aims at filling this gap and presents a discussion for the general case $\a, \b \in \bbR$. In particular, we derive the Weyl--Titchmarsh--Kodaira $m$-functions (for separated boundary conditions) and the Green's function induced $M$-functions (for coupled boundary conditions), for all self-adjoint realizations associated with \eqref{1.1}. We note that in the context of Donoghue $m$-functions, the general case was recently studied in \cite{GPS23} when at least one endpoint was in the limit circle case (see also \cite{GLNPS21} for background). In particular, the Donoghue $m$-functions associated with all coupled boundary conditions were explicitly constructed in \cite{GPS23}, which, to the best of our knowledge, was the first complete study of coupled boundary conditions in this setting.

We now turn to the basic structure of the present paper. Based on \cite{GLN20} and \cite[Ch.~13]{GNZ23}, we recall the basics of Weyl--Titchmarsh--Kodaira theory in Section \ref{s2}, and explicitly treat the corresponding $m$ and $M$ functions for separated and coupled boundary conditions, respectively, in Section \ref{s3}. In particular, the basics of the Green's function induced $M$-function in the presence of coupled boundary conditions is recalled in the last part of Section \ref{s3}. The minimal and maximal Jacobi operators and associated boundary values employing principal and nonprincipal solutions are discussed in Section \ref{s4}. 

Our main results on the explicit construction of Jacobi operator $m$ and $M$-functions are presented in Section \ref{s5}--\ref{s7}. Section \ref{s5} treats the regular and limit circle cases of $\tau_{\a,\b}$ corresponding to $\a, \b \in (-1,1)$, for both separated and coupled boundary conditions; the case where $\tau_{\a,\b}$ is in the limit point case at precisely one of the interval endpoints $\pm 1$ is treated in Section \ref{s6}; the case where $\tau_{\a,\b}$ is in the limit point case at $-1$ and $1$ is presented in Section \ref{s7}. In Section \ref{s8} we revisit situations where the Weyl--Titchmarsh--Kodaira $m$-function possesses the Nevanlinna--Herglotz property with particular emphasis on the case supporting Jacobi polynomials as eigenfunctions.

Appendix \ref{sA} on solutions of the hypergeometric and Jacobi differential equations, Appendix \ref{sB} on connection formulas, and Appendix \ref{sC} on the behavior of solutions near the endpoints $\pm1$, complement the bulk of this paper. Moreover, Appendix \ref{sD} addresses how to obtain Weyl--Titchmarsh--Kodaira $m$-functions for the Laguerre operator as a limit of the Weyl--Titchmarsh--Kodaira $m$-functions constructed in Sections \ref{s5}--\ref{s6} (see also Section \ref{s8}). Appendix \ref{sE} provides a brief discussion of special cases of the Jacobi differential expression, including the Gegenbauer (ultraspherical), Chebyshev, and radial part of the Zernike differential expressions. Some material in these appendices was also used in \cite{GPS23}. 

Finally, we briefly comment on some of the basic notation used throughout this paper.  If $T$ is a linear operator mapping (a subspace of) a Hilbert space into another, then $\dom(T)$ and $\ker(T)$ denote the domain and kernel (i.e., null space) of $T$. The spectrum and resolvent set of a closed linear operator in a Hilbert space will be denoted by $\sigma(\cdot)$ and $\rho(\cdot)$, respectively. We also use the shorthand notation $\bbN_0 = \bbN \cup \{0\}$ and  $\bbC_{\pm} = \{z \in \bbC \,|\, \pm \Im(z) > 0\}$.

\section{The Basics of Weyl--Titchmarsh--Kodaira Theory} \lb{s2}

In this section, following \cite{GLN20} and \cite[Ch.~13]{GNZ23}, we summarize the singular 
Weyl--Titchmarsh--Kodaira theory as needed to treat the Jacobi operator in the remainder of this paper. 

Throughout this section we make the following assumptions:

\begin{hypothesis} \lb{h2.1}
Let $(a,b) \subseteq \bbR$ and suppose that $p,q,r$ are $($Lebesgue\,$)$ measurable functions on $(a,b)$ 
such that the following items $(i)$--$(iii)$ hold: \\[1mm] 
$(i)$ \hspace*{1.1mm} $r>0$ a.e.~on $(a,b)$, $r\in\Ll$. \\[1mm] 
$(ii)$ \hspace*{.1mm} $p>0$ a.e.~on $(a,b)$, $1/p \in\Ll$. \\[1mm] 
$(iii)$ $q$ is real-valued a.e.~on $(a,b)$, $q\in\Ll$. 
\end{hypothesis}

Given Hypothesis \ref{h2.1}, we study Sturm--Liouville operators associated with the general, 
three-coefficient differential expression
\begin{equation}
\tau=\f{1}{r(x)}\left[-\f{d}{dx}p(x)\f{d}{dx} + q(x)\right] \, \text{ for a.e.~$x\in(a,b) \subseteq \bbR$,} 
   \lb{2.1}
\end{equation} 
and introduce maximal and minimal operators in $\Lr$ associated with $\tau$ in the usual manner as follows. 

\begin{definition} \lb{d2.2}
Assume Hypothesis \ref{h2.1}. Given $\tau$ as in \eqref{2.1}, the \textit{maximal operator} $T_{max}$ in $\Lr$ associated with $\tau$ is defined by
\begin{align}
&T_{max} f = \tau f,    \no
\\
& f \in \dom(T_{max})=\big\{g\in\Lr \, \big| \,g,g^{[1]}\in\ACl;   \lb{2.2} \\ 
& \hspace*{6.3cm}  \tau g\in\Lr\big\}.   \no
\end{align}
The \textit{preminimal operator} $\dot T_{min} $ in $\Lr$ associated with $\tau$ is defined by 
\begin{align}
&\dot T_{min}  f = \tau f,   \no
\\
&f \in \dom \big(\dot T_{min}\big)=\big\{g\in\Lr \, \big| \, g,g^{[1]}\in\ACl;   \lb{2.3}
\\
&\hspace*{3.25cm} \supp \, (g)\subset(a,b) \text{ is compact; } \tau g\in\Lr\big\}.   \no
\end{align}

One can prove that $\dot T_{min} $ is closable, and one then defines the \textit{minimal operator} $T_{min}$ as the closure of $\dot T_{min} $.
\end{definition}

The following facts then are well known:
\begin{equation} 
\big(\dot T_{min}\big)^* = T_{max}, 
\end{equation} 
and hence $T_{max}$ is closed and $T_{min}=\ol{\dot T_{min} }$ is given by
\begin{align}
&T_{min} f = \tau f, \no
\\
&f \in \dom(T_{min})=\big\{g\in\Lr  \, \big| \,  g,g^{[1]}\in\ACl;     \lb{2.5} \\
& \qquad \text{for all } h\in\dom(T_{max}), \, W(h,g)(a)=0=W(h,g)(b); \, \tau g\in\Lr\big\}   
\no \\
& \quad =\big\{g\in\dom(T_{max})  \, \big| \, W(h,g)(a)=0=W(h,g)(b) \, 
\text{for all } h\in\dom(T_{max}) \big\}.   \no 
\end{align}
Moreover, $\dot T_{min} $ is essentially self-adjoint if and only if\; $T_{max}$ is symmetric, and then 
$\ol{\dot T_{min} }=T_{min}=T_{max}$.

Here the Wronskian of $f$ and $g$, for $f,g\in\ACl$, is defined by
\begin{equation}
W(f,g)(x) = f(x)g^{[1]}(x) - f^{[1]}(x)g(x), \quad x \in (a,b),    \lb{23.2.3.1} \no \\
\end{equation}
with 
\begin{equation}
y^{[1]}(x) = p(x) y'(x), \quad x \in (a,b),
\end{equation}
denoting the first quasi-derivative of a function $y\in AC_{loc}((a,b))$.

The celebrated Weyl alternative then can be stated as follows:

\begin{theorem}[Weyl's Alternative] \lb{t2.3} ${}$ \\
Assume Hypothesis \ref{h2.1}. Then the following alternative holds: Either \\[1mm] 
$(i)$ for every $z\in\bbC$, all solutions $u$ of $(\tau-z)u=0$ are in $\Lr$ near $b$ 
$($resp., near $a$$)$, \\[1mm] 
or, \\[1mm]
$(ii)$  for every $z\in\bbC$, there exists at least one solution $u$ of $(\tau-z)u=0$ which is not in $\Lr$ near $b$ $($resp., near $a$$)$. In this case, for each $z\in\bbC\bs\bbR$, there exists precisely one solution $u_b$ $($resp., $u_a$$)$ of $(\tau-z)u=0$ $($up to constant multiples$)$ which lies in $\Lr$ near $b$ $($resp., near $a$$)$. 
\end{theorem}

This yields the limit circle/limit point classification of $\tau$ at an interval endpoint and links self-adjointness of $T_{min}$ (resp., $T_{max}$) and the limit point property of $\tau$ at both endpoints as follows. 

\begin{definition} \lb{d2.4} 
Assume Hypothesis \ref{h2.1}. \\[1mm]  
In case $(i)$ in Theorem \ref{t2.3}, $\tau$ is said to be in the \textit{limit circle case} at $b$ $($resp., at $a$$)$. $($Frequently, $\tau$ is then called \textit{quasi-regular} at $b$ $($resp., $a$$)$.$)$
\\[1mm] 
In case $(ii)$ in Theorem \ref{t2.3}, $\tau$ is said to be in the \textit{limit point case} at $b$ $($resp., at $a$$)$. \\[1mm]
If $\tau$ is in the limit circle case at $a$ and $b$ then $\tau$ is also called \textit{quasi-regular} on $(a,b)$. 
\end{definition}

\begin{theorem} \lb{t2.5}
Assume Hypothesis~\ref{h2.1}, then the following items $(i)$ and $(ii)$ hold: \\[1mm] 
$(i)$ If $\tau$ is in the limit point case at $a$ $($resp., $b$$)$, then 
\begin{equation} 
W(f,g)(a)=0 \, \text{$($resp., $W(f,g)(b)=0$$)$ for all $f,g\in\dom(T_{max})$.} 
\end{equation} 
$(ii)$ Let $T_{min}=\ol{\dot T_{min} }$. Then
\begin{align}
\begin{split}
n_\pm(T_{min}) &= \dim(\ker(T_{max} \mp i I))    \\
& = \begin{cases}
2 & \text{if $\tau$ is in the limit circle case at $a$ and $b$,}\\
1 & \text{if $\tau$ is in the limit circle case at $a$} \\
& \text{and in the limit point case at $b$, or vice versa,}\\
0 & \text{if $\tau$ is in the limit point case at $a$ and $b$}.
\end{cases}
\end{split} 
\end{align}
In particular, $T_{min} = T_{max}$ is self-adjoint $($i.e., $\dot T_{min,}$ is essentially self-adjoint\,$)$ if and only if $\tau$ is in the limit point case at $a$ and $b$. 
\end{theorem}

One now recalls the following fundamental result characterizing all self-adjoint extensions of $T_{min}$:

\begin{theorem} \lb{t2.6}
Assume Hypothesis \ref{h2.1} and that $\tau$ is in the limit circle case at $a$ and $b$ $($i.e., $\tau$ is quasi-regular on $(a,b)$$)$. In addition, assume that 
$v_j \in \dom(T_{max})$, $j=1,2$, satisfy 
\begin{equation}
W(\ol{v_1}, v_2)(a) = W(\ol{v_1}, v_2)(b) = 1, \quad W(\ol{v_j}, v_j)(a) = W(\ol{v_j}, v_j)(b) = 0, \; j= 1,2.  
\end{equation}
$($E.g., real-valued solutions $v_j$, $j=1,2$, of $(\tau - \lambda) u = 0$ with $\lambda \in \bbR$, such that 
$W(v_1,v_2) = 1$.$)$ For $g\in\dom(T_{max})$ we introduce the generalized boundary values 
\begin{align}
\begin{split} 
\wti g_1(a) &= - W(v_2, g)(a), \quad \wti g_1(b) = - W(v_2, g)(b),    \\
\wti g_2(a) &= W(v_1, g)(a), \quad \;\,\,\, \wti g_2(b) = W(v_1, g)(b).   \lb{2.10}
\end{split} 
\end{align}
Then the following items $(i)$--$(iii)$ hold: \\[1mm]
$(i)$ All self-adjoint extensions $T_{\gamma,\delta}$ of $T_{min}$ with separated boundary conditions are of the form
\begin{align}
& T_{\gamma,\delta} f = \tau f, \quad \gamma,\delta \in[0,\pi),   \no \\
& f \in \dom(T_{\gamma,\delta})=\big\{g\in\dom(T_{max}) \, \big| \, \wti g_1(a)\cos(\gamma)+ \wti g_2(a)\sin(\gamma)=0;   \lb{2.11} \\ 
& \hspace*{5.5cm} \, \wti g_1(b)\cos(\delta)+ \wti g_2(b)\sin(\delta) = 0 \big\}.    \no 
\end{align}
$(ii)$ All self-adjoint extensions $T_{\eta,R}$ of $T_{min}$ with coupled boundary conditions are of the type
\begin{align}
\begin{split} 
& T_{\eta,R} f = \tau f,    \\
& f \in \dom(T_{\eta,R})=\bigg\{g\in\dom(T_{max}) \, \bigg| \begin{pmatrix} \wti g_1(b)\\ \wti g_2(b)\end{pmatrix} 
= e^{i\eta}R \begin{pmatrix}
\wti g_1(a)\\ \wti g_2(a)\end{pmatrix} \bigg\}, \lb{2.11A}
\end{split}
\end{align}
where $\eta\in[0,\pi)$, and $R$ is a real $2\times2$ matrix with $\det(R)=1$ 
$($i.e., $R \in SL(2,\bbR)$$)$.  \\[1mm] 
$(iii)$ Every self-adjoint extension of $T_{min}$ is either of type $(i)$ $($i.e., separated\,$)$ or of type 
$(ii)$ $($i.e., coupled\,$)$.
\end{theorem}

\begin{remark} \lb{r2.7}
$(i)$ If $\tau$ is in the limit point case at one endpoint, say, at the endpoint $b$, one omits the corresponding boundary condition involving $\delta \in [0, \pi)$ at $b$ in \eqref{2.11} to obtain all self-adjoint extensions $T_{\gamma}$ of 
$T_{min}$, indexed by $\gamma \in [0, \pi)$. In the case where $\tau$ is in the limit point case at both endpoints, all boundary values and boundary conditions become superfluous as in this case $T_{min} = T_{max}$ is self-adjoint. \\[1mm] 
$(ii)$ Assume the special case where $\tau$ is regular on the finite interval $[a,b]$, that is, suppose that Hypothesis \ref{h2.1} is replaced by the more stringent set of assumptions: \\[1mm] 
{\bf Hypothesis} ($\tau$ regular on $[a,b]$.) \\[1mm] 
Let $(a,b) \subset \bbR$ be a finite interval and suppose that $p,q,r$ are $($Lebesgue\,$)$ measurable functions on $(a,b)$  
such that the following items $(i)$--$(iii)$ hold: \\[1mm] 
$(i')$ $r > 0$ a.e.~on $(a,b)$, $r \in L^1((a,b);dx)$. \\[1mm]
$(ii')$ $p > 0$ a.e.~on $(a,b)$, $1/p \in L^1((a,b);dx)$. \\[1mm]
$(iii')$ $q$ is real-valued a.e.~on $(a,b)$, $q \in L^1((a,b);dx)$.  \\[1mm] 
\indent 
In this case one chooses $v_j \in \dom(T_{max})$, $j=1,2$, 
such that 
\begin{align}
v_1(x) = \begin{cases} \theta_0(\lambda,x,a), & \text{for $x$ near a}, \\
\theta_0(\lambda,x,b), & \text{for $x$ near b},  \end{cases}   \quad 
v_2(x) = \begin{cases} \phi_0(\lambda,x,a), & \text{for $x$ near a}, \\
\phi_0(\lambda,x,b), & \text{for $x$ near b},  \end{cases}   \lb{2.12}
\end{align} 
where $\phi_0(\lambda,\, \cdot \,,d)$, $\theta_0(\lambda,\, \cdot \,,d)$, $d \in \{a,b\}$, are real-valued solutions of $(\tau - \lambda) u = 0$, $\lambda \in \bbR$, satisfying the boundary conditions 
\begin{align}
\begin{split} 
& \phi_0(\lambda,a,a) = \theta_0^{[1]}(\lambda,a,a) = 0, \quad \theta_0(\lambda,a,a) = \phi_0^{[1]}(\lambda,a,a) = 1, \\ 
& \phi_0(\lambda,b,b) = \theta_0^{[1]}(\lambda,b,b) = 0, \quad \; \theta_0(\lambda,b,b) = \phi_0^{[1]}(\lambda,b,b) = 1. 
\lb{2.13}
\end{split} 
\end{align} 
Then one verifies that
\begin{align}
\wti g_1 (a) = g(a), \quad \wti g_1 (b) = g(b), \quad \wti g_2 (a) = g^{[1]}(a), \quad \wti g_2 (b) = g^{[1]}(b),   \lb{2.14} 
\end{align}
and hence Theorem \ref{t2.6} in the special regular case recovers the well-known situation of separated self-adjoint boundary conditions for three-coefficient regular Sturm--Liouville operators in $\Lr$. 
\\[1mm]
$(iii)$ In connection with \eqref{2.10}, an explicit calculation demonstrates that for $g, h \in \dom(T_{max})$,
\begin{equation}
W(g,h)(d) = \wti g_1(d) \wti h_2(d) - \wti g_2(d) \wti h_1(d), \quad d \in \{a,b\},   \lb{2.15}
\end{equation} 
where either side in \eqref{2.15} is finite for $d \in \{a, b\}$.  
Of course, for \eqref{2.15} to hold at $d \in \{a,b\}$, it suffices that $g$ and $h$ lie locally in $\dom(T_{max})$ near $x=d$. \\[1mm]
$(iv)$ Clearly, $\wti g_1, \wti g_2$ depend on the choice of $v_j$, $j=1,2$, and a more precise notation would indicate this as $\wti g_{1,v_2}, \wti g_{2,v_1}$, etc. \hfill $\diamond$
\end{remark} 

In the special case where $T_{min}$ is bounded from below, one can further analyze the generalized boundary values  \eqref{2.10} in the singular context by invoking principal and nonprincipal solutions of $\tau u = \lambda u$ for appropriate $\lambda \in \bbR$. This leads to natural analogs of \eqref{2.14} also in the singular case, and we will turn to this topic next. 

We start by reviewing some oscillation theory with particular emphasis on principal and nonprincipal solutions, a notion originally due to Leighton and Morse \cite{LM36}, Rellich \cite{Re43}, \cite{Re51}, and Hartman and Wintner \cite[Appendix]{HW55} (see also \cite{CGN16}, \cite[Sects.~13.6, 13.9, 13.0]{DS88}, 
\cite[Ch.~XI]{Ha02}, \cite{NZ92}, \cite[Chs.~4, 6--8]{Ze05}). 

\begin{definition} \lb{d2.8}
Assume Hypothesis \ref{h2.1}. \\[1mm] 
$(i)$ Fix $c\in (a,b)$ and $\lambda\in\bbR$. Then $\tau - \lam$ is
called {\it nonoscillatory} at $a$ $($resp., $b$$)$, 
if every real-valued solution $u(\lambda,\dott)$ of 
$\tau u = \lambda u$ has finitely many
zeros in $(a,c)$ $($resp., $(c,b)$$)$. Otherwise, $\tau - \lam$ is called {\it oscillatory}
at $a$ $($resp., $b$$)$. \\[1mm] 
$(ii)$ Let $\lambda_0 \in \bbR$. Then $T_{min}$ is called bounded from below by $\lambda_0$, 
and one writes $T_{min} \geq \lambda_0 I$, if 
\begin{equation} 
(u, [T_{min} - \lambda_0 I]u)_{L^2((a,b);rdx)}\geq 0, \quad u \in \dom(T_{min}).
\end{equation}
\end{definition}

The following is a key result. 

\begin{theorem} \lb{t2.9} 
Assume Hypothesis \ref{h2.1}. Then the following items $(i)$--$(iii)$ are
equivalent: \\[1mm] 
$(i)$ $T_{min}$ $($and hence any symmetric extension of $T_{min})$
is bounded from below. \\[1mm] 
$(ii)$ There exists a $\nu_0\in\bbR$ such that for all $\lambda < \nu_0$, $\tau - \lam$ is
nonoscillatory at $a$ and $b$. \\[1mm]
$(iii)$ For fixed $c, d \in (a,b)$, $c \leq d$, there exists a $\nu_0\in\bbR$ such that for all
$\lambda<\nu_0$, $\tau u = \lambda u$ has $($real-valued\,$)$ nonvanishing solutions
$u_a(\lambda,\dott) \neq 0$,
$\hatt u_a(\lambda,\dott) \neq 0$ in the neighborhood $(a,c]$ of $a$, and $($real-valued\,$)$ nonvanishing solutions
$u_b(\lambda,\dott) \neq 0$, $\hatt u_b(\lambda,\dott) \neq 0$ in the neighborhood $[d,b)$ of
$b$, such that 
\begin{align}
&W(\hatt u_a (\lambda,\dott),u_a (\lambda,\dott)) = 1,
\quad u_a (\lambda,x)=\oh(\hatt u_a (\lambda,x))
\text{ as $x\downarrow a$,} \lb{2.17} \\
&W(\hatt u_b (\lambda,\dott),u_b (\lambda,\dott))\, = 1,
\quad u_b (\lambda,x)\,=\oh(\hatt u_b (\lambda,x))
\text{ as $x\uparrow b$,} \lb{2.18} \\
&\int_a^c dx \, p(x)^{-1}u_a(\lambda,x)^{-2}=\int_d^b dx \, 
p(x)^{-1}u_b(\lambda,x)^{-2}=\infty,  \lb{2.19} \\
&\int_a^c dx \, p(x)^{-1}{\hatt u_a(\lambda,x)}^{-2}<\infty, \quad 
\int_d^b dx \, p(x)^{-1}{\hatt u_b(\lambda,x)}^{-2}<\infty. \lb{2.20}
\end{align}
\end{theorem}

\begin{definition} \lb{d2.10}
Assume Hypothesis \ref{h2.1}, suppose that $T_{min}$ is bounded from below, and let 
$\lambda\in\bbR$. Then $u_a(\lambda,\dott)$ $($resp., $u_b(\lambda,\dott)$$)$ in Theorem
\ref{t2.9}\,$(iii)$ is called a {\it principal} $($or {\it minimal}\,$)$
solution of $\tau u=\lambda u$ at $a$ $($resp., $b$$)$. A real-valued solution 
$\wti{\wti u}_a(\lambda,\dott)$ $($resp., $\wti{\wti u}_b(\lambda,\dott)$$)$ of $\tau
u=\lambda u$ linearly independent of $u_a(\lambda,\dott)$ $($resp.,
$u_b(\lambda,\dott)$$)$ is called {\it nonprincipal} at $a$ $($resp., $b$$)$.
\end{definition}

Principal and nonprincipal solutions are well-defined due to Lemma \ref{l2.11} below. 

\begin{lemma} \lb{l2.11} Assume Hypothesis \ref{h2.1} and suppose that $T_{min}$ is bounded 
from below. Then $u_a(\lambda,\dott)$ and $u_b(\lambda,\dott)$ in Theorem
\ref{t2.9}\,$(iii)$ are unique up to $($nonvanish- ing\,$)$ real constant multiples. Moreover,
$u_a(\lambda,\dott)$ and $u_b(\lambda,\dott)$ are minimal solutions of
$\tau u=\lambda u$ in the sense that 
\begin{align}
u(\lambda,x)^{-1} u_a(\lambda,x)&=\oh(1) \text{ as $x\downarrow a$,} 
\lb{2.21} \\ 
u(\lambda,x)^{-1} u_b(\lambda,x)&=\oh(1) \text{ as $x\uparrow b$,} \lb{2.22}
\end{align}
for any other solution $u(\lambda,\dott)$ of $\tau u=\lambda u$
$($which is nonvanishing near $a$, resp., $b$$)$ with
$W(u_a(\lambda,\dott),u(\lambda,\dott))\neq 0$, respectively, 
$W(u_b(\lambda,\dott),u(\lambda,\dott))\neq 0$. 
\end{lemma}

Given these oscillation theoretic preparations, one can now revisit and complement Theorem \ref{t2.6} as follows: 

\begin{theorem} \lb{t2.12}
Assume Hypothesis \ref{h2.1} and that $\tau$ is in the limit circle case at $a$ and $b$ $($i.e., $\tau$ is quasi-regular 
on $(a,b)$$)$. In addition, assume that $T_{min} \geq \lambda_0 I$ for some $\lambda_0 \in \bbR$, and denote by 
$u_a(\lambda_0, \dott)$ and $\hatt u_a(\lambda_0, \dott)$ $($resp., $u_b(\lambda_0, \dott)$ and 
$\hatt u_b(\lambda_0, \dott)$$)$ principal and nonprincipal solutions of $\tau u = \lambda_0 u$ at $a$ 
$($resp., $b$$)$, satisfying
\begin{equation}
W(\hatt u_a(\lambda_0,\dott), u_a(\lambda_0,\dott)) = W(\hatt u_b(\lambda_0,\dott), u_b(\lambda_0,\dott)) = 1.  
\lb{2.23} 
\end{equation}
Introducing $v_j \in \dom(T_{max})$, $j=1,2$, via 
\begin{align}
v_1(x) = \begin{cases} \hatt u_a(\lambda_0,x), & \text{for $x$ near a}, \\
\hatt u_b(\lambda_0,x), & \text{for $x$ near b},  \end{cases}   \quad 
v_2(x) = \begin{cases} u_a(\lambda_0,x), & \text{for $x$ near a}, \\
u_b(\lambda_0,x), & \text{for $x$ near b},  \end{cases}   \lb{2.24}
\end{align} 
one obtains for all $g \in \dom(T_{max})$, 
\begin{align}
\begin{split} 
\wti g(a) &= - W(v_2, g)(a) = \wti g_1(a) =  - W(u_a(\lambda_0,\dott), g)(a)    \\
&= \lim_{x \downarrow a} \f{g(x)}{\hatt u_a(\lambda_0,x)},    \\
\wti g(b) &= - W(v_2, g)(b) = \wti g_1(b) =  - W(u_b(\lambda_0,\dott), g)(b)   \\
&= \lim_{x \uparrow b} \f{g(x)}{\hatt u_b(\lambda_0,x)},    
\lb{2.25} 
\end{split} \\
\begin{split} 
{\wti g}^{\, \prime}(a) &= W(v_1, g)(a) = \wti g_2(a) = W(\hatt u_a(\lambda_0,\dott), g)(a)   \\
&= \lim_{x \downarrow a} \f{g(x) - \wti g(a) \hatt u_a(\lambda_0,x)}{u_a(\lambda_0,x)},    \\ 
{\wti g}^{\, \prime}(b) &= W(v_1, g)(b) = \wti g_2(b) = W(\hatt u_b(\lambda_0,\dott), g)(b)   \\ 
&= \lim_{x \uparrow b} \f{g(x) - \wti g(b) \hatt u_b(\lambda_0,x)}{u_b(\lambda_0,x)}.    \lb{2.26}
\end{split} 
\end{align}
In particular, the limits on the right-hand sides in \eqref{2.25}, \eqref{2.26} exist and all self-adjoint extensions $T_{\gamma,\delta}$ of $T_{min}$ with separated boundary conditions are of the form
\begin{align}
& T_{\gamma,\delta} f = \tau f, \quad \gamma,\delta \in[0,\pi),   \no \\
& f \in \dom(T_{\gamma,\delta})=\big\{g\in\dom(T_{max}) \, \big| \, \sin(\gamma) {\wti g}^{\, \prime}(a) + \cos(\gamma) \wti g(a) = 0;   \lb{2.27} \\ 
& \hspace*{5.5cm} \,  \sin(\delta) {\wti g}^{\, \prime}(b) + \cos(\delta) \wti g(b) = 0 \big\},    \no 
\end{align}
whereas self-adjoint extensions $T_{\eta,R}$ of $T_{min}$ with coupled boundary conditions are of the form
\begin{align}
\begin{split} 
& T_{\eta,R} f = \tau f,  \quad  \eta\in[0,\pi),\;  R \in SL(2,\bbR),\\
& f \in \dom(T_{\eta,R})=\bigg\{g\in\dom(T_{max}) \, \bigg| \begin{pmatrix} \wti g(b)\\ \wti g^{\, \prime}(b)\end{pmatrix} 
= e^{i\eta}R \begin{pmatrix}
\wti g(a)\\ \wti g^{\, \prime}(a)\end{pmatrix} \bigg\}. \lb{2.27A}
\end{split}
\end{align}
Moreover, $\sigma(T_{\gamma,\delta})$ is simple. 
\end{theorem}

The Friedrichs extension $T_F$ of $T_{min}$ now permits a particularly simple characterization in terms of the generalized boundary values $\wti g(a), \wti g(b)$ as derived by Kalf  \cite{Ka78} and subsequently by Niessen and Zettl \cite{NZ92} (see also \cite{Re51}, \cite{Ro85} and the extensive literature cited in \cite{GLN20}, 
\cite[Ch.~13]{GNZ23}):

\begin{theorem} \lb{t2.13}
Assume Hypothesis \ref{h2.1} and that $\tau$ is in the limit circle case at $a$ and $b$ $($i.e., $\tau$ 
is quasi-regular on $(a,b)$$)$. In addition, assume that $T_{min} \geq \lambda_0 I$ for some $\lambda_0 \in \bbR$. Then the Friedrichs extension $T_F$ of $T_{min}$ is characterized by
\begin{align}
T_F f = \tau f, \quad f \in \dom(T_F)= \big\{g\in\dom(T_{max})  \, \big| \, \wti g(a) = \wti g(b) = 0\big\}.    \lb{2.28}
\end{align}
In particular, $T_F = T_{0,0}$.
\end{theorem}

\begin{remark} \lb{r2.14}
$(i)$ As in \eqref{2.15}, one readily verifies for $g, h \in \dom(T_{max})$,
\begin{equation}
W(g,h)(d) = \wti g(d) {\wti h}^{\, \prime}(d) - {\wti g}^{\, \prime}(d) \wti h(d), \quad d \in \{a,b\},    \lb{2.29} 
\end{equation} 
where either side in \eqref{2.29} has a finite limit as $d \in \{a, b\}$. Moreover, if $\tau$ is regular at an endpoint, then the 
generalized boundary values in \eqref{2.25}, \eqref{2.26} reduce to the canonical ones in \eqref{2.14} as long as principal and nonprincipal solutions chosen are appropriately normalized such that \eqref{2.12} and \eqref{2.24} coincide. In addition (and now independently of the chosen normalization of principal and nonprincipal solutions), in the regular context \eqref{2.29} becomes for $g, h \in \dom(T_{max})$,
\begin{equation}
W(g,h)(d) = \wti g(d) {\wti h}^{\, \prime}(d) - {\wti g}^{\, \prime}(d) \wti h(d) = g(d) h^{[1]}(d) - g^{[1]}(d) h(d), \quad d \in \{a,b\}.    \lb{2.32} 
\end{equation} 
\\[1mm] 
$(ii)$ While the generalized boundary values at the endpoint $d \in \{a,b\}$ clearly depend on the choice of nonprincipal solution $\hatt u_{d}(\lambda_0, \dott)$ of $\tau u = \lambda_0 u$ at $d$, the Friedrichs boundary conditions $\wti g(a) = \wti g(b) = 0$ are independent of the choice of this nonprincipal solution. \\[1mm]
$(iii)$ As always in this context, if $\tau$ is in the limit point case at one (or both) interval endpoints, the corresponding boundary conditions at that endpoint are dropped and only a separated boundary condition at the other end point (if the latter is a limit circle endpoint for $\tau$), has to be imposed in Theorems \ref{t2.12} and \ref{t2.13}. In the case where $\tau$ is in the limit point case at both endpoints, all boundary values and boundary conditions become superfluous as $T_{min} = T_{max}$ is self-adjoint. \hfill $\diamond$
\end{remark}

All results surveyed in this section can be found in \cite{GLN20} and \cite[Ch.~13]{GNZ23} which contain very detailed lists of references to the basics of Weyl--Titchmarsh theory. Here we just mention a few additional and classical sources such as \cite[Sect.~129]{AG81}, \cite[Ch.~6]{BHS20}, \cite{BFL20}, 
\cite[Chs.~8, 9]{CL85}, \cite[Sects.~13.6, 13.9, 13.0]{DS88}, \cite{Fr20}, 
\cite[Ch.~III]{JR76}, \cite[Ch.~V]{Na68}, \cite{NZ92}, \cite[Ch.~6]{Pe88}, \cite[Ch.~9]{Te14}, \cite[Sect.~8.3]{We80}, \cite[Ch.~13]{We03}, \cite[Chs.~4, 6--8]{Ze05}.

\section{The (Singular) Weyl--Titchmarsh--Kodaira $m$-Function} \lb{s3}

In this section we recall the construction of the (singular) Weyl--Titchmarsh--Kodaira $m$ and $M$-functions for separated and coupled boundary conditions, respectively, primarily following the detailed treatment in \cite[Ch.~13]{GNZ23} (see also \cite{GLN20}) and demonstrate that the generalized boundary values in \eqref{2.25}, \eqref{2.26} naturally fit into this framework. For simplicity we single out the left endpoint $a$ in the following as the endpoint $b$ can be treated in precisely the same manner.  

\begin{hypothesis} \lb{h3.1}
In addition to Hypothesis \ref{h2.1}, let $T_{\gamma, \delta}$, $\gamma, \delta \in [0,\pi)$, be any self-adjoint extension of $T_{min}$ in $L^2((a,b); r dx)$ with separated boundary conditions as in \eqref{2.11} 
$($if $\tau$ is in the limit circle case at $a$ and/or $b$$)$, and suppose that for some $($and hence for all\,$)$ $x_0 \in (a,b)$, the self-adjoint operator 
$T_{\gamma, 0, a,x_0}$ in $L^2((a,x_0); r dx)$, associated with $\tau|_{(a,x_0)}$ and a Dirichlet boundary condition at $x_0$ $($i.e., $g(x_0)=0$, $g \in \dom(T_{max,a,x_0})$, with $T_{max,a,x_0}$ the maximal operator associated with $\tau|_{(a,x_0)}$ in $L^2((a,x_0); rdx)$$)$, has purely discrete spectrum.
\end{hypothesis}

It is known (see, e.g., \cite{EGNT13}, \cite[Sect.~13.2]{GNZ23}, \cite{KST12}) that Hypothesis \ref{h3.1} is equivalent to the existence of an entire solution $\varphi_{\gamma}(z, \dott)$ of $\tau u = z u$, $z \in \bbC$, that is real-valued for $z \in \bbR$, and lies in 
$\dom(T_{\gamma, \delta})$ near the point $a$. In particular, $\varphi_{\gamma}(z, \dott)$ satisfies the boundary condition indexed by $\gamma$ at the left endpoint $a$ if $\tau$ is in the limit circle case at $a$, and 
$\varphi_{\gamma}(z, \dott) \in L^2((a,x_0); r dx)$ if $\tau$ is in the limit point case at $a$. In addition, it is known that a second, linearly independent entire solution $\vartheta_{\gamma}(z, \dott)$ of $\tau u = z u$ exists, with $\vartheta_{\gamma}(z, \dott)$ real-valued for $z \in \bbR$, satisfying 
\begin{equation}
W(\vartheta_{\gamma}(z, \dott),\varphi_{\gamma}(z, \dott))(x) =1, \quad \gamma \in [0,\pi), \; z \in \bbC, \; x \in (a,b).   \lb{3.1} 
\end{equation}

We note that $\varphi_{\gamma}(z, \dott)$ is unique up to a nonvanishing entire factor (real on the real line) with respect to $z \in \bbC$. Hence, if $\tau$ is in the limit circle case at $a$ we will always normalize 
$\varphi_{\gamma}(z, \dott)$ and $\vartheta_{\gamma} (z,\dott)$ such that
\begin{align}
\begin{split}
\wti \varphi_{\gamma} (z,a) &= - \sin(\gamma), \quad \wti \varphi_{\gamma}^{\, \prime} (z,a) = \cos(\gamma),    \lb{3.2} \\
\wti \vartheta_{\gamma} (z,a) &= \cos(\gamma), \quad \wti \vartheta_{\gamma}^{\, \prime} (z,a) = \sin(\gamma); 
\quad  \gamma \in [0,\pi), \; z \in \bbC.    
\end{split} 
\end{align}

In addition to the entire fundamental system $\varphi_{\gamma}(z, \dott), \vartheta_{\gamma}(z, \dott)$ 
of $\tau u = z u$, we also mention the standard entire fundamental system 
$\theta_0(z, \dott,x_0), \phi_0(z, \dott,x_0)$ of $\tau u = z u$ normalized at $x_0 \in (a,b)$ in the 
usual manner, 
\begin{align}
& \theta_0(z,x_0,x_0) =1, \quad \theta_0^{[1]}(z,x_0,x_0) =0,  \quad  
\phi_0(z,x_0,x_0) =0, \quad \phi_0^{[1]}(z,x_0,x_0) =1,      \no \\
& \hspace*{10cm} z \in \bbC,   \lb{3.3} 
\end{align}
and the Weyl--Titchmarsh solutions $\psi_{\gamma,0,-} (z,\dott,x_0)$ and $\psi_{0,\delta,+} (z,\dott,x_0)$ 
of $\tau u = z u$  given by 
\begin{align}
\begin{split}
\psi_{\gamma,0,-} (z,x,x_0) = \theta_0(z,x,x_0) + m_{\gamma,0,-} (z,x_0) \phi_0(z,x,x_0),&     \\
\gamma \in [0,\pi), \; z \in \rho(T_{\gamma, 0, a,x_0}), \; x \in (a,b),& 
\end{split} \\
\begin{split}
\psi_{0,\delta,+} (z,x,x_0) = \theta_0(z,x,x_0) + m_{0,\delta,+} (z,x_0)  \phi_0(z,x,x_0),&    \\
\gamma \in [0,\pi), \; z \in \rho(T_{0,\delta,x_0,b}), \; x \in (a,b),& 
\end{split} 
\end{align}
where $T_{0,\delta,x_0,b}$ in $L^2((x_0,b); r dx)$ is the self-adjoint operator associated with 
$\tau|_{(x_0,b)}$ and a Dirichlet boundary condition at $x_0$ (i.e., $g(x_0)=0$, $g \in \dom(T_{max,x_0,b})$, 
with $T_{max,x_0,b}$ the maximal operator associated with $\tau|_{(x_0,b)}$ in $L^2((x_0,b); rdx)$). In particular, $\psi_{\gamma,0,-} (z, \dott,x_0)$ satisfies the boundary condition indexed by $\gamma$ at the left endpoint $a$; in addition, $\psi_{\gamma,0,-} (z, \dott,x_0) \in L^2((a,x_0); r dx)$ if $\tau$ is in the limit circle case at $a$. Similarly, $\psi_{\gamma,0,-} (z, \dott,x_0)$ satisfies $\psi_{\gamma,0,-} (z, \dott,x_0) \in L^2((a,x_0); r dx)$ if $\tau$ is in the limit point case at $a$. Analogously, $\psi_{0,\delta,+} (z,\dott,x_0)$ satisfies the boundary condition indexed by $\delta$ at the right endpoint $b$ and the condition $\psi_{0,\delta,+} (z,\dott,x_0) \in L^2((x_0,b); r dx)$ if $\tau$ is in the limit circle case at $b$. Moreover, $\psi_{0,\delta,+} (z,\dott,x_0) \in L^2((x_0,b); r dx)$ if $\tau$ is in the limit point case at $b$.
Thus, $m_{\gamma,0,-} (\dott,x_0)$ is analytic on $\bbC \backslash \bbR$, meromorphic on $\bbC$, with simple poles on the real axis precisely at the simple eigenvalues of $T_{\gamma, 0, a,x_0}$, and 
$\varphi_{\gamma} (z,\dott)$ is a $z$-dependent multiple of $\psi_{\gamma,0,-} (z,\dott,x_0)$, 
\begin{equation}
\varphi_{\gamma} (z,\dott) = C_{\gamma,-}(z,x_0) \psi_{\gamma,0,-} (z,\dott,x_0) 
\end{equation}
(with $C_{\gamma,-}(z,x_0)$ entire with respect to $z$), real-valued for $z \in \bbR$, and having simple zeros at the simple poles of $m_{\gamma,0,-} (\dott,x_0)$. In addition, one readily confirms that
\begin{equation}
m_{\gamma,0,-} (z,x_0) = \psi_{\gamma,0,-}^{[1]} (z,x_0,x_0) = \varphi_{\gamma}^{[1]} (z,x_0)/\varphi_{\gamma} (z,x_0), \quad z \in \rho(T_{\gamma, 0, a,x_0}).     \lb{3.7}
\end{equation}

Next, we rewrite $\psi_{0,\delta,+}(z,\dott,x_0)$ in terms of the entire fundamental system 
$\varphi_{\gamma}(z, \dott), \vartheta_{\gamma}(z, \dott)$ and thereby introduce the singular 
Weyl--Titchmarsh--Kodaira solution, $\psi_{\gamma,\delta} (z,\dott)$,
and the  Weyl--Titchmarsh--Kodaira $m$-function, $m_{\gamma,\delta}(\dott)$, via 
\begin{align} 
\begin{split}
\psi_{\gamma,\delta} (z,x) &= C_{\gamma,\delta}(z,x_0) \psi_{0,\delta,+} (z,x,x_0)   \\
&= \vartheta_{\gamma}(z,x) + m_{\gamma,\delta} (z) \varphi_{\gamma}(z,x), \quad 
z \in \rho(T_{\gamma,\delta}),     \lb{3.8}
\end{split} 
\end{align}
where $C_{\gamma,\delta}(\dott,x_0)$ is analytic and nonvanishing on $\bbC \backslash \bbR$. In analogy to 
\eqref{3.7} one obtains
\begin{equation}
m_{0,\delta,+} (z,x_0) = \psi_{0,\delta,+}^{[1]} (z,x_0,x_0) 
= \psi_{\gamma,\delta}^{[1]} (z,x_0)/\psi_{\gamma,\delta} (z,x_0), \quad 
z \in \rho(T_{0,\delta,x_0,b}).
\end{equation}
In particular,
\begin{equation}
m_{0,\delta,+} (z,x_0) = \f{\vartheta_{\gamma}^{[1]} (z,x_0) 
+ m_{\gamma,\delta} (z) \varphi_{\gamma}^{[1]} (z,x_0)}{\vartheta_{\gamma} (z,x_0) 
+ m_{\gamma,\delta} (z) \varphi_{\gamma} (z,x_0)}, \quad 
z \in \bbC \rho(T_{0,\delta,x_0,b}), 
\end{equation}
and thus,
\begin{align}
 m_{\gamma,\delta} (z) &= \f{- \vartheta_{\gamma}^{[1]} (z,x_0)  
 + m_{0,\delta,+} (z,x_0) \vartheta_{\gamma} (z,x_0)}{\varphi_{\gamma}^{[1]} (z,x_0) 
 - m_{0,\delta,+} (z,x_0) \varphi_{\gamma} (z,x_0)}     \no \\
 &= \f{1}{\varphi_{\gamma} (z,x_0)} \f{- \vartheta_{\gamma}^{[1]} (z,x_0)  
 + m_{0,\delta,+} (z,x_0) \vartheta_{\gamma} (z,x_0)}{m_{\gamma,0,-} (z,x_0) - m_{0,\delta,+} (z,x_0)}   \no \\
 &= \f{1}{[\varphi_{\gamma} (z,x_0)]^2} \f{1}{m_{\gamma,0,-} (z,x_0) - m_{0,\delta,+} (z,x_0)} 
 - \f{\vartheta_{\gamma} (z,x_0)}{\varphi_{\gamma} (z,x_0)},    \lb{3.11} \\
 & \hspace*{4.95cm} \gamma, \delta  \in [0,\pi), \; z \in \rho(T_{\gamma,\delta}),      \no
\end{align}
employing \eqref{3.1}.

\begin{remark} \lb{r3.2}
In the case where $\tau_{\a,\b}$ is in the limit circle case at $a$, \eqref{3.2} uniquely determines 
$\varphi_{\gamma}(z,\dott)$ and $\vartheta_{\g}(z,\dott)$. If $\tau_{\a,\b}$ is in the limit point case at $a$, equation \eqref{3.1} alone cannot determine $\varphi_{\gamma}(z,\dott)$ and $\vartheta_{\g}(z,\dott)$ uniquely. In particular, 
$\varphi_{\gamma}$ is only determined up to a nonvanishing, entire, multiplicative factor and $\vartheta_{\g}(z,\dott)$ can contain an additive term of the type $c_{\g}(\dott) \varphi_{\g}(z,\dott)$, with $c_{\g}(\dott)$ entire. This will influence the actual form of $m_{\g,\d}(\dott)$ in \eqref{3.8}. However, since $c_{\g}(\dott)$ and $\varphi_{\g}(\dott,x)$, $x \in (a,b)$, are entire, the multiplicative ambiguity in $\varphi_{\gamma}(z,\dott)$, as well as the additive term 
$c_{\g}(\dott) \varphi_{\g}(z,\dott)$, cannot affect the measure $d \rho_{\g,\d}$ obtained via the (generalized) Stieltjes inversion formula from $m_{\g,\d}$. Thus, from a spectral theoretic point of view, the ambiguities in the choices of 
$\varphi_{\gamma}(z,\dott)$ and $\vartheta_{\g}(z,\dott)$ can safely be ignored and hence in a concrete example, the actual choice of $\varphi_{\gamma}(z,\dott)$ and $\vartheta_{\g}(z,\dott)$ is largely a matter of convenience. 
\hfill $\diamond$ 
\end{remark}

Since $- m_{\gamma,0,-} (\dott,x_0)$ and $m_{0,\delta,+} (\dott,x_0)$ are Nevanlinna--Herglotz functions and 
$\vartheta_{\gamma} (z,\dott)$ and $\varphi_{\gamma} (z,\dott)$ are entire with respect to $z$ and real-valued for 
$z \in \bbR$, it follows that $m_{\gamma,\delta} (\dott)$ is analytic on $\bbC \backslash \bbR$ and that 
\begin{equation}
m_{\gamma,\delta} (z) = \ol{m_{\gamma,\delta} (\ol{z})}, \quad z \in \bbC_+.
\end{equation}

Since the Green's function $G_{\gamma,\delta}(z,\dott,\dott)$ of $T_{\gamma,\delta}$ (i.e., the integral kernel of its resolvent) is of the form
\begin{equation}
G_{\gamma,\delta}(z,x,x') = \begin{cases} \varphi_{\gamma} (z,x) \psi_{\gamma,\delta} (z,x'), 
& a < x \leq x' < b, \\
\varphi_{\gamma} (z,x') \psi_{\gamma,\delta} (z,x), & a < x' \leq x < b,
\end{cases} \quad z \in \rho(T_{\gamma,\delta}), 
\end{equation}
one obtains
\begin{align}
\begin{split} 
G_{\gamma,\delta}(z,x_0,x_0) &= \varphi_{\gamma} (z,x_0) [\vartheta_{\gamma}(z,x_0) + m_{\gamma,\delta} (z) \varphi_{\gamma}(z,x_0)]    \\
&= [m_{\gamma,0,-} (z,x_0) - m_{0,\delta,+} (z,x_0)]^{-1}, \quad z \in \rho(T_{\gamma,\delta}),
\end{split} 
\end{align}
confirming the generally known fact that all diagonal Green's functions such as $G_{\gamma,\delta}(z,x_0,x_0)$ possess the Nevanlinna--Herglotz property. Equivalently, one obtains
\begin{equation}
m_{\gamma,\delta} (z) = [\varphi_{\gamma}(z,x_0)]^{-2} G_{\gamma,\delta}(z,x_0,x_0) 
- [\vartheta_{\gamma}(z,x_0)/\varphi_{\gamma}(z,x_0)], \quad z \in \rho(T_{\gamma,\delta}),
\end{equation} 
illustrating that generally, $m_{\gamma,\delta}$ is a generalized Nevanlinna--Herglotz function. One notes that 
$\varphi_{\gamma}(\dott,x_0)$ has a discrete set of real zeros (accumulating at $+ \infty$) precisely at the eigenvalues of $T_{\gamma,0,a,x_0}$, but since by definition $m_{\gamma,\delta} (\dott)$ is independent  of $x_0$, a small change in $x_0$ removes a possible ambiguity when applying an 
analog of the Stieltjes inversion formula to $m_{\gamma,\delta}$ to derive the spectral function 
$\rho_{\gamma,\delta}$ associated with $T_{\gamma, \delta}$. Even though $m_{\gamma,\delta}$, 
and hence, $\rho_{\gamma,\delta}$, is nonunique, the measure equivalence class generated by the spectral function $\rho_{\gamma,\delta}$ is unique and hence the spectrum (and it's subdivisions) are related to the singularity structure of $\Im(m_{\gamma,\delta} (\dott))$ on the real line. 

One also notices that if $\tau$ is in the limit circle case at $a$, the normalization chosen in \eqref{3.2}, combined with \eqref{3.8} readily implies
\begin{equation}
m_{\gamma,\delta} (z) = 
\cos(\gamma) \wti \psi_{\gamma,\delta}^{\, \prime} (z,a)  - \sin(\gamma) \wti \psi_{\gamma,\delta} (z,a), 
\quad z \in \rho(T_{\gamma, \delta}),
\end{equation}
a result familiar from the special case where $a$ is a regular endpoint (employing the fact \eqref{2.14}),  illustrating once more that the generalized boundary values \eqref{2.25}, \eqref{2.26}, in the context of a singular endpoint $a$, are natural extensions of the familiar boundary values in the case of regular endpoints.

Moreover, if $\tau$ is in the limit circle case at $b$, then $\psi_{\gamma,\delta}(z,\dott)$, like 
$\psi_{0,\delta,+}(z,\dott,x_0)$, satisfies the boundary condition indexed by $\delta \in [0,\pi)$ at $b$,
\begin{equation}
\sin(\delta) \wti\psi_{\gamma,\delta}' (z,b) + \cos(\delta) \wti\psi_{\gamma,\delta} (z,b) = 0, \quad z \in \bbC\backslash\bbR, 
\end{equation}
and hence 
\begin{align} \lb{3.18}
m_{\gamma,\delta}(z)= 
- \f{\cos(\delta) \wti\vartheta_{\gamma} (z,b) + \sin(\delta) \wti\vartheta_{\gamma}^{\, \prime} (z,b)}{\cos(\delta) \wti\varphi_{\gamma} (z,b) + \sin(\delta) \wti\varphi_{\gamma}^{\, \prime} (z,b)}, \quad z \in \bbC\backslash\bbR.   
\end{align}

In addition, given that $m_{\gamma,0,-} (z,x_0)$ is meromorphic on $\bbC$ by hypothesis, one infers that 
\begin{equation}
\text{$ m_{\gamma,\delta} (\dott)$ is meromorphic on $\bbC$ if and only if $m_{0,\delta,+} (\dott,x_0)$ is.} 
\end{equation}
In this case (cf.\ \eqref{3.11}) all poles of $m_{\gamma,\delta} (\dott)$ are simple and hence 
\begin{equation}
\text{$\lambda_0 \in \sigma_p(T_{\gamma,\delta})$ if and only if 
$\varphi_{\gamma} (\lambda_0,\dott) \in L^2((a,b);rdx)$} 
\end{equation}
and then
\begin{equation}
- {\rm Res}_{z=\lambda_0} (m_{\gamma,\delta}(\dott)) = \|\varphi_{\gamma} (\lambda_0,\dott)\|_{L^2((a,b);rdx)}^{-2},
\end{equation}
in accordance with the first order pole behavior of $G_{\gamma,\delta}(z,x_0,x_0)$ at $z=\lambda_0$, 
\begin{align}
G_{\gamma,\delta}(z,x,x') \underset{z \to \lambda_0}{=} 
- \f{{\rm Res}_{z=\lambda_0} (m_{\gamma,\delta}(\dott))}{z-\lambda_0} \varphi_{\gamma} (\lambda_0,x) 
\varphi_{\gamma} (\lambda_0,x') + \Oh(1). 
\end{align}

Fixing the boundary condition at $b$ indexed by $\delta \in [0, \pi)$ (if any), and varying the boundary condition at the left endpoint $a$ then yields the standard linear fractional transformation 
\begin{align}
\begin{split} 
m_{\gamma_{2}, \delta}(z) =\frac{-\sin(\gamma_{2}-\gamma_1) +
\cos(\gamma_{2}-\gamma_1) m_{\gamma_1,\delta}(z)}
{\cos(\gamma_{2}-\gamma_1) +\sin(\gamma_{2}-\gamma_1)
m_{\gamma_1, \delta}(z)},&      \lb{3.23} \\ 
\gamma_1, \gamma_2, \delta \in [0,\pi), \; z \in \bbC \backslash \bbR.&
\end{split}  
\end{align}

We also note a well-known observation (see, \cite{EGNT13}, \cite{KT11}) when $\tau$ is in the limit circle case at $a$: In this situation  Hypothesis \ref{h3.1} is always satisfied and one infers that 
\begin{equation}
\f{\Im(m_{\gamma,\delta} (z))}{\Im(z)} = \int_a^b r(x)dx \, |\psi_{\gamma,\delta} (z,x)|^2 > 0, \quad 
z \in \rho(T_{\gamma, \delta}),     \lb{3.24}
\end{equation}
in particular, in this special case $m_{\gamma,\delta} (\dott)$ is actually a Nevanlinna--Herglotz function.

In case $\tau$ is in the limit circle case at $a$ and in the limit point case at $b$, one simply drops the $\delta$-dependence of all quantities. Similarly, the $\gamma$-dependence of all quantities is dropped if $\tau$ is in the limit point case at $a$. 

Next, we briefly recall the case of coupled boundary conditions following \cite[Ch.~13]{GNZ23}. Recalling \eqref{3.1}, \eqref{3.2}, one obtains for the Green's function of $T_{\eta,R}$, 
\begin{align}
& G_{\eta,R}(z,x,x') = \f{- e^{i \eta}}{\wti F_{\eta,R}(z)}     \no \\
& \qquad \times \Big\{\big[ - R_{1,2} {\wti \varphi_0}^{\, \prime}(z,b) + R_{2,2} \wti \varphi_0(z,b)\big] \vartheta_0(z,x) \vartheta_0(z,x')     \no \\
& \hspace*{1.3cm} + \big[- R_{1,1} {\wti \vartheta_0}^{\, \prime}(z,b) + R_{2,1} \wti \vartheta_0(z,b)\big] \varphi_0(z,x) \varphi_0(z,x')   \no \\
& \hspace*{1.3cm} + \big[e^{-i \eta} + R_{1,2} {\wti \vartheta_0}^{\, \prime}(z,b) - R_{2,2} \wti \vartheta_0(z,b)\big]
\vartheta_0(z,x) \varphi_0(z,x')    \no \\
& \hspace*{1.3cm} + \big[- e^{-i \eta} + R_{1,1} {\wti \varphi_0}^{\, \prime}(z,b) - R_{2,1} \wti \varphi_0(z,b)\big] \varphi_0(z,x) \vartheta_0(z,x')\Big\}    \no \\
& \quad + \begin{cases} 0, & a < x \leq x' < b, \\
[\vartheta_0(z,x) \varphi_0(z,x') - \varphi_0(z,x) \vartheta_0(z,x')], & a < x' \leq x < b,
\end{cases}     \lb{3.25} \\
& \hspace*{7.5cm} z \in \rho(T_{\eta,R}),    \no
\end{align}
where 
\begin{align} 
& \wti F_{\eta,R}(z) = - e^{i \eta} \big[R_{1,1} {\wti \varphi_0}^{\, \prime}(z,b) + R_{2,2} \wti \vartheta_0(z,b) - R_{2,1} 
\wti \varphi_0(z,b) - R_{1,2} {\wti \vartheta_0}^{\, \prime}(z,b)    \no \\
& \hspace*{2.45cm} - 2 \cos(\eta)\big].      \lb{3.26}
\end{align}
The corresponding Green's function induced $M$-function is then given by
\begin{align}
& M_{\eta,R}(z,a) = \begin{pmatrix} \wti G_{\eta,R}(z,a,a)
& 2^{-1} \big(\big[\wti \partial_1 + \wti \partial_2 \big] G_{\eta,R}\big)^{\wti{}}(z,a,a) \\[1mm]
2^{-1} \big(\big[\wti \partial_1 + \wti \partial_2 \big] G_{\eta,R}\big)^{\wti{}}(z,a,a)
& \big(\wti \partial_1 \wti \partial_2 G_{\eta,R}\big)(z,a,a)
\end{pmatrix}     \lb{3.27} \\[1mm]
& \quad = \f{- e^{i \eta}}{\wti F_{\eta,R}(z)}
\left(\begin{smallmatrix}
R_{2,2} \wti \varphi_0(z,b) - R_{1,2} {\wti \varphi_0}^{\, \prime}(z,b)
& 2^{-1} \big[R_{1,1} {\wti \varphi_0}^{\, \prime}(z,b) - R_{2,1} \wti \varphi_0(z,b)  \\[1mm]
{} & \qquad \, - R_{2,2} \wti \vartheta_0(z,b) + R_{1,2}  {\wti \vartheta_0}^{\, \prime}(z,b)\big]  \\[2mm]
2^{-1} \big[R_{1,1} {\wti \varphi_0}^{\, \prime}(z,b) - R_{2,1} \wti \varphi_0(z,b)
& R_{2,1} \wti \vartheta_0(z,b) - R_{1,1} {\wti \vartheta_0}^{\, \prime}(z,b) \\[1mm] \\[1mm]
\qquad \, - R_{2,2} \wti \vartheta_0(z,b) + R_{1,2} {\wti \vartheta_0}^{\, \prime}(z,b)\big]
& {}
\end{smallmatrix}\right),     \no \\[1mm]
& \hspace*{8.4cm} z \in \bbC \backslash \sigma(T_{\eta,R}),     \lb{3.28}
\end{align}
with
\begin{equation}
\big(\wti \partial_1 \wti \partial_2 G_{\eta,R}\big)(z,a,a) = \big(\wti \partial_2 \wti \partial_1 G_{\eta,R}\big)(z,a,a), 
\quad z \in \bbC \backslash \sigma(T_{\eta,R}).     \lb{3.29}
\end{equation} 
Here we employed the notation, 
\begin{align}
& \wti G(z,a,a) = \lim_{x_1 \downarrow a} \lim_{x_2 \downarrow a} \f{G(z,x_1,x_2)}{\hatt u_a(\lambda_0,x_1) \hatt u_a(\lambda_0,x_2) },    \no \\
&\big(\big[\wti \partial_1 + \wti \partial_2 \big] G\big)^{\wti{}} (z,a,a) = \lim_{x_2 \downarrow a} \lim_{x_1 \downarrow a} 
\bigg[\f{G(z,x_1,x_2) - \wti G(z,a,x_2) \hatt u_a(\lambda_0,x_1)}{u_a(\lambda_0,x_1) \hatt u_a(\lambda_0,x_2)}\bigg]    \no \\
& \hspace*{3.55cm} + \lim_{x_1 \downarrow a} \lim_{x_2 \downarrow a} 
\bigg[\f{G(z,x_1,x_2) - \wti G(z,x_1,a) \hatt u_a(\lambda_0,x_2)}{\hatt u_a(\lambda_0,x_1) u_a(\lambda_0,x_2)}\bigg]    \no \\
& \wti G(z,a,x_2) = \lim_{x_1 \downarrow a} \f{G(z,x_1,x_2)}{\hatt u_a(\lambda_0,x_1)}, \quad 
\wti G(z,x_1,a) = \lim_{x_2 \downarrow a} \f{G(z,x_1,x_2)}{\hatt u_a(\lambda_0,x_2)},     \lb{3.30} \\
& \big(\wti \partial_1 \wti \partial_2 G\big)(z,a,a)     \no \\
& \quad = \lim_{x_1 \downarrow a} \Bigg\{\f{\lim_{x_2 \downarrow a} 
\Big(\f{G(z,x_1,x_2) - \wti G(z,x_1,a) \hatt u_a(\lambda_0,x_2)}{u_a(\lambda_0,x_2)}\Big)}{u_a(\lambda_0,x_1)}     \no \\
& \hspace*{1.7cm} - \f{\lim_{x'_1 \downarrow a} \lim_{x_2 \downarrow a} \Big(\f{G(z,x'_1,x_2) 
- \wti G(z,x'_1,a) \hatt u_a(\lambda_0,x_2)}{\hatt u_a(\lambda_0,x'_1) u_a(\lambda_0,x_2)}\Big) 
\hatt u_a(\lambda_0,x_1)}{u_a(\lambda_0,x_1)}\Bigg\}     \no \\
& \quad = \lim_{x_2 \downarrow a} \Bigg\{\f{\lim_{x_1 \downarrow a} 
\Big(\f{G(z,x_1,x_2) - \wti G(z,a,x_2) \hatt u_a(\lambda_0,x_1)}{u_a(\lambda_0,x_1)}\Big)}{u_a(\lambda_0,x_2)}     \no \\
& \hspace*{1.7cm} - \f{\lim_{x'_2 \downarrow a} \lim_{x_1 \downarrow a} \Big(\f{G(z,x_1,x'_2) 
- \wti G(z,a,x'_2) \hatt u_a(\lambda_0,x_1)}{u_a(\lambda_0,x_1) \hatt u_a(\lambda_0,x'_2)}\Big) 
\hatt u_a(\lambda_0,x_2)}{u_a(\lambda_0,x_2)}\Bigg\}     \no \\
& \quad = \big(\wti \partial_2 \wti \partial_1 G\big)(z,a,a),     \no 
\end{align}
where $G$ represents $G_{\eta,R}$. The actual computations are greatly simplified since by \eqref{3.25}, $G_{\eta,R}$ is a finite sum of products of solutions of $\tau u = z u$ and hence one can employ the facts,  
\begin{align}
& [f(x_1) g(x_2)]^{\; \wti{}} \, \big|_{x_1=a, \, x_2=a} = \wti f(a) \wti g(a),    \no \\
& \wti \partial_1 [f(x_1) g(x_2)]^{\; \wti{}} \, \big|_{x_1=a, \, x_2=a} = {\wti f}^{\, \prime}(a) \wti g(a),    \no \\
& \wti \partial_2 [f(x_1) g(x_2)]^{\; \wti{}} \, \big|_{x_1=a, \, x_2=a} = \wti f(a) {\wti g}^{\, \prime}(a),   \lb{3.31} \\
& \wti \partial_1 \wti \partial_2 [f(x_1) g(x_2)] \big|_{x_1=a, \, x_2=a} = {\wti f}^{\, \prime}(a) {\wti g}^{\, \prime}(a) 
= \wti \partial_2 \wti \partial_1 [f(x_1) g(x_2)] \big|_{x_1=a, \, x_2=a},   \no \\
& \hspace*{8.15cm} f, g \in \dom(T_{max}).     \no 
\end{align}

For basic literature on the notion of Weyl--Titchmarsh--Kodaira $m$-functions see \cite{EGNT13}, \cite{FL23}, \cite{Fu08}, \cite{FL10}, \cite{GLN20}, \cite{GZ06}, \cite[Ch.~13]{GNZ23}, \cite{Ko49}, \cite{KST12}, \cite{KT11}, and the references cited therein.

\section{The Jacobi Operator and its Boundary Values} \lb{s4}

We now turn to the principal topic of this paper, the Jacobi differential expression 
\begin{align} \lb{4.1}
\begin{split} 
\tau_{\a,\b} = - (1-x)^{-\a} (1+x)^{-\b}(d/dx) \big((1-x)^{\a+1}(1+x)^{\b+1}\big) (d/dx),&     \\ 
\a, \b \in \bbR, \; x \in (-1,1), 
\end{split} 
\end{align}
that is, in connection with Sections \ref{s2}, \ref{s3} one now has  
\begin{align}
& a=-1, \quad b = 1,    \no \\
& p(x) = p_{\a,\b}(x) = (1-x)^{\a+1}(1+x)^{\b+1}, \quad q (x) = q_{\a,\b}(x) = 0,  \lb{4.2} \\ 
& r(x) = r_{\a,\b}(x) = (1-x)^\a (1+x)^\b, \quad \a,\b\in\bbR, \; x\in(-1,1)     \no 
\end{align}
(see, e.g. \cite[Ch.~22]{AS72}, \cite{BEZ01}, \cite{Bo29}, \cite[Sect.~23]{Ev05}, \cite[Ch.~4]{Is05}, \cite[Sects.~VII.6.1, XIV.2]{Kr02}, \cite[Ch.~18]{OLBC10}, \cite[Ch.~IV]{Sz75}). 

$L^2$-realizations of $\tau_{\a,\b}$ are thus most naturally associated with the Hilbert space 
$L^2((-1,1); r_{\a,\b} dx)$. However, occasionally the weight function is absorbed into the differential expression, as in \eqref{4.4a} below, leading to the unweighted Hilbert space $L^2((-1,1);dx)$ (cf. \cite[p.~1510--1520]{DS88}, \cite[Sect.~24]{Ev05}, \cite{Gr74}). Indeed, utilizing the unitary map
\begin{equation}
U_{\al,\be} \colon \begin{cases}
L^2((-1,1); dx) \to L^2\big((-1,1); r_{\al,\be}dx\big),     \lb{4.3a} \\
f \mapsto (1 - \dott)^{-\al/2} (1 + \dott)^{-\be/2} f
\end{cases}
\end{equation}
one confirms that 
\begin{align}
& (1-x)^{\al/2} (1+x)^{\be/2} \tau_{\al,\be} (1-x)^{- \al/2} (1+x)^{- \be/2}     \no \\
& \quad = - \f{d}{dx} \big(1-x^2\big) \f{d}{dx} + \f{\al^2}{4} \f{1+x}{1-x} 
+ \f{\be^2}{4} \f{1-x}{1+x} - \f{\al + \be + \al \be}{2},       \lb {4.4a} \\
& \hspace*{5.93cm} \al, \be \in \bbR, \; x \in (-1,1).    \no
\end{align}
For more recent developments see, for instance, \cite{EKLWY07}, \cite{Fr20}, \cite{FL20}, \cite{KKT18}, \cite{KMO05}.

For later convenience we also mention the following relations that hold for all $\a, \b \in \bbR$,
\begin{align} 
& (1+x)^{-\b} \tau_{\a,-\b} (1+x)^\b = \tau_{\a,\b} + (1+\a)\b,    \no\\
& (1-x)^{-\a} \tau_{-\a,\b} (1-x)^\a = \tau_{\a,\b} + (1+\b)\a,   \lb{4.3} \\
& (1-x)^{-\a}(1+x)^{-\b} \tau_{-\a,-\b} (1-x)^\a(1+x)^\b = \tau_{\a,\b} + \a+\b,  \no 
    \end{align}
where $(1+x)^{\pm \b}$ and $(1-x)^{\pm \a}$ are regarded as operators of multiplication.

To decide the limit point/limit circle classification of $\tau_{\a,\b}$ at the interval endpoints $\pm 1$, it suffices 
to note that if $y_1$ is a given solution of $\tau y = 0$, then a 2nd linearly independent solution $y_2$ of 
$\tau y = 0$ is obtained via the standard formula
\begin{equation}
y_2(x) = y_1(x) \int_c^x dx' \, p(x')^{-1} y_1(x')^{-2}, \quad c, x \in (a,b).    \lb{4.4} 
\end{equation}
Returning to the concrete Jacobi case at hand, one notices that 
\begin{align} 
& y_1(x) = 1, \quad x \in (-1,1), \no \\
& y_2(x) = \int_0^x dx' \, (1 - x')^{-1-\a} (1+x')^{-1-\b}      \lb{4.5} \\
&\quad = \begin{cases}
2^{-1-\a} \b^{-1} (1+ x)^{-\b} [1+\Oh(1+ x)] + \Oh(1), & \a \in \bbR, \, \b \in \bbR\backslash\{0\}, \; \text{as $x \downarrow -1$}, \\
- 2^{-1-\a} \ln(1+ x) + \Oh(1), & \a \in \bbR, \, \b = 0, \; \text{as $x \downarrow -1$},  \\
2^{-1-\b} \a^{-1} (1 - x)^{-\a} [1+\Oh(1- x)] + \Oh(1), & \a \in \bbR\backslash\{0\}, \, \b \in \bbR, \; \text{as $x \uparrow 1$}, \\
- 2^{-1-\b} \ln(1 - x) + \Oh(1), & \a =0, \, \b \in \bbR, \; \text{as $x \uparrow 1$}.   
\end{cases}    \no 
\end{align} 
Thus, an application of Theorem \ref{t2.3}, Definition \ref{d2.4}, and Remark \ref{r2.7}\,$(ii)$ implies the classification,
\begin{equation}
\tau_{\a,\b} \, \text{ is } \begin{cases} \text{regular at $-1$ if and only if $\a \in \bbR$, $\b \in (-1,0)$,} \\
\text{in the limit circle case (and not regular) at $-1$ if and only if} \\
\quad \text{$\a \in \bbR$, $\b \in [0,1)$,} \\
\text{in the limit point case at $-1$ if and only if $\a \in \bbR$, $\b \in \bbR \backslash (-1,1)$,} \\
\text{regular at $1$ if and only if $\a \in (-1,0)$, $\b \in \bbR$,} \\
\text{in the limit circle case (and not regular) at $1$ if and only if} \\ 
\quad \text{$\a \in [0,1)$, $\b \in \bbR$,} \\
\text{in the limit point case at $1$ if and only if $\a \in \bbR \backslash (-1,1)$, $\b \in \bbR$.}
\end{cases}     \lb{4.6}
\end{equation}
The maximal and preminimal operators, $T_{max,\a,\b}$ and $\dot T_{min,\a,\b}$, associated to $\tau_{\a,\b}$ 
in $L^2((-1,1); r_{\a,\b} dx)$ are then given by
\begin{align}
&T_{max,\a,\b} f = \tau_{\a,\b} f,     \no
\\
& f \in \dom(T_{max,\a,\b})=\big\{g\in L^2((-1,1); r_{\a,\b} dx) \, \big| \,  g,g^{[1]}\in AC_{loc}((-1,1)); \no \\
& \hspace*{6.8cm} \tau_{\a,\b} g\in L^2((-1,1); r_{\a,\b} dx)\big\},
\end{align}
and
\begin{align}
& \dot T_{min,\a,\b} f = \tau_{\a,\b} f, \no
\\
& f \in \dom\big(\dot T_{min,\a,\b}\big)=\big\{g\in L^2((-1,1); r_{\a,\b} dx)  \, \big| \,  g,g^{[1]}\in AC_{loc}((-1,1));    \\ 
&\hspace*{2.05cm} \supp \, (g)\subset(-1,1) \text{ is compact; } 
\tau_{\a,\b} g\in L^2((-1,1); r_{\a,\b} dx)\big\}.   \no 
\end{align}

The fact \eqref{4.5} naturally leads to principal and nonprincipal solutions 
$u_{\pm1,\a,\b}(0,x)$ and $\hatt u_{\pm1,\a,\b}(0,x)$ of $\tau_{\a,\b}y=0$ near $\pm1$ as follows: 
\begin{align} 
\begin{split} 
u_{-1,\a,\b}(0,x)&=\begin{cases}
-2^{-\a-1}\b^{-1} (1+x)^{-\b}[1+\Oh(1+x)], & \b\in(-\infty,0),\\
1, & \b\in[0,\infty),
\end{cases}     \\
\hatt u_{-1,\a,\b}(0,x)&=\begin{cases}
1, & \b\in(-\infty,0),\\
-2^{-\a-1}\ln((1+x)/2), & \b=0,\\
2^{-\a-1} \b^{-1} (1+x)^{-\b} [1+\Oh(1+x)], & \b\in(0,\infty);
\end{cases}
\end{split} 
\quad \a\in\bbR,     \lb{4.9} 
\end{align}
and
\begin{align}
\begin{split} 
u_{1,\a,\b}(0,x)&=\begin{cases}
2^{-\b-1} \a^{-1} (1-x)^{-\a} [1+\Oh(1-x)], & \a\in(-\infty,0),\\
1, & \a\in[0,\infty),
\end{cases}      \\
\hatt u_{1,\a,\b}(0,x)&=\begin{cases}
1, & \a\in(-\infty,0),\\
2^{-\b-1}\ln((1-x)/2), & \a=0,\\
-2^{-\b-1} \a^{-1} (1-x)^{-\a} [1+\Oh(1-x)], & \a\in(0,\infty);
\end{cases}
\end{split} 
\quad \b\in\bbR.     \lb{4.10} 
\end{align}

Combining the fact \eqref{4.6} with Theorem \ref{t2.5}, $\dot T_{min,\a,\b}$ is essentially self-adjoint in 
$L^2((-1,1); r_{\a,\b} dx)$ if and only if 
$\a, \b \in \bbR \backslash (-1,1)$. Thus, boundary values for $T_{max,\a,\b}$ at $-1$ exist if and only if 
$\a \in \bbR$, $\b \in (-1,1)$, and similarly, boundary values for $T_{max,\a,\b}$ at $1$ exist if and only if 
$\a \in (-1,1)$, $\b \in \bbR$. 

Employing the principal and nonprincipal solutions \eqref{4.9}, \eqref{4.10} at $\pm 1$, according to 
\eqref{2.25}, \eqref{2.26}, generalized boundary values for $g\in\dom(T_{max,\a,\b})$ are of the form
\begin{align}
\begin{split} 
\wti g(-1)&=\begin{cases}
g(-1), & \b\in(-1,0),\\
-2^{\a+1}\lim_{x\downarrow-1}g(x)/\ln((1+x)/2), & \b=0,\\
\b 2^{\a+1}\lim_{x\downarrow-1}(1+x)^\b g(x), & \b\in(0,1),
\end{cases}    \\
{\wti g}^{\, \prime}(-1)&=\begin{cases}
g^{[1]}(-1), &\b\in(-1,0),\\
\lim_{x\downarrow-1}\big[g(x)+\wti g(-1)2^{-\a-1}\ln((1+x)/2)\big], & \b=0,\\
\lim_{x\downarrow-1}\big[g(x)-\wti g(-1)2^{-\a-1}\b^{-1}(1+x)^{-\b}\big], & \b\in(0,1);
\end{cases}
\end{split} 
\quad \a\in\bbR,     \lb{4.11} \\
\begin{split}
\wti g(1)&=\begin{cases}
g(1), & \a\in(-1,0),\\
2^{\b+1}\lim_{x\uparrow1}g(x)/\ln((1-x)/2), & \a=0,\\
-\a2^{\b+1}\lim_{x\uparrow1}(1-x)^\a g(x), & \a\in(0,1),
\end{cases}\\
{\wti g}^{\, \prime}(1)&=\begin{cases}
g^{[1]}(1), & \a\in(-1,0),\\
\lim_{x\uparrow1}\big[g(x)-\wti g(1)2^{-\b-1}\ln((1-x)/2)\big], & \a=0,\\
\lim_{x\uparrow1}\big[g(x)+\wti g(1)2^{-\b-1}\a^{-1}(1-x)^{-\a}\big], & \a\in(0,1);
\end{cases}
\end{split}
\quad \b\in\bbR.
\end{align}

As a result (cf., \eqref{2.5}, \eqref{2.10}, \eqref{2.25}, \eqref{2.26}), the minimal operator 
$T_{min,\a,\b}$ associated to $\tau_{\a,\b}$, that is, $T_{min,\a,\b} = \ol{\dot T_{min,\a,\b}}$, is thus given by 
\begin{align}
&T_{min,\a,\b} f = \tau_{\a,\b} f, \no \\
& f \in \dom(T_{min,\a,\b})=\big\{g\in L^2((-1,1); r_{\a,\b} dx) \, \big| \, g,g^{[1]}\in AC_{loc}((-1,1));    \\ 
&\hspace*{1.1cm} \wti g(-1) = {\wti g}^{\, \prime}(-1) = \wti g(1) = {\wti g}^{\, \prime}(1) = 0; \, 
\tau_{\a,\b} g\in L^2((-1,1); r_{\a,\b} dx)\big\}.  \no
\end{align}

All self-adjoint extensions of $T_{min,\a,\b}$ are thus either of the type \eqref{2.27} in the case of separated boundary conditions, or of the type \eqref{2.27A} in the case of coupled boundary conditions. In particular, the Friedrichs extension $T_{F,\a,\b}$ of $T_{min,\a,\b}$ is given as in \eqref{2.28}. 

We note that self-adjoint extensions of $T_{min,\a,\b}$ in the limit circle case at both endpoints $x=\pm1$, that is, 
$\a, \b \in [0,1)$, were studied in \cite{Fr20} and \cite{FL20}, in particular, the self-adjoint extension containing Jacobi polynomials in its domain was characterized. In addition, \cite{Fr20} and \cite{FL20} discussed the domains of powers of self-adjoint Jacobi operators. Self-adjoint extensions in the case $\a,\b \in (0,1)$ are also studied in \cite{BFL20}; the case $\a, \b \in (-1,\infty)$ was also treated by Gr\"unewald \cite{Gr74}. 

For a detailed treatment of solutions of the Jacobi differential equation and the associated hypergeometric differential equations we refer to Appendices \ref{sA}--\ref{sC}.

\section{The Regular and Limit Circle Case $\a, \b \in (-1,1)$}\lb{s5}

In this section we compute the Weyl--Titchmarsh--Kodaira and Green's function induced $m$ (resp., $M$) functions when the Jacobi differential expression $\tau_{\al, \be}$ is either in the regular or limit circle case at $\pm 1$. 

We begin by determining the solutions $\varphi_{\g,\a,\b}(z,\dott)$ and $\vartheta_{\g,\a,\b}(z,\dott)$ of $\tau_{\a,\b}u=zu,z\in\bbC$, that are subject to the conditions
\begin{align}
\begin{split}
\wti \varphi_{\g,\a,\b} (z,-1) &= - \sin(\g), \quad \wti \varphi_{\g,\a,\b}^{\, \prime} (z,-1) = \cos(\g),\\
\wti \vartheta_{\g,\a,\b} (z,-1) &= \cos(\g), \quad \;\;\,\, \wti \vartheta_{\g,\a,\b}^{\, \prime} (z,-1) = \sin(\g); 
\quad \g\in[0,\pi).    \lb{5.1} 
\end{split}
\end{align}
In particular, one can write (see Appendix \ref{sA} for the solutions $y_{j,\alpha,\beta,-1}(z,\dott)$, $j=1,2$)
\begin{align}
\varphi_{\g,\a,\b} (z,x) &= c_{\g,\varphi,1}(z)y_{1,\a,\b,-1}(z,x)+c_{\g,\varphi,2}(z)y_{2,\a,\b,-1}(z,x),   \lb{5.2}\\
\vartheta_{\g,\a,\b} (z,x) &=c_{\g,\vartheta,1}(z)y_{1,\a,\b,-1}(z,x)+c_{\g,\vartheta,2}(z)y_{2,\a,\b,-1}(z,x);\ \ z\in\bbC,\ x\in(-1,1),   
\no 
\end{align}
where the coefficients $c_{\g,\varphi,1}(z),c_{\g,\varphi,2}(z),c_{\g,\vartheta,1}(z),c_{\g,\vartheta,2}(z)\in\bbC$ are given by 
\begin{align}
\begin{split}
c_{\g,\varphi,1}(z)&=\f{-\sin(\g)\wti y_{2,\a,\b,-1}^{\, \prime}(z,-1)-\cos(\g)\wti y_{2,\a,\b,-1}(z,-1)}{\wti y_{1,\a,\b,-1}(z,-1)\wti y_{2,\a,\b,-1}^{\, \prime}(z,-1)-\wti y_{1,\a,\b,-1}^{\, \prime}(z,-1)\wti y_{2,\a,\b,-1}(z,-1)},\\[1mm]
c_{\g,\varphi,2}(z)&=\f{\cos(\g)\wti y_{1,\a,\b,-1}(z,-1)+\sin(\g)\wti y_{1,\a,\b,-1}^{\, \prime}(z,-1)}{\wti y_{1,\a,\b,-1}(z,-1)\wti y_{2,\a,\b,-1}^{\, \prime}(z,-1)-\wti y_{1,\a,\b,-1}^{\, \prime}(z,-1)\wti y_{2,\a,\b,-1}(z,-1)},\\[1mm]
c_{\g,\vartheta,1}(z)&=\f{\cos(\g)\wti y_{2,\a,\b,-1}^{\, \prime}(z,-1)-\sin(\g)\wti y_{2,\a,\b,-1}(z,-1)}{\wti y_{1,\a,\b,-1}(z,-1)\wti y_{2,\a,\b,-1}^{\, \prime}(z,-1)-\wti y_{1,\a,\b,-1}^{\, \prime}(z,-1)\wti y_{2,\a,\b,-1}(z,-1)},\\[1mm]
c_{\g,\vartheta,2}(z)&=\f{\sin(\g)\wti y_{1,\a,\b,-1}(z,-1)-\cos(\g)\wti y_{1,\a,\b,-1}^{\, \prime}(z,-1)}{\wti y_{1,\a,\b,-1}(z,-1)\wti y_{2,\a,\b,-1}^{\, \prime}(z,-1)-\wti y_{1,\a,\b,-1}^{\, \prime}(z,-1)\wti y_{2,\a,\b,-1}(z,-1)}.     \lb{5.3}
\end{split}
\end{align}

\subsection{The Regular and Limit Circle Case for Separated Boundary Conditions}\lb{s5a}
\hfill

Given these preparations, one obtains the principal result on $m$-functions corresponding to separated boundary conditions of this section:

\begin{theorem} \lb{t5.1}
Let $T_{\g,\d,\a,\b}$, $\g, \d \in [0,\pi)$, be the self-adjoint operator associated with $\tau_{\a,\b}$ applying 
$\g$ -boundary conditions at $x=-1$ and $\d$-boundary conditions at $x=1$ in the regular or limit circle case 
$\a, \b \in (-1,1)$. Then the associated Weyl--Titchmarsh--Kodaira $m$-function is of the form 
\begin{align}\no
m_{\g,\d,\a,\b}(z)=&\begin{cases}
\big[\b 2^{\a+1}\cos(\g)(\cos(\d)\wti y_{1,\a,\b,-1}(z,1)+\sin(\d)\wti y_{1,\a,\b,-1}^{\, \prime}(z,1))\\
\quad -\sin(\g)(\cos(\d)\wti y_{2,\a,\b,-1}(z,1)+\sin(\d)\wti y_{2,\a,\b,-1}^{\, \prime}(z,1))\big]\\
\times\big[\b 2^{\a+1}\sin(\g)(\cos(\d)\wti y_{1,\a,\b,-1}(z,1)+\sin(\d)\wti y_{1,\a,\b,-1}^{\, \prime}(z,1))\\
\quad +\cos(\g)(\cos(\d)\wti y_{2,\a,\b,-1}(z,1)+\sin(\d)\wti y_{2,\a,\b,-1}^{\, \prime}(z,1))\big]^{-1},\\
\hspace{7cm} \b\in(-1,0),\\[1mm]
\big[ -2^{\a+1}\sin(\g)(\cos(\d)\wti y_{1,\a,\b,-1}(z,1)+\sin(\d)\wti y_{1,\a,\b,-1}^{\, \prime}(z,1))\\
\quad +\cos(\g)(\cos(\d)\wti y_{2,\a,\b,-1}(z,1)+\sin(\d)\wti y_{2,\a,\b,-1}^{\, \prime}(z,1))\big]\\
\times\big[2^{\a+1}\cos(\g)(\cos(\d)\wti y_{1,\a,\b,-1}(z,1)+\sin(\d)\wti y_{1,\a,\b,-1}^{\, \prime}(z,1))\\
\quad +\sin(\g)(\cos(\d)\wti y_{2,\a,\b,-1}(z,1)+\sin(\d)\wti y_{2,\a,\b,-1}^{\, \prime}(z,1))\big]^{-1},\\[1mm]
\hspace{7.8cm} \b=0,\\[1mm]
\big[\b 2^{\a+1}\sin(\g)(\cos(\d)\wti y_{1,\a,\b,-1}(z,1)+\sin(\d)\wti y_{1,\a,\b,-1}^{\, \prime}(z,1))\\
\quad +\cos(\g)(\cos(\d)\wti y_{2,\a,\b,-1}(z,1)+\sin(\d)\wti y_{2,\a,\b,-1}^{\, \prime}(z,1))\big]\\
\times\big[-\b 2^{\a+1}\cos(\g)(\cos(\d)\wti y_{1,\a,\b,-1}(z,1)+\sin(\d)\wti y_{1,\a,\b,-1}^{\, \prime}(z,1))\\
\quad +\sin(\g)(\cos(\d)\wti y_{2,\a,\b,-1}(z,1)+\sin(\d)\wti y_{2,\a,\b,-1}^{\, \prime}(z,1))\big]^{-1},\\[1mm]
\hspace{7.35cm} \b\in(0,1); 
\end{cases}\\
&\hspace{3cm} \a\in(-1,1),\ \g,\d\in[0,\pi),\ z\in\rho(T_{\g,\d,\a,\b}). \lb{5.4}
\end{align}
The $($necessarily simple\,$)$ poles of $m_{\g,\d,\a,\b}$ occur precisely at the $($necessarily simple\,$)$ eigenvalues 
of $T_{\g,\d,\a,\b}$. 
\end{theorem}
\begin{proof}
Substituting the boundary values \eqref{C.13} into \eqref{5.3} and combining this with equation \eqref{5.2} yields $\varphi_{\g,\a,\b} (z,x)$ and $\vartheta_{\g,\a,\b} (z,x)$, $\g\in[0,\pi)$. Thus, \eqref{3.18} implies the form of the general $m$-function as displayed in \eqref{5.4}.
\end{proof}

Theorem \ref{t5.1} permits one to efficiently compute any $m$-function for $\a,\b\in(-1,1)$ by substituting the appropriate boundary values \eqref{C.14}--\eqref{C.16} into \eqref{5.4}. In particular, the special Legendre case, 
$\a=\b=0$, to be reconsidered in Example \ref{e5.3} {\boldmath $\mathbf{(II)}$} below, has frequently been discussed in the literature, 
see, for instance, \cite{GLN20} and the extensive list of references cited therein.

\begin{remark}
We mention a few other special cases of interest other than Legendre. The Gegenbauer, or ultraspherical, equation (see, e.g., \cite[Ch. 22]{AS72}, \cite[Ch.~18]{OLBC10}, \cite[Ch. IV]{Sz75}) can be realized by choosing the parameters $\a=\b=\mu-1/2$, noting  at the endpoints $x=\pm1$, $\tau_\mu$ is regular for $\mu\in(-1/2,1/2)$, in the limit circle case for $\mu\in[1/2,3/2)$, and in the limit point case for $\mu\in\bbR\backslash(-1/2,3/2)$. In particular, choosing $\mu=1/2$ we arrive at the Legendre equation once again. After the appropriate change in parameter, \eqref{5.4} describes the form of the $m$-function for the new parameter $\mu\in(-1/2,3/2)$. For $\mu\in\bbR\backslash(-1/2,3/2)$, see Theorem \ref{t7.1}. 

The Chebyshev equations of the first and second kinds are two more important special cases covered in Example \ref{e5.3} $\mathbf{(I)}$ and $\mathbf{(III)}$ respectively. The Chebyshev equation of the first kind is realized by choosing $\mu=0$ in the Gegenbauer equation, or $\a=\b=-1/2$ in the Jacobi equation (see, e.g., \cite[Ch. 22]{AS72}, \cite[Ch.~18]{OLBC10}, \cite[Ch. IV]{Sz75}), whereas the Chebyshev equation of the second kind is realized by choosing $\mu=1$ in the Gegenbauer equation, or $\a=\b=1/2$ in the Jacobi equation (see, e.g., \cite[Ch. 22]{AS72}, \cite[Ch.~18]{OLBC10}, \cite[Ch. IV]{Sz75}).

See Appendix \ref{sE} for more details.
\hfill$\diamond$
\end{remark}

In the following we employ the abbreviation
\begin{equation}
\sigma_{\a,\b}(z) = \big[(1+\a+\b)^2 + 4z\big]^{1/2}, \quad \a, \b \in \bbR, \; z \in \bbC.   \lb{5.5} 
\end{equation}

The following example provides an illustration of Theorem \ref{t5.1} while applying \eqref{C.14}--\eqref{C.16}, \eqref{5.1}--\eqref{5.4}. Throughout, we recall that the $($necessarily simple\,$)$ poles of $m_{\g,\d,\a,\b}$ occur precisely at the $($necessarily simple\,$)$ eigenvalues of $T_{\g,\d,\a,\b}$.

\begin{example}\lb{e5.3}
When boundary conditions are chosen such that the Jacobi polynomials, $P_n^{\a,\b}(\dott)$ $($see \eqref{A.31}$)$, satisfy them for the given choice of parameters $\a,\b$, one verifies the spectrum is given by $\{n(n+1+\a+\b)\}_{n\in\bbN_0}$ in accordance with Remark \ref{rA1} $($see $\mathbf{(I)}$--$\mathbf{(III)}$, $\mathbf{(VII)}$, $\mathbf{(VIII)}$ below$)$.\\[1mm]
\noindent 
{\boldmath $\mathbf{(I)}$ {\bf The case} $\a,\b\in(-1,0)$$:$}  \\[1mm] 
\noindent The Neumann extension $\g=\d=\pi/2$, $T_{N,\a,\b} = T_{\f{\pi}{2},\f{\pi}{2},\a,\b}$, yields for $z\in\rho(T_{N,\a,\b})$,
\begin{align}
&m_{\frac{\pi}{2},\frac{\pi}{2},\a,\b}(z)=\f{-\wti y_{2,\a,\b,-1}^{\, \prime}(z,1)}{\b2^{\a+1}\wti y_{1,\a,\b,-1}^{\, \prime}(z,1)}  \no \\
& \quad =\f{-\Gamma(1-\b)\Gamma([1+\a+\b+\sigma_{\a,\b}(z)]/2)
\Gamma([1+\a+\b-\sigma_{\a,\b}(z)]/2)}{\b 2^{1+\a+\b}\Gamma(1+\b) 
\Gamma([1+\a-\b+\sigma_{\a,\b}(z)]/2)\Gamma([1+\a-\b-\sigma_{\a,\b}(z)]/2)},     \no \\
&\quad \sigma(T_{N,\a,\b}) = \{n(n+1+\a+\b)\}_{n\in\bbN_0}; \quad \a,\b\in(-1,0).
\end{align}
\noindent 
{\boldmath $\mathbf{(II)}$ {\bf The case} $\a=\b=0$ $(${\bf Legendre}$)$$:$}\\[1mm] 
The Friedrichs extension, $\g=\d=0$, $T_{F,Leg} = T_{0,0,0,0}$, yields for $z\in\rho(T_{F,Leg})$,
\begin{align}
m_{0,0,Leg}(z)&=\f{\wti y_{2,0,0,-1}(z,1)}{2\wti y_{1,0,0,-1}(z,1)}    \no \\
&=-\gamma_E- \{\psi([1+\sigma_{0,0}(z)]/2)+\psi([1-\sigma_{0,0}(z)]/2)\}/2   \no \\
&=-\gamma_E-\psi([1+\sigma_{0,0}(z)]/2) + (\pi/2)\cot(\pi [1- \sigma_{0,0}(z)]/2),  \no \\
\sigma(T_{F,Leg}) &= \{n(n+1)\}_{n\in\bbN_0}; \quad \al = \b = 0,  \lb{5.6}
\end{align}
in agreement with \cite[eq. (6.46)]{GLN20}. Here we used the reflection formula 
$($cf. \cite[eq. 6.3.6]{AS72}$)$
\begin{align}
\psi(1-z)-\psi(z)=\pi\cot(\pi z). 
\end{align}
\noindent 
{\boldmath $\mathbf{(III)}$ {\bf The case} $\a,\b\in(0,1)$$:$}\\[1mm] 
The Friedrichs extension, $\g=\d=0$, $T_{F,\a,\b} = T_{0,0,\a,\b}$, yields for $z\in\rho(T_{F,\a,\b})$,
\begin{align}
& m_{0,0,\a,\b}(z)=\f{\wti y_{2,\a,\b,-1}(z,1)}{-\b2^{\a+1}\wti y_{1,\a,\b,-1}(z,1)}   \no \\
& \quad = \f{-\Gamma(1-\b)\Gamma([1+\a+\b+\sigma_{\a,\b}(z)]/2)
\Gamma([1+\a+\b-\sigma_{\a,\b}(z)]/2)}{\b 2^{1+\a+\b}\Gamma(1+\b)
\Gamma([1+\a-\b+\sigma_{\a,\b}(z)]/2)\Gamma([1+\a-\b-\sigma_{\a,\b}(z)]/2)},    \no\\
&\ \ \sigma(T_{F,\a,\b}) = \{n(n+1+\a+\b)\}_{n\in\bbN_0};  \quad \a,\b \in(0,1).
\end{align}
\noindent 
{\boldmath $\mathbf{(IV)}$ {\bf The case} $\a \in (0,1), \; \b \in (-1,0)$$:$}\\[1mm]
The Friedrichs extension, $\g=\d=0$, $T_{F,\a,\b} = T_{0,0,\a,\b}$, yields for $z\in\rho(T_{F,\a,\b})$,
\begin{align}
& m_{0,0,\a,\b}(z)=\f{\b2^{\a+1}\wti y_{1,\a,\b,-1}(z,1)}{\wti y_{2,\a,\b,-1}(z,1)}   \no \\
& \quad = \f{\b 2^{1+\a+\b}\Gamma(1+\b)\Gamma([1+\a-\b+\sigma_{\a,\b}(z)]/2)
\Gamma([1+\a-\b-\sigma_{\a,\b}(z)]/2)}{\Gamma(1-\b)
\Gamma([1+\a+\b+\sigma_{\a,\b}(z)]/2)\Gamma([1+\a+\b-\sigma_{\a,\b}(z)]/2)},    \no\\
&\ \ \sigma(T_{F,\a,\b}) = \{(n-\b)(n+1+\a)\}_{n\in\bbN_0};  \quad \a \in (0,1), \; \b \in (-1,0).
\end{align}
\noindent 
{\boldmath $\mathbf{(V)}$ {\bf The case} $\a \in (-1,0), \; \b \in (0,1)$$:$}\\[1mm]
The Friedrichs extension, $\g=\d=0$, $T_{F,\a,\b} = T_{0,0,\a,\b}$, yields for $z\in\rho(T_{F,\a,\b})$,
\begin{align}
& m_{0,0,\a,\b}(z)=\f{\wti y_{2,\a,\b,-1}(z,1)}{-\b2^{\a+1}\wti y_{1,\a,\b,-1}(z,1)}   \no \\
& \quad = \f{-\Gamma(1-\b)\Gamma([1-\a+\b+\sigma_{\a,\b}(z)]/2)
\Gamma([1-\a+\b-\sigma_{\a,\b}(z)]/2)}{\b 2^{1+\a+\b}\Gamma(1+\b)
\Gamma([1-\a-\b+\sigma_{\a,\b}(z)]/2)\Gamma([1-\a-\b-\sigma_{\a,\b}(z)]/2)},    \no\\
&\ \ \sigma(T_{F,\a,\b}) = \{(n-\a)(n+1+\b)\}_{n\in\bbN_0};  \quad \a \in (-1,0); \; \b \in (0,1).
\end{align}
\noindent 
{\boldmath $\mathbf{(VI)}$ {\bf The case} $\a = 0, \; \b \in (-1,0)$$:$}\\[1mm]
The Friedrichs extension, $\g=\d=0$, $T_{F,0,\b} = T_{0,0,0,\b}$, yields for $z\in\rho(T_{F,0,\b})$,
\begin{align}
m_{0,0,0,\b}(z)&=\f{\b2\wti y_{1,0,\b,-1}(z,1)}{\wti y_{2,0,\b,-1}(z,1)}   \no \\
&= \f{\b 2^{1+\b}\Gamma(1+\b)\Gamma([1-\b+\sigma_{0,\b}(z)]/2)
\Gamma([1-\b-\sigma_{0,\b}(z)]/2)}{\Gamma(1-\b)
\Gamma([1+\b+\sigma_{0,\b}(z)]/2)\Gamma([1+\b-\sigma_{0,\b}(z)]/2)},    \no\\
\sigma(T_{F,0,\b}) &= \{(n-\b)(n+1)\}_{n\in\bbN_0}; \quad \a = 0, \; \b \in (-1,0).
\end{align}
\noindent 
{\boldmath $\mathbf{(VII)}$ {\bf The case} $\a =0, \; \b \in (0,1)$$:$}\\[1mm]
The Friedrichs extension, $\g=\d=0$, $T_{F,0,\b} = T_{0,0,0,\b}$, yields for $z\in\rho(T_{F,0,\b})$,
\begin{align}
m_{0,0,0,\b}(z)&=\f{\wti y_{2,0,\b,-1}(z,1)}{-\b2\wti y_{1,0,\b,-1}(z,1)}   \no \\
& = \f{-\Gamma(1-\b)\Gamma([1+\b+\sigma_{0,\b}(z)]/2)
\Gamma([1+\b-\sigma_{0,\b}(z)]/2)}{\b 2^{1+\b}\Gamma(1+\b)
\Gamma([1-\b+\sigma_{0,\b}(z)]/2)\Gamma([1-\b-\sigma_{0,\b}(z)]/2)},    \no\\
\sigma(T_{F,0,\b}) &= \{n(n+1+\b)\}_{n\in\bbN_0};  \quad \a =0, \; \b \in (0,1).
\end{align}
\noindent 
{\boldmath $\mathbf{(VIII)}$ {\bf The case $\a \in (0,1), \; \b = 0$\,:}}\\[1mm]
The Friedrichs extension, $\g=\d=0$, $T_{F,\a,0} = T_{0,0,\a,0}$, yields for $z\in\rho(T_{F,\a,0})$,
\begin{align}
m_{0,0,\a,0}(z)&=\f{\wti y_{2,\a,0,-1}(z,1)}{2^{\a+1}\wti y_{1,\a,0,-1}(z,1)}   \no \\
&= -2^{-1-\a}[2\gamma_E+\psi([1+\a+\sigma_{\a,0}(z)]/2)+\psi([1+\a-\sigma_{\a,0}(z)]/2)],    \no\\
\sigma(T_{F,\a,0}) &= \{n(n+1+\a)\}_{n\in\bbN_0};  \quad \a \in (0,1); \; \b = 0.
\end{align}
\noindent 
{\boldmath $\mathbf{(IX)}$ {\bf The case $\a \in (-1,0), \; \b = 0$\,:}}\\[1mm]
The Friedrichs extension, $\g=\d=0$, $T_{F,\a,0} = T_{0,0,\a,0}$, yields for $z\in\rho(T_{F,\a,0})$,
\begin{align}
m_{0,0,\a,0}(z)&=\f{\wti y_{2,\a,\b,-1}(z,1)}{2^{\a+1}\wti y_{1,\a,\b,-1}(z,1)}   \no \\
&= -2^{-1-\a}[2\gamma_E+\psi([1-\a+\sigma_{\a,0}(z)]/2)+\psi([1-\a-\sigma_{\a,0}(z)]/2)], \no \\
\sigma(T_{F,\a,0}) &= \{(n-\a)(n+1)\}_{n\in\bbN_0};  \quad \a \in (-1,0), \; \b = 0.
\end{align}
\end{example}
We note that the relations in \eqref{4.3} can be extended to the Hilbert space setting using the unitary operators $U_{\pm \b}$ and $V_{\pm \a}$:
\begin{align} \lb{5.15}
\begin{split}
& U_{\pm \b}\colon \begin{cases} L^2((-1,1);r_{\a,\pm\b}dx) \rightarrow L^2((-1,1);r_{\a,\mp \b}dx),   \\
 f \mapsto (1+ \,\cdot\,)^{\pm \b}f,
 \end{cases}     \\[1mm]
& V_{\pm \a}\colon \begin{cases} L^2((-1,1);r_{\pm\a,\b}dx) \rightarrow L^2((-1,1);r_{\mp\a,\b}dx),     \\
  f \mapsto (1+ \,\cdot\,)^{\pm \a}f,
\end{cases} 
\end{split}
\end{align}
resulting in 
\begin{align}
\begin{split} 
& (1+x)^{-\b} T_{\g,\d,\a,-\b} (1+x)^\b = T_{\g,\d,\a,\b} + (1+\a)\b,   \\
&  (1-x)^{-\a} T_{\g,\d,-\a,\b} (1-x)^\a = T_{\g,\d,\a,\b} + (1+\b)\a,    \\
&  (1-x)^{-\a}(1+x)^{-\b} T_{\g,\d,-\a,-\b} (1-x)^\a(1+x)^\b = T_{\g,\d,\a,\b} + \a+\b;   \\
& \hspace*{7.8cm} \g, \d \in \{ 0, \pi/2 \}.
\end{split} 
\end{align}
These relations can be used to explain the patterns found in Example \ref{e5.3}. 

\subsection{The Regular and Limit Circle Case for Coupled Boundary Conditions} \lb{s5b}
\hfill

We now turn to the case of coupled boundary conditions as described in \eqref{2.27A}. Following \cite[Ch.~13]{GNZ23} (see also \cite{GN22}, \cite{GN22a}) and utilizing generalized boundary values, we introduce the $2\times2$ matrix-valued Green's function induced $M$-function corresponding to $T_{\eta,R,\a,\b}$ via 
\begin{align}
& M_{\eta,R,\a,\b}(z,-1)      \no \\
& \quad =\Small\begin{pmatrix}
\wti G_{\eta,R,\a,\b}(z,-1,-1)   &    2^{-1}\big(\big[\wti \partial_1+\wti \partial_2\big] G_{\eta,R,\a,\b}\big)^{\wti{}} (z,-1,-1) \\[1mm]
2^{-1}\big(\big[\wti \partial_1+\wti \partial_2\big] G_{\eta,R,\a,\b}\big)^{\wti{}} (z,-1,-1)   &    \big(\wti \partial_1 \wti \partial_2 G_{\eta,R,\a,\b}\big) (z,-1,-1)
\end{pmatrix}   \no\\
&\quad =
\Small\begin{pmatrix}
\begin{matrix}
R_{2,2}\wti\varphi_{0,\a,\b}(z,1)-R_{1,2} \wti\varphi^{\, \prime}_{0,\a,\b}(z,1)\\
\quad
\end{matrix}   &
\begin{matrix}
2^{-1}\big[R_{1,1} \wti\varphi^{\, \prime}_{0,\a,\b}(z,1)-R_{2,1} \wti\varphi_{0,\a,\b}(z,1) \\
\quad-R_{2,2} \wti\vartheta_{0,\a,\b}(z,1)+R_{1,2} \wti\vartheta^{\, \prime}_{0,\a,\b}(z,1)\big]
\end{matrix}\\[4mm]
\begin{matrix}
2^{-1}\big[R_{1,1} \wti\varphi^{\, \prime}_{0,\a,\b}(z,1)-R_{2,1} \wti\varphi_{0,\a,\b}(z,1) \\
\quad-R_{2,2} \wti\vartheta_{0,\a,\b}(z,1) + R_{1,2} \wti\vartheta^{\, \prime}_{0,\a,\b}(z,1)\big]
\end{matrix}  & 
\begin{matrix}
R_{2,1}\wti\vartheta_{0,\a,\b}(z,1)-R_{1,1} \wti\vartheta^{\, \prime}_{0,\a,\b}(z,1)\\
\quad
\end{matrix} 
\end{pmatrix}   \no\\
&\qquad\times \big[-e^{i\eta}\big] \big[\wti F_{\eta,R,\a,\b}(z)\big]^{-1}, \quad  \a,\b\in(-1,1),\ z\in\rho(T_{\eta,R,\a,\b}),    \lb{8.1}
\end{align}
where $\eta\in[0,\pi)$, $R\in SL(2,\bbR)$, and the Green's function, $G_{\eta,R,\a,\b}(z,\dott,\dott)$, of $T_{\eta,R,\a,\b}$ is given by
\begin{align}
& G_{\eta,R,\a,\b}(z,x,x')=\dfrac{-e^{i\eta}}{\wti F_{\eta,R,\a,\b}(z)}  \no \\
&\qquad\times\Big\{[-R_{1,2}\wti\varphi^{\, \prime}_{0,\a,\b}(z,1)+R_{2,2} \wti\varphi_{0,\a,\b}(z,1)]\vartheta_{0,\a,\b}(z,x)\vartheta_{0,\a,\b}(z,x')  \no \\
&\hspace*{1.2cm} +[-R_{1,1}\wti\vartheta^{\, \prime}_{0,\a,\b}(z,1)+R_{2,1} \wti\vartheta_{0,\a,\b}(z,1)]\varphi_{0,\a,\b}(z,x)\varphi_{0,\a,\b}(z,x')  \no \\
&\hspace*{1.2cm} + [e^{-i\eta}+R_{1,2}\wti\vartheta^{\, \prime}_{0,\a,\b}(z,1)-R_{2,2} \wti\vartheta_{0,\a,\b}(z,1)]\vartheta_{0,\a,\b}(z,x)\varphi_{0,\a,\b}(z,x')  \no \\
&\hspace*{1.2cm} +[-e^{-i\eta}+R_{1,1}\wti\varphi^{\, \prime}_{0,\a,\b}(z,1)-R_{2,1} \wti\varphi_{0,\a,\b}(z,1)]\varphi_{0,\a,\b}(z,x)\vartheta_{0,\a,\b}(z,x')\Big\}  \no \\
&\quad +\begin{cases}
0, & a < x\leq x' < b,\\
[\vartheta_{0,\a,\b}(z,x)\varphi_{0,\a,\b}(z,x')-\varphi_{0,\a,\b}(z,x)\vartheta_{0,\a,\b}(z,x')], & a < x' \leq x < b,
\end{cases}     \no \\
&\hspace*{9cm} z\in\rho(T_{\eta,R}),  
\end{align}
with
\begin{align}
\begin{split}
\wti F_{\eta,R,\a,\b}(z)&=-e^{i\eta}[ R_{1,1}\wti\varphi^{\, \prime}_{0,\a,\b}(z,1)+R_{2,2} \wti\vartheta_{0,\a,\b}(z,1)-R_{2,1}\wti\varphi_{0,\a,\b}(z,1)    \\
&\hspace*{1.3cm}-R_{1,2}\wti\vartheta^{\, \prime}_{0,\a,\b}(z,1)-2\cos(\eta)],     \lb{8.3}\\
\varphi_{0,\a,\b}(z,x)&=\begin{cases}
-\b^{-1} 2^{-\a-1} y_{2,\a,\b,-1}(z,x), & \b\in(-1,0),\\
y_{1,\a,0,-1}(z,x), &  \b=0,\\
y_{1,\a,\b,-1}(z,x), &  \b\in(0,1),
\end{cases}      \\
\vartheta_{0,\a,\b}(z,x)&=\begin{cases}
y_{1,\a,\b,-1}(z,x), & \b\in(-1,0),\\
-2^{-\a-1} y_{2,\a,0,-1}(z,x), &  \b=0,\\
\b^{-1} 2^{-\a-1} y_{2,\a,\b,-1}(z,x), &  \b\in(0,1),
\end{cases}
\end{split}  \qquad \a\in\bbR.
\end{align}
We note that $M_{\eta,R,\a,\b}(\dott,-1)$ is a $2\times2$ matrix-valued Nevanlinna--Herglotz function as shown in \cite{GN22}.

Next, we discuss in detail two important cases of coupled boundary conditions: the general $\eta$-periodic and Krein--von Neumann extensions.

In the $\eta$-periodic case where $R=I_2$, $\eta\in[0,\pi),$ (see \cite{GN22a}) one has from \eqref{8.1} and \eqref{8.3},
\begin{align}
M_{\eta,I_2,\a,\b}(z,-1) & =\Small\begin{pmatrix}
\wti\varphi_{0,\a,\b}(z,1)  &
2^{-1}\big[ \wti\varphi^{\, \prime}_{0,\a,\b}(z,1) -\wti\vartheta_{0,\a,\b}(z,1)\big]\\
2^{-1}\big[\wti\varphi^{\, \prime}_{0,\a,\b}(z,1)- \wti\vartheta_{0,\a,\b}(z,1)\big]
  & 
-\wti\vartheta^{\, \prime}_{0,\a,\b}(z,1)
\end{pmatrix}   \no  \\
&\quad\; \times [\wti\varphi^{\, \prime}_{0,\a,\b}(z,1)+ \wti\vartheta_{0,\a,\b}(z,1)-2\cos(\eta)]^{-1},  \lb{8.4} \\
& \hspace*{6.5mm} \a,\b\in(-1,1),\, \eta\in[0,\pi),\, z\in\rho(T_{\eta,I_2,\a,\b}).  \no
\end{align}
The principal result on $\eta$-periodic Green's function induced $M$-functions of this section can then be written as follows:

\begin{theorem} \lb{t8.1}
Let $T_{\eta,I_2,\a,\b}$, be the self-adjoint operator associated with $\tau_{\a,\b}$ with $\eta$-periodic coupled boundary conditions, $\eta\in[0,\pi)$, 
$\a, \b \in (-1,1)$. Then the associated Green's function induced $M$-function is of the form 
\begin{align}\no
& M_{\eta,I_2,\a,\b}(z,-1)\\
& =\begin{cases}
\Small\begin{pmatrix}
-\b^{-1} 2^{-\a-1} \wti y_{2,\a,\b,-1}(z,1)  &
\begin{matrix}
2^{-1}\big[ -\b^{-1} 2^{-\a-1} \wti y^{\, \prime}_{2,\a,\b,-1}(z,1) \\
\hfill-\wti y_{1,\a,\b,-1}(z,1)\big]
\end{matrix}\\
\begin{matrix}
2^{-1}\big[ -\b^{-1} 2^{-\a-1} \wti y^{\, \prime}_{2,\a,\b,-1}(z,1) \\
\hfill-\wti y_{1,\a,\b,-1}(z,1)\big]
\end{matrix}
  & 
-\wti y^{\, \prime}_{1,\a,\b,-1}(z,1)
\end{pmatrix} \\
\ \ \times [-\b^{-1} 2^{-\a-1} \wti y^{\, \prime}_{2,\a,\b,-1}(z,1)+ \wti y_{1,\a,\b,-1}(z,1)-2\cos(\eta)]^{-1}, & \b\in(-1,0),\\[3mm]
\Small\begin{pmatrix}
\wti y_{1,\a,0,-1}(z,1) &
\begin{matrix}
2^{-1}\big[2^{-\a-1} \wti y_{2,\a,0,-1}(z,1) \\
\hfill +\wti y^{\, \prime}_{1,\a,0,-1}(z,1)\big]
\end{matrix}\\
\begin{matrix}
2^{-1}\big[2^{-\a-1} \wti y_{2,\a,0,-1}(z,1) \\
\hfill +\wti y^{\, \prime}_{1,\a,0,-1}(z,1)\big]
\end{matrix}
  & 
2^{-\a-1} \wti y^{\, \prime}_{2,\a,0,-1}(z,1)
\end{pmatrix} \\
\quad\ \times [\wti y^{\, \prime}_{1,\a,0,-1}(z,1) -2^{-\a-1} \wti y_{2,\a,0,-1}(z,1)-2\cos(\eta)]^{-1}, & \b=0,\\[3mm]
\Small\begin{pmatrix}
\wti y_{1,\a,\b,-1}(z,1)  &
\begin{matrix}
2^{-1}\big[-\b^{-1} 2^{-\a-1} \wti y_{2,\a,\b,-1}(z,1) \\
\hfill +\wti y^{\, \prime}_{1,\a,\b,-1}(z,1)\big]
\end{matrix}\\
\begin{matrix}
2^{-1}\big[-\b^{-1} 2^{-\a-1} \wti y_{2,\a,\b,-1}(z,1) \\
\hfill +\wti y^{\, \prime}_{1,\a,\b,-1}(z,1)\big]
\end{matrix}
  & 
-\b^{-1} 2^{-\a-1} \wti y^{\, \prime}_{2,\a,\b,-1}(z,1)
\end{pmatrix} \\
\quad\ \times [\wti y^{\, \prime}_{1,\a,\b,-1}(z,1)+ \b^{-1} 2^{-\a-1} \wti y_{2,\a,\b,-1}(z,1)-2\cos(\eta)]^{-1}, & \b\in(0,1);
\end{cases}   \no   \\
& \hspace*{5cm} \a\in(-1,1),\, \eta\in[0,\pi),\, z\in\rho(T_{\eta,I_2,\a,\b}).  \lb{8.5}
\end{align}
The poles of $M_{\eta,I_2,\a,\b}(\dott,-1)$ occur precisely at the eigenvalues 
of $T_{\eta,I_2,\a,\b}$. 
\end{theorem}

Hence, utilizing the boundary values given by \eqref{C.14}--\eqref{C.16} in \eqref{8.5} yields the $2\times2$ matrix-valued $M$-function $M_{\eta,I_2,\a,\b}(\dott,-1)$.

The following example provides an explicit illustration of Theorem \ref{t8.1} in the singular case $\a,\b\in(0,1)$:

\begin{example}[$\eta$-periodic extension]
Let $\a,\b\in(0,1)$. Substituting the boundary values \eqref{C.14} and \eqref{C.15} into \eqref{8.5} yields the $\eta$-periodic $M$-function
\begin{align}\no
& M_{\eta,I_2,\a,\b}(z,-1)\\
&\quad=
\Small\begin{pmatrix}
\dfrac{- 2^{1+\a+\b}\Gamma(1+\a)\Gamma(1+\b)}{\Gamma(a_{\a,\b,\sigma_{\a,\b}(z)})\Gamma(a_{\a,\b,-\sigma_{\a,\b}(z)})}  &
\begin{matrix}
2^{-1}\Bigg[ \dfrac{-\Gamma(1+\a)\Gamma(-\b)}{\Gamma(a_{\a,-\b,\sigma_{\a,\b}(z)})\Gamma(a_{\a,-\b,-\sigma_{\a,\b}(z)})} \\
\hfill+\dfrac{\Gamma(1+\b)\Gamma(-\a)}{\Gamma(a_{-\a,\b,\sigma_{\a,\b}(z)})\Gamma(a_{-\a,\b,-\sigma_{\a,\b}(z)})}\Bigg]
\end{matrix}\\
\begin{matrix}
2^{-1}\Bigg[ \dfrac{-\Gamma(1+\a)\Gamma(-\b)}{\Gamma(a_{\a,-\b,\sigma_{\a,\b}(z)})\Gamma(a_{\a,-\b,-\sigma_{\a,\b}(z)})} \\
\hfill+\dfrac{\Gamma(1+\b)\Gamma(-\a)}{\Gamma(a_{-\a,\b,\sigma_{\a,\b}(z)})\Gamma(a_{-\a,\b,-\sigma_{\a,\b}(z)})}\Bigg]
\end{matrix}
  & 
\dfrac{2^{-1-\a-\b}\Gamma(-\b)\Gamma(-\a)}{\Gamma(a_{-\a,-\b,\sigma_{\a,\b}(z)})\Gamma(a_{-\a,-\b,-\sigma_{\a,\b}(z)})}
\end{pmatrix}  \no \\
&\hspace{1cm} \times \Bigg[-2\cos(\eta)+\dfrac{\Gamma(1+\b)\Gamma(-\a)}{\Gamma(a_{-\a,\b,\sigma_{\a,\b}(z)})\Gamma(a_{-\a,\b,-\sigma_{\a,\b}(z)})}    \\
&\hspace{3.4cm}+ \dfrac{\Gamma(1+\a)\Gamma(-\b)}{\Gamma(a_{\a,-\b,\sigma_{\a,\b}(z)})\Gamma(a_{\a,-\b,-\sigma_{\a,\b}(z)})}\Bigg]^{-1},  \no\\
& \hspace*{3.45cm} \a,\b\in(0,1),\ \eta\in[0,\pi),\, z\in\rho(T_{\eta,I_2,\a,\b}), \no
\end{align}
where we abbreviated 
\begin{align}
a_{\mu,\nu,\pm\sigma} = [1 + \mu+\nu \pm \sigma]/2,\quad \mu, \nu, \sigma \in \bbC.    \lb{5.24}
\end{align}
\end{example}

As a final concrete example of coupled boundary conditions, we now consider the Krein--von Neumann extension following Example 4.3 in \cite{FGKLNS20}: 

\begin{example}[Krein--von Neumann extension]
For $\a,\b\in(-1,1)$, the following five cases are associated with a strictly positive minimal operator 
$T_{min,\a,\b}$ and we now provide the corresponding $R_{K,\a,\b}$ that yield the Krein--von Neumann extension $T_{0,R_K,\a,\b}$ of $T_{min,\a,\b}$:
\begin{align}
& T_{0,R_K,\a,\b} f = \tau_{\a,\b} f,    \\
& f \in \dom(T_{0,R_K,\a,\b})=\Bigg\{g\in\dom(T_{max,\a,\b}) \, \Bigg| \Bigg(\begin{matrix} \wti g(1) 
\\ {\wti g}^{\, \prime}(1) \end{matrix}\Bigg) = R_{K,\a,\b} \Bigg(\begin{matrix}
\wti g(-1) \\ {\wti g}^{\, \prime}(-1) \end{matrix}\Bigg) \Bigg\}, \no \\
& R_{K,\a,\b}=\begin{cases}
\begin{pmatrix}  1 & 2^{-\a-\b-1}\dfrac{\Gamma(-\a)\Gamma(-\b)}{\Gamma(-\a-\b)} \\
0 & 1
\end{pmatrix}, & \a,\b\in(-1,0),\\[7mm]
\begin{pmatrix}  -2^{-\a-\b-1}\dfrac{\Gamma(-\a)\Gamma(-\b)}{\Gamma(-\a-\b)} & 1 \\
-1 & 0
\end{pmatrix}, & \a\in(-1,0),\; \b\in(0,1), \\[7mm]
\begin{pmatrix}  0 & -1 \\
1 & 2^{-\a-\b-1}\dfrac{\Gamma(-\a)\Gamma(-\b)}{\Gamma(-\a-\b)}
\end{pmatrix}, & \a\in(0,1),\; \b\in(-1,0), \\[7mm]
\begin{pmatrix}  0 & -1 \\
1 & -2^{-\b-1}[\gamma_{E}+\psi(-\b)]
\end{pmatrix}, & \a=0,\; \b\in(-1,0) \\[7mm]
\begin{pmatrix}  2^{-\a-1}[\gamma_{E}+\psi(-\a)] & 1 \\
-1 & 0
\end{pmatrix}, & \a\in(-1,0),\; \b=0,
\end{cases}\lb{8.8}
\end{align}
where we interpret $1/\Gamma(0) = 0$.

For the remaining four cases given by all combinations of $\a=0,\ \b=0,\ \a\in(0,1),$ and $\b\in(0,1)$, we note that \cite[Theorem 3.5]{FGKLNS20} is not applicable as the minimal operator, $T_{\a,\b,min}$, is not strictly positive. This is easily seen by considering the spectrum explicitly given in Example \ref{e5.3} $(ii),\ (iii),\ (vii)$ and $(viii)$.
In particular, $0\in\sigma(T_{F,\a,\b}),\ \a,\b\in[0,1),$ and hence $T_{min, \a,\b} \geq 0$ is nonnegative, but not strictly positive when $\a,\b \in [0,1)$. We omit further details at this point. 

Finally, from \eqref{8.1} and \eqref{8.8}, one can compute the Green's function induced $M$-function for the Krein extension. We provide the explicit form for $\a,\b\in(-1,0)$ next, noting that the remaining cases can be treated analogously using the appropriate generalized boundary values at each endpoint corresponding to the ranges chosen for $\alpha$ and $\beta$. Abbreviating $a_{\mu,\nu,\pm\sigma} = [1 + \mu+\nu \pm \sigma]/2$, $\mu, \nu, \sigma \in \bbC$ as in \eqref{5.24}, one finds
\begin{align}
& M_{0,R_K,\a,\b}(z,-1)  \no\\
&\quad =
\Small \begin{pmatrix}
\begin{matrix}
-\b^{-1} 2^{-\a-1} \wti y_{2,\a,\b,-1}(z,1)\\
\quad
\end{matrix}   &
\begin{matrix}
2^{-1}\Bigg[ -\b^{-1} 2^{-\a-1} \wti y_{2,\a,\b,-1}^{\, \prime}(z,1) \hfill \\
+2^{-2\a-\b-2}\dfrac{\Gamma(-\a)\Gamma(-\b)}{\b\Gamma(-\a-\b)}  \wti y_{2,\a,\b,-1}(z,1) \\
\hfill- \wti y_{1,\a,\b,-1}(z,1)\Bigg]
\end{matrix}\\[4mm]
\begin{matrix}
2^{-1}\Bigg[-\b^{-1} 2^{-\a-1} \wti y_{2,\a,\b,-1}^{\, \prime}(z,1) \hfill \\
+2^{-2\a-\b-2}\dfrac{\Gamma(-\a)\Gamma(-\b)}{\b\Gamma(-\a-\b)} \wti y_{2,\a,\b,-1}(z,1) \\
\hfill- \wti y_{1,\a,\b,-1}(z,1)\Bigg]
\end{matrix}  & 
\begin{matrix}
2^{-\a-\b-1}\dfrac{\Gamma(-\a)\Gamma(-\b)}{\Gamma(-\a-\b)}\wti y_{1,\a,\b,-1}(z,1)\\
\hfill- \wti y_{1,\a,\b,-1}^{\, \prime}(z,1)\\
\quad
\end{matrix} 
\end{pmatrix}   \no\\
&\qquad\quad\times \Bigg[\b^{-1} 2^{-\a-1} \wti y_{2,\a,\b,-1}^{\, \prime}(z,1)- \wti y_{1,\a,\b,-1}(z,1) \no \\
&\hspace*{3cm}+2^{-\a-\b-1}\dfrac{\Gamma(-\a)\Gamma(-\b)}{\Gamma(-\a-\b)}\wti y_{1,\a,\b,-1}^{\, \prime}(z,1)+2\Bigg]^{-1} \no \\
&\quad =
\begin{pmatrix}
M_{1,1}   &
M_{1,2}\\
M_{2,1}  & 
M_{2,2}
\end{pmatrix}   \no\\
&\qquad\quad\times \Bigg[ \dfrac{\Gamma(1+\a)\Gamma(1-\b)}{\b\Gamma(a_{\a,-\b,\sigma_{\a,\b}(z)})\Gamma(a_{\a,-\b,-\sigma_{\a,\b}(z)})}- \dfrac{\Gamma(1+\b)\Gamma(-\a)}{\Gamma(a_{-\a,\b,\sigma_{\a,\b}(z)})\Gamma(a_{-\a,\b,-\sigma_{\a,\b}(z)})} \no \\
&\hspace*{4.5cm}+\dfrac{\Gamma(-\a)\Gamma(-\b)\Gamma(1+\a)\Gamma(1+\b)}{\Gamma(-\a-\b)\Gamma(a_{\a,\b,\sigma_{\a,\b}(z)})\Gamma(a_{\a,\b,-\sigma_{\a,\b}(z)})}+2\Bigg]^{-1}, \no \\
&\hspace*{6cm}\a,\b\in(-1,0),\ z\in\rho(T_{0,R_K,\a,\b}),
\end{align}
where
\begin{align}
\begin{split}
M_{1,1}&=\dfrac{2^{-\a-\b-1}\Gamma(-\b)\Gamma(-\a)}{\Gamma(a_{-\a,-\b,\sigma_{\a,\b}(z)})\Gamma(a_{-\a,-\b,-\sigma_{\a,\b}(z)})},\\[2mm]
M_{1,2}&=M_{2,1}=2^{-1}\Bigg[ \dfrac{\Gamma(1+\a)\Gamma(-\b)}{\Gamma(a_{\a,-\b,\sigma_{\a,\b}(z)})\Gamma(a_{\a,-\b,-\sigma_{\a,\b}(z)})} \\[1mm]
&\hspace*{2.3cm} -\dfrac{2^{-2\a-2\b-2}(\Gamma(-\a)\Gamma(-\b))^2}{\Gamma(-\a-\b)\Gamma(a_{-\a,-\b,\sigma_{\a,\b}(z)})\Gamma(a_{-\a,-\b,-\sigma_{\a,\b}(z)})} \\[1mm]
&\hspace*{2.3cm} - \dfrac{\Gamma(1+\b)\Gamma(-\a)}{\Gamma(a_{-\a,\b,\sigma_{\a,\b}(z)})\Gamma(a_{-\a,\b,-\sigma_{\a,\b}(z)})}\Bigg],\\[2mm]
M_{2,2}&=\dfrac{2^{-\a-\b-1}\Gamma(-\a)^2\Gamma(-\b)\Gamma(1+\b)}{\Gamma(-\a-\b)\Gamma(a_{-\a,\b,\sigma_{\a,\b}(z)})\Gamma(a_{-\a,\b,-\sigma_{\a,\b}(z)})}
\\[1mm]
&\quad\, - \dfrac{2^{1+\a+\b}\Gamma(1+\a)\Gamma(1+\b)}{\Gamma(a_{\a,\b,\sigma_{\a,\b}(z)})\Gamma(a_{\a,\b,-\sigma_{\a,\b}(z)})},  \\
& \hspace*{6.5mm} \a,\b\in(-1,0),\ z\in\rho(T_{0,R_K,\a,\b}). 
\end{split}
\end{align}
\end{example}

\section{Precisely One Interval Endpoint in the Limit Point Case} \lb{s6}

In this section we determine the Weyl--Titchmarsh--Kodaira $m$-function in all situations where precisely one interval endpoint is in the limit point case. We will focus on the case when $\a \in (-\infty, -1]$ or 
$\a \in [1,\infty)$, so that the right endpoint $x=1$ represents the limit point case. The converse situation can be obtained by reflection with respect to the origin (i.e., considering the transform $(-1,1) \ni x \mapsto -x \in (-1,1)$).  

We start with the case $\a \in [1,\infty)$, $\b\in(-1,1)$ so that the left endpoint $-1$ is regular or limit circle and the right endpoint $1$ is in the limit point case. In accordance with Section \ref{s3}, we single out the left endpoint $x=-1$ and determine the Weyl--Titchmarsh--Kodaira solution and $m$-function via
\begin{align}
\begin{split} 
\psi_{\g,\a,\b} (z,\dott) = \vartheta_{\g,\a,\b}(z,\dott) + m_{\g,\a,\b} (z) \varphi_{\g,\a,\b}(z,\dott) 
\in L^2((c,1); r_{\a,\b} dx),&  \\
z \in \rho(T_{\g,\a,\b}), \;  \a \in [1,\infty), \; \b\in(-1,1), \; c \in (-1,1).&
\end{split}
\end{align}
To simplify matters we next focus on the Friedrichs extension, $\g=0$, noting that for $\g\in[0,\pi)$, the general $m_{\g,\a,\b}$-function can be found by employing the linear fractional transformation \eqref{3.23}, choosing 
$\gamma_1 =0$ and $\gamma_2 = \gamma$, fixing the limit point endpoint $x=1$ (dropping the parameter 
$\delta \in [0,\pi)$ as there is no boundary condition to be imposed at the limit point endpoint $x=1$). We provide a summary at the end of this section in Theorem \ref{t6.1}.\\[1mm]
\noindent 
{\boldmath $\mathbf{(I)}$ {\bf The Case $\a \in [1,\infty)$ and $\b\in(-1,0)$:}} \\[1mm] 
Consider the Friedrichs extension $\gamma = 0$, $T_{F,\a,\b}= T_{0,\a,\b}$, $\a \in [1,\infty)$, $\b\in(-1,0)$. Then 
\begin{align}
\varphi_{0,\a,\b}(z,x)=-2^{-\a-1}\b^{-1}y_{2,\a,\b,-1}(z,x),\quad \vartheta_{0,\a,\b}(z,x)=y_{1,\a,\b,-1}(z,x), 
\end{align}
and the requirement $\psi_{0,\a,\b} (z,\dott)\in L^2\big((c,1); r_{\a,\b} dx\big)$, $c \in (-1,1)$, 
$z \in \bbC \backslash \bbR$, can be recast as
\begin{equation}
\big[y_{1,\a,\b,-1}(z,\dott)-2^{-\a-1}\b^{-1}y_{2,\a,\b,-1}(z,\dott) m_{0,\a,\b}(z)\big] 
\in L^2\big((c,1); r_{\a,\b} dx\big). 
\lb{6.3}
\end{equation}
Using the limiting behavior of $y_{j,\a,\b,-1}(z,x),\ j=1,2$, near $x=1$ in \eqref{C.1}, \eqref{C.4}, \eqref{C.5}, \eqref{C.8}, and observing that only the part $(1-x)^{\a}$ in $r_{\a,\b}(x)$ matters for $x$ in a neighborhood of $1$, one concludes that the square integrability in \eqref{6.3} reduces to integrability of 
\begin{align}
& (1-x)^{-\a} \bigg[\f{2^\a\Gamma(1+\b)\Gamma(\a)}{\Gamma([1+\a+\b+\sigma_{\a,\b}(z)]/2)
\Gamma([1+\a+\b-\sigma_{\a,\b}(z)]/2)}    \\
& \hspace*{1.7cm}  -\f{2^{-\b-1}\b^{-1}\Gamma(1-\b)\Gamma(\a)}{\Gamma([1+\a-\b+\sigma_{\a,\b}(z)]/2)
\Gamma([1+\a-\b-\sigma_{\a,\b}(z)]/2)} m_{0,\a,\b}(z)\bigg]^2,  \no 
\end{align}
near $x=1$ which, in turn, happens if and only if the expression in brackets vanishes. Thus, 
\begin{align}
& m_{0,\a\,\b}(z) = 2^{1+\a+\b}\b \f{\Gamma(1+\b)}{\Gamma(1-\b)}    \no \\
& \hspace*{1.8cm}  \times \f{\Gamma([1+\a-\b+\sigma_{\a,\b}(z)]/2)
\Gamma([1+\a-\b-\sigma_{\a,\b}(z)]/2)}{
\Gamma([1+\a+\b+\sigma_{\a,\b}(z)]/2)\Gamma([1+\a+\b-\sigma_{\a,\b}(z)]/2)},    \lb{6.5} \\
&\hspace*{4.7cm}  z\in\rho(T_{F,\a,\b}), \; \a \in [1,\infty), \; \b\in(-1,0),   \no \\
&\sigma(T_{F,\a,\b}) = \{(n-\b)(n+1+\a)\}_{n\in\bbN_0}, \quad \a\in [1,\infty), \; \b\in(-1,0). \no
\end{align}
\noindent 
{\boldmath $\mathbf{(II)}$ {\bf The Case $\a \in [1,\infty)$ and $\b=0$\,:}} \\[1mm] 
Consider the Friedrichs extension $\g = 0$, $T_{F,\a,0} = T_{0,\a,0}$, $\a \in [1,\infty)$, $\b=0$. Then 
\begin{align}
\varphi_{0,\a,0}(z,x)=y_{1,\a,0,-1}(z,x),\quad \vartheta_{0,\a,0}(z,x)=-2^{-\a-1}y_{2,\a,0,-1}(z,x), 
\end{align}
and the requirement $\psi_{0,\a,0} (z,\dott)\in L^2\big((c,1); r_{\a,\b} dx\big)$, $c \in (-1,1)$, 
$z \in \bbC \backslash \bbR$, can again be recast as
\begin{equation}
\big[-2^{-\a-1}y_{2,\a,0,-1}(z,\dott)+y_{1,\a,0,-1}(z,\dott) m_{0,\a,0}(z)\big] 
\in L^2\big((c,1); r_{\a,\b} dx\big).     \lb{6.7}
\end{equation}
Using the limiting behavior $y_{j,\a,\b,-1}(z,x),\ j=1,2$, near $x=1$ in \eqref{C.1}, \eqref{C.4}, \eqref{C.9}, and \eqref{C.12}, one concludes that the square integrability in \eqref{6.7} reduces to integrability of  
\begin{align} 
\begin{split} 
& (1-x)^{-\a}\bigg[\f{2\gamma_E+\psi([1+\a+\sigma_{\a,0}(z)]/2)+\psi([1+\a-\sigma_{\a,0}(z)]/2)}
{2 \Gamma([1+\a+\sigma_{\a,0}(z)]/2)\Gamma([1+\a-\sigma_{\a,0}(z)]/2)}\Gamma(\a)    \\
& \hspace*{1.6cm} + \f{2^\a\Gamma(\a)}{\Gamma([1+\a+\sigma_{\a,0}(z)]/2)
\Gamma([1+\a-\sigma_{\a,0}(z)]/2)}m_{0,\a,0}(z)\bigg]^2,
\end{split} 
\end{align}
near $x=1$, which happens if and only if the expression in brackets vanishes. Thus, 
\begin{align}
& m_{0,\a,0}(z)=-2^{-\a-1} \{2\gamma_E+\psi([1+\a+\sigma_{\a,0}(z)]/2) 
+\psi([1+\a-\sigma_{\a,0}(z)]/2)\},  \no \\
& \hspace*{6cm}  z\in\rho(T_{F,\a,0}), \; \a \in [1,\infty), \; \b=0,      \lb{6.9}\\
&\sigma(T_{F,\a,0}) = \{n(n+1+\a)\}_{n\in\bbN_0},  \quad \a \in [1,\infty), \; \b = 0. \no
\end{align}

\noindent 
{\boldmath $\mathbf{(III)}$ {\bf The Case $\a \in [1,\infty)$ and $\b\in(0,1)$:}} \\[1mm] 
Finally, consider the Friedrichs extension $\g = 0$, $T_{F,\a,\b} = T_{0,\a,\b}$, $\a \in [1,\infty)$, $\b\in(0,1)$. Then 
\begin{align}
\varphi_{0,\a,\b}(z,x)=y_{1,\a,\b,-1}(z,x),\quad \vartheta_{0,\a,\b}(z,x)=2^{-\a-1}\b^{-1}y_{2,\a,\b,-1}(z,x), 
\end{align}
and the requirement $\psi_{0,\a,\b} (z,\dott)\in L^2\big((c,1); r_{\a,\b} dx\big)$, $c \in (-1,1)$, 
$z \in \bbC \backslash \bbR$, can again be recast as
\begin{align}
\big[2^{-\a-1}\b^{-1}y_{2,\a,\b,-1}(z,\dott)+y_{1,\a,\b,-1}(z,\dott) m_{0,\a,\b}(z)\big] \in L^2\big((c,1); r_{\a,\b} dx\big). 
\lb{6.11} 
\end{align}
Once again, using the limiting behavior of $y_{j,\a,\b,-1}(z,x),\ j=1,2$, near $x=1$ in \eqref{C.1}, \eqref{C.4}, \eqref{C.5}, \eqref{C.8}, one concludes that the square integrability in \eqref{6.11} reduces to integrability of  
\begin{align}
& (1-x)^{-\a} \bigg[\f{2^{-\b-1}\b^{-1}\Gamma(1-\b)\Gamma(\a)}{\Gamma([1+\a-\b+\sigma_{\a,\b}(z)]/2)
\Gamma([1+\a-\b-\sigma_{\a,\b}(z)]/2)}    \\
& \hspace*{1.6cm} +\f{2^\a\Gamma(1+\b)\Gamma(\a)}{\Gamma([1+\a+\b+\sigma_{\a,\b}(z)]/2)
\Gamma([1+\a+\b-\sigma_{\a,\b}(z)]/2)} m_{0,\a,\b}(z)\bigg]^2,    \no
\end{align}
near $x=1$, which happens if and only if the expression in brackets is zero. Thus, 
\begin{align}
\no m_{0,\a\,\b}(z)&= \b^{-1} 2^{-1 - \a - \b}\f{-\Gamma(1-\b)}{\Gamma(1+\b)}    \no \\
& \quad \times \f{\Gamma([1+\a+\b+\sigma_{\a,\b}(z)]/2)
\Gamma([1+\a+\b-\sigma_{\a,\b}(z)]/2)}{\Gamma([1+\a-\b+\sigma_{\a,\b}(z)]/2)
\Gamma((1+\a-\b-\sigma_{\a,\b}(z))/2)},     \lb{6.13} \\
&\hspace{3.5cm} z\in\rho(T_{F,\a,\b}), \; \a \in [1,\infty), \; \b\in(0,1),     \no \\
\sigma(T_{F,\a,\b}) &= \{n(n+1+\a+\b)\}_{n\in\bbN_0},   \quad \a \in [1,\infty), \; \b\in(0,1). \no
\end{align}

\medskip

Next, we turn to the remaining case $\a \in (-\infty, -1]$, $\b\in(-1,1)$, and determine the Weyl--Titchmarsh--Kodaira solution and $m$-function in accordance with Section \ref{s3} via
\begin{align}
\begin{split}
\psi_{\g,\a,\b} (z,\dott) = \vartheta_{\g,\a,\b}(z,\dott) + m_{\g,\a,\b} (z) \varphi_{\g,\a,\b}(z,\dott)
\in L^2((c,1); r_{\a,\b} dx),&  \\
z \in \rho(T_{\g,\a,\b}), \; \a \in (-\infty. -1], \; \b\in(-1,1), \; c \in (-1,1).&
\end{split}
\end{align}
To simplify matters we once more focus on the Friedrichs extension, $\g=0$ since for 
$\g\in[0,\pi)$ the general $m_{\g,\a,\b}$-function can again be determined by employing the linear fractional transformation \eqref{3.23}. \\[1mm] 
\noindent 
{\boldmath $\mathbf{(IV)}$ {\bf The Case $\a \in (-\infty,-1]$ and $\b\in(-1,0)$:}} \\[1mm] 
Consider the Friedrichs extension $\g=0$, $T_{F,\a,\b} = T_{0,\a,\b}$, $\a \in (-\infty,-1]$, $\b\in(-1,0)$. Then  
\begin{align}
\varphi_{0,\a,\b}(z,x)=-2^{-\a-1}\b^{-1}y_{2,\a,\b,-1}(z,x),\quad \vartheta_{0,\a,\b}(z,x)=y_{1,\a,\b,-1}(z,x),
\end{align}
and the requirement $\psi_{0,\a,\b} (z,\dott)\in L^2\big((c,1); r_{\a,\b} dx\big)$, $c \in (-1,1)$, 
$z \in \bbC \backslash \bbR$, can be recast as
\begin{align}
\big[y_{1,\a,\b,-1}(z,\dott)-2^{-\a-1}\b^{-1}y_{2,\a,\b,-1}(z,\dott) m_{0,\a,\b}(z)\big] 
\in L^2\big((c,1); r_{\a,\b} dx\big).    \lb{6.16}
\end{align}
Using again the limiting behavior of $y_{j,\a,\b,-1}(z,x),\ j=1,2$, near $x=1$ in \eqref{C.1}, \eqref{C.2}, 
\eqref{C.5}, \eqref{C.6}, 
one concludes that the square integrability in \eqref{6.16} reduces to the integrability of 
\begin{align}
& (1-x)^{\a} \bigg[\f{\Gamma(1+\b)\Gamma(-\a)}{\Gamma([1+\b-\a+\sigma_{\a,\b}(z)]/2)
\Gamma([1+\b-\a-\sigma_{\a,\b}(z)]/2)}    \\
& \hspace*{1.4cm}
-\f{2^{-\a-\b-1}\b^{-1}\Gamma(1-\b)\Gamma(-\a)}{\Gamma([1-\a-\b+\sigma_{\a,\b}(z)]/2)
\Gamma([1-\a-\b-\sigma_{\a,\b}(z)]/2)} m_{0,\a,\b}(z)\bigg]^2,     \no 
\end{align}
near $x=1$, which happens if and only if the expression in brackets is zero. Thus, 
\begin{align}
\no m_{0,\a\,\b}(z)&=2^{1+\a+\b}\b \f{\Gamma(1+\b)}{\Gamma(1-\b)}    \no \\
& \quad \times \f{\Gamma([1-\a-\b+\sigma_{\a,\b}(z)]/2)\Gamma([1-\a-\b-\sigma_{\a,\b}(z)]/2)}
{\Gamma([1+\b-\a+\sigma_{\a,\b}(z)]/2)\Gamma([1+\b-\a-\sigma_{\a,\b}(z)]/2)},    \lb{6.18} \\
&\hspace{2.55cm} \ z\in\rho(T_{F,\a,\b}), \; \a \in (-\infty, -1], \; \b\in(-1,0), \no \\
\sigma(T_{F,\a,\b}) &= \{(n-\a-\b)(n+1)\}_{n\in\bbN_0},  \quad \a \in (-\infty, -1], \; \b\in(-1,0). \no
\end{align}
\noindent 
{\boldmath $\mathbf{(V)}$ {\bf The Case $\a \in (-\infty, -1]$ and $\b=0$\,:}} \\[1mm] 
Consider the Friedrichs extension $\g = 0$, $T_{F,\a,0} = T_{0,\a,0}$, $\a \in (-\infty, -1]$, $\b=0$. Then 
\begin{align}
\varphi_{0,\a,0}(z,x)=y_{1,\a,0,-1}(z,x),\quad \vartheta_{0,\a,0}(z,x)=-2^{-\a-1}y_{2,\a,0,-1}(z,x), 
\end{align}
and the requirement $\psi_{0,\a,0} (z,\dott)\in L^2\big((c,1); r_{\a,\b} dx\big)$, $c \in (-1,1)$, 
$z \in \bbC \backslash \bbR$, can be recast as
\begin{align}
\big[-2^{-\a-1}y_{2,\a,0,-1}(z,\dott)+y_{1,\a,0,-1}(z,\dott)m_{0,\a,0}(z)\big] 
\in L^2\big((c,1); r_{\a,\b} dx\big).    \lb{6.20}
\end{align}
Using again the limiting behavior of $y_{j,\a,\b,-1}(z,x),\ j=1,2$, near $x=1$ in \eqref{C.1}, \eqref{C.2}, \eqref{C.9}, and \eqref{C.10}, one concludes that the square integrability in \eqref{6.20} reduces to 
integrability of 
\begin{align}
\begin{split}
& (1-x)^{\a}\bigg[\f{(2\gamma_E+\psi([1-\a+\sigma_{\a,0}(z)]/2)+\psi([1-\a-\sigma_{\a,0}(z)]/2))
\Gamma(-\a)}{2^{\a+1}\Gamma([1-\a+\sigma_{\a,0}(z)]/2)\Gamma([1-\a-\sigma_{\a,0}(z)]/2)}   \\
& \hspace*{1.4cm} +\f{\Gamma(-\a)}{\Gamma([1-\a+\sigma_{\a,0}(z)]/2)\Gamma([1-\a-\sigma_{\a,0}(z)]/2)} 
m_{0,\a,0}(z)\bigg]^2,
\end{split}
\end{align}
near $x=1$, which happens if and only if the expression in brackets is zero. Thus, 
\begin{align}
& m_{0,\a,0}(z)=-2^{-\a-1} \{2\gamma_E+\psi([1-\a+\sigma_{\a,0}(z)]/2)+\psi([1-\a-\sigma_{\a,0}(z)]/2)\}, 
\no \\
& \hspace*{5.5cm} z\in\rho(T_{F,\a,0}), \; \a \in (-\infty,-1], \; \b = 0,      \lb{6.22} \\
&\sigma(T_{F,\a,0}) = \{(n-\a)(n+1)\}_{n\in\bbN_0}, \quad \a \in (-\infty,-1], \; \b = 0.\no
\end{align}
\noindent 
{\boldmath $\mathbf{(VI)}$ {\bf The Case $\a \in (-\infty,-1]$ and $\b\in(0,1)$:}} \\[1mm] 
Finally, consider the Friedrichs extension $\g = 0$, $T_{F,\a,\b} = T_{0,\a,\b}$, $\a \in (-\infty,-1]$, $\b\in(0,1)$. Then 
\begin{align}
\varphi_{0,\a,\b}(z,x)=y_{1,\a,\b,-1}(z,x),\quad \vartheta_{0,\a,\b}(z,x)=2^{-\a-1}\b^{-1}y_{2,\a,\b,-1}(z,x),
\end{align}
and the requirement $\psi_{0,\a,\b} (z,\dott)\in L^2\big((c,1); r_{\a,\b} dx\big)$, $c \in (-1,1)$, 
$z \in \bbC \backslash \bbR$, can be recast as
\begin{align}
\big[2^{-\a-1}\b^{-1}y_{2,\a,\b,-1}(z,\dott)+y_{1,\a,\b,-1}(z,\dott)m_{0,\a,\b}(z)\big] 
\in L^2\big((c,1); r_{\a,\b} dx\big).    \lb{6.24} 
\end{align}
Once again using the limiting behavior near $x=1$ of $y_{j,\a,\b,-1}(z,x),\ j=1,2,$ found in \eqref{C.1}, 
\eqref{C.2}, \eqref{C.5}, \eqref{C.6} one concludes that the square integrability in \eqref{6.24} reduces to integrability of 
\begin{align}
\begin{split}
& (1-x)^{\a} \bigg[\f{2^{-\a-\b-1}\b^{-1}\Gamma(1-\b)\Gamma(-\a)}
{\Gamma([1-\a-\b+\sigma_{\a,\b}(z)]/2)\Gamma([1-\a-\b-\sigma_{\a,\b}(z)]/2)}\\
& \quad +\f{\Gamma(1+\b)\Gamma(-\a)}{\Gamma([1+\b-\a+\sigma_{\a,\b}(z)]/2)
\Gamma([1+\b-\a-\sigma_{\a,\b}(z)]/2)} m_{0,\a,\b}(z)\bigg]^2,
\end{split}
\end{align}
near $x=1$, which happens if and only if the expression in brackets is zero. Thus, 
\begin{align}
\no m_{0,\a\,\b}(z)&= - \b^{-1} 2^{-1 - \a - \b} \f{\Gamma(1-\b)}{\Gamma(1+\b)}     \no \\
& \quad \times \f{\Gamma([1+\b-\a+\sigma_{\a,\b}(z)]/2)
\Gamma([1+\b-\a-\sigma_{\a,\b}(z)]/2)}{\Gamma([1-\a-\b+\sigma_{\a,\b}(z)]/2)
\Gamma([1-\a-\b-\sigma_{\a,\b}(z)]/2)},       \lb{6.26} \\
&\hspace{2.95cm}  z\in\rho(T_{F,\a,\b}), \; \a \in (-\infty,-1], \; \b\in(0,1),   \no \\
\sigma(T_{F,\a,\b}) &= \{(n-\a)(n+1+\b)\}_{n\in\bbN_0},  \quad \a \in (-\infty, -1], \; \b\in(0,1). \no
\end{align}

We briefly summarize the principal result of this section as follows:

\begin{theorem}\lb{t6.1}
Let $\g \in [0,\pi)$, $\a \in (-\infty,-1] \cup [1,\infty)$ and $\b \in (-1,1)$. Then the Weyl--Titchmarsh--Kodaira 
$m$-function $m_{\g,\a,\b}$ associated with $T_{\g,\a,\b}$ is given by
\begin{align}
m_{\g,\a,\b}(z) =\frac{-\sin(\g) + \cos(\gamma) m_{0,\a,\b}(z)}
{\cos(\gamma) +\sin(\gamma) m_{0,\a,\b}(z)},  \quad 
z \in \rho(T_{\g,\a,\b}), 
\end{align}
where $m_{0,\a,\b}$ is given by \eqref{6.5}, \eqref{6.9}, \eqref{6.13}, \eqref{6.18}, \eqref{6.22}, and \eqref{6.26}, respectively. The $($necessarily simple\,$)$ poles of $m_{\g,\a,\b}$ occur precisely at the $($necessarily simple\,$)$ eigenvalues of $T_{\g,\a,\b}$. 
\end{theorem}

\section{Both Endpoints in the Limit Point Case} \lb{s7} 

In this section we treat the remaining case where $\tau_{\a,\b}$ is in the limit point case at both interval endpoints $\pm 1$, that is, where $\a, \b \in (-\infty,-1] \cup [1,\infty)$. 

We start by recalling a number of facts implied by the fact that $\tau_{\a,\b}$ is in the limit point case at $\pm 1$:
\begin{align}
\psi_{\a,\b}(z,\dott) &= C_{\a,\b}(z,x_0) \psi_{0,+,\a,\b} (z,\dott,x_0)     \lb{7.1} \\
&= \vartheta_{\a,\b}(z,\dott) + m_{\a,\b}(z) \varphi_{\a,\b}(z,\dott) \in 
L^2((c,1); r_{\a,\b} dx), \quad c \in (-1,1),    \no 
\end{align}
where 
\begin{equation}
\varphi_{\a,\b}(z,\dott) = C_{\a,\b}(z) \psi_{0,-,\a,\b}(z,\dott) \in L^2((-1,d); r_{\a,\b} dx), \quad d \in (-1,1).   
\end{equation} 
Thus, \eqref{C.1}--\eqref{C.12}, \eqref{C.17}--\eqref{C.20} imply 
\begin{align}
& \varphi_{\a,\b}(z,x) = c_{\a,\b}(z) \begin{cases} y_{1,\a,\b,-1}(z,x), & \b \in [1,\infty), \\
y_{2,\a,\b,-1}(z,x), & \b \in (-\infty, -1], 
\end{cases}  \quad x \in (-1,1), \\
& \hspace*{1.3cm} \underset{x \downarrow -1}{=} c_{\a,\b}(z) \begin{cases} 1 + \Oh(1+x), & \b \in [1,\infty), \\
(1+x)^{-\b} [1 + \Oh(1+x)], & \b \in (-\infty, -1], 
\end{cases} \\
& \psi_{0,+,\a,\b} (z,\dott,x_0) = c_{\a,\b}(z,x_0) \begin{cases} y_{1,\a,\b,1}(z,x), & \a \in [1,\infty), \\
y_{2,\a,\b,1}(z,x), & \a \in (-\infty, -1], 
\end{cases}
\end{align}
with $C_{\a,\b}(\dott,x_0), C_{\a,\b}(\dott) , c_{\a,\b}(\dott), c_{\a,\b}(\dott,x_0) \in \bbC\backslash\{0\}$, appropriate coefficients. 

Since $\vartheta_{\a,\b}(z,\dott)$ must satisfy $W(\vartheta_{\a,\b}(z,\dott),\varphi_{\a,\b}(z,\dott)) = 1$, we conveniently choose $\vartheta_{\a,\b}(z,\dott)$ (cf.\ Remark \ref{r3.2}) such that 
\begin{equation}
\vartheta_{\a,\b}(z,x) = d_{\a,\b}(z) \begin{cases} (1+x)^{-\b} [1+\Oh(1+x)], & \b \in [1,\infty), \\
1+\Oh(1+x), & \b \in (-\infty, -1], 
\end{cases}  
\end{equation}
with $d_{\a,\b}(\dott) \in \bbC\backslash\{0\}$ an appropriate coefficient. 

Our strategy to find (a multiple of) $m_{\a,\b}(\dott)$ is now rather simple: We will search for ``the part of 
$\varphi_{\a,\b}(z,\dott)$ in $\psi_{0,+,\a,\b} (z,\dott,x_0)$,'' as the former is multiplied by $m_{\a,\b}(\dott)$ according 
to \eqref{7.1}. In this context we emphasize once more that since $\tau_{\a,\b}$ is in the limit point case at $\pm 1$, $m_{\a,\b}(\dott)$ is nonunique and can only be characterized up to inessential $z$-dependent multiples as well as an additive entire term as discussed in Section \ref{s3} (see again Remark \ref{r3.2}). Since 
$\a, \b \in (-\infty,-1] \cup [1,\infty)$, this requires to distinguish the following eight cases. We once again provide a summary at the end of this section in Theorem \ref{t7.1} and once more employ the abbreviation \eqref{5.24}, 
$a_{\mu,\nu,\pm\sigma} = [1 + \mu+\nu \pm \sigma]/2$, $\mu, \nu, \sigma \in \bbC$.  \\[1mm] 
\noindent 
{\boldmath $\mathbf{(I)}$ {\bf The case $\a\in (-\infty, -1]$, $\b \in (-\infty, -1] \backslash (-\N)$:}} 
\\[1mm] 
Then an application of \eqref{C.17}--\eqref{C.20} yields 
\begin{align}
& \psi_{0,+,\a,\b} (z,\dott,x_0) \underset{x \downarrow -1}{\sim} y_{2,\a,\b,1}(z,\dott)   \no \\
& \quad \underset{x \downarrow -1}{\sim} \f{2^{-\a} \Gamma(1-\a) \Gamma(-\b)}
{\Gamma(a_{-\a,-\b,\sigma_{\a,\b}(z)}) \Gamma(a_{-\a,-\b,-\sigma_{\a,\b}(z)})}\bigg\{y_{1,\a,\b,-1}(z,\dott)     \\
& \hspace*{1.2cm} + \f{2^{\b} [\Gamma(\b)/\Gamma(-\b)] \Gamma(a_{-\a,-\b,\sigma_{\a,\b}(z)}) 
\Gamma(a_{-\a,-\b,-\sigma_{\a,\b}(z)})}{\Gamma(a_{-\a,\b,\sigma_{\a,\b}(z)})
\Gamma(a_{-\a,\b,-\sigma_{\a,\b}(z)})} \underbrace{y_{2,\a,\b,-1}(z,\dott)}_{\sim \varphi_{\a,\b}(z,\dott)}\bigg\},    
\no \\
& \hspace*{6.6cm} \a \in (-\infty,-1], \, \b \in (-\infty,-1] \backslash (- \N),    \no
\end{align}
implying 
\begin{align}
m_{\a,\b}(z) \sim \f{\Gamma([1-\a-\b+\sigma_{\a,\b}(z)]/2) \Gamma([1-\a-\b-\sigma_{\a,\b}(z)]/2)}
{\Gamma([1-\a+\b+\sigma_{\a,\b}(z)]/2) \Gamma([1-\a+\b-\sigma_{\a,\b}(z)]/2)},&      \\
\a \in (-\infty,-1], \; \b \in (-\infty,-1] \backslash (- \N).&     \no 
\end{align}
\noindent 
{\boldmath $\mathbf{(II)}$ {\bf The case $\a \in [1,\infty)$, $\b \in (-\infty, -1] \backslash (-\N)$:}} 
\\[1mm] 
Then an application of \eqref{C.17}--\eqref{C.20} yields 
\begin{align}
& \psi_{0,+,\a,\b} (z,\dott,x_0) \underset{x \downarrow -1}{\sim} y_{1,\a,\b,1}(z,\dott)   \no \\
& \quad \underset{x \downarrow -1}{\sim} \f{\Gamma(1+\a) \Gamma(-\b)}
{\Gamma(a_{\a,-\b,\sigma_{\a,\b}(z)}) \Gamma(a_{\a,-\b,-\sigma_{\a,\b}(z)})}\bigg\{y_{1,\a,\b,-1}(z,\dott)     \\
& \hspace*{1.1cm} \, + \f{2^{\b} [\Gamma(\b)/\Gamma(-\b)] \Gamma(a_{\a,-\b,\sigma_{\a,\b}(z)}) 
\Gamma(a_{\a,-\b,-\sigma_{\a,\b}(z)})}{\Gamma(a_{\a,\b,\sigma_{\a,\b}(z)})
\Gamma(a_{\a,\b,-\sigma_{\a,\b}(z)})} \underbrace{y_{2,\a,\b,-1}(z,\dott)}_{\sim \varphi_{\a,\b}(z,\dott)}\bigg\},    
\no \\
& \hspace*{6.6cm} \a \in [1,\infty), \; \b \in (-\infty,-1] \backslash (- \N),    \no
\end{align}
implying 
\begin{align}
m_{\a,\b}(z) \sim \f{\Gamma([1+\a-\b+\sigma_{\a,\b}(z)]/2) \Gamma([1+\a-\b-\sigma_{\a,\b}(z)]/2)}
{\Gamma([1+\a+\b+\sigma_{\a,\b}(z)]/2) \Gamma([1+\a+\b-\sigma_{\a,\b}(z)]/2)},&      \\
\a \in [1,\infty), \; \b \in (-\infty,-1] \backslash (- \N).&     \no 
\end{align} 
\noindent 
{\boldmath $\mathbf{(III)}$ {\bf The case $\a \in (-\infty,-1]$, $\b \in [1,\infty) \backslash \N$:}} 
\\[1mm] 
Then an application of \eqref{C.17}--\eqref{C.20} yields 
\begin{align}
& \psi_{0,+,\a,\b} (z,\dott,x_0) \underset{x \downarrow -1}{\sim} y_{2,\a,\b,1}(z,\dott)   \no \\
& \quad \underset{x \downarrow -1}{\sim} \f{2^{-\a} \Gamma(1-\a) \Gamma(\b)}
{\Gamma(a_{-\a,+\b,\sigma_{\a,\b}(z)}) \Gamma(a_{-\a,+\b,-\sigma_{\a,\b}(z)})}\bigg\{
2^{\b} y_{2,\a,\b,-1}(z,\dott)     \\
& \hspace*{1.15cm} + \f{[\Gamma(-\b)/\Gamma(\b)] \Gamma(a_{-\a,\b,\sigma_{\a,\b}(z)}) 
\Gamma(a_{-\a,\b,-\sigma_{\a,\b}(z)})}{\Gamma(a_{-\a,-\b,\sigma_{\a,\b}(z)})
\Gamma(a_{-\a,-\b,-\sigma_{\a,\b}(z)})} \underbrace{y_{1,\a,\b,-1}(z,\dott)}_{\sim \varphi_{\a,\b}(z,\dott)}\bigg\},    
\no \\
& \hspace*{6.8cm} \a \in (-\infty,-1], \; \b \in [1,\infty) \backslash \N,    \no
\end{align}
implying 
\begin{align}
m_{\a,\b}(z) \sim \f{\Gamma([1-\a+\b+\sigma_{\a,\b}(z)]/2) \Gamma([1-\a+\b-\sigma_{\a,\b}(z)]/2)}
{\Gamma([1-\a-\b+\sigma_{\a,\b}(z)]/2) \Gamma([1-\a-\b-\sigma_{\a,\b}(z)]/2)},&      \\
\a \in (-\infty,-1], \; \b \in [1,\infty) \backslash \N.&     \no 
\end{align}
\noindent 
{\boldmath $\mathbf{(IV)}$ {\bf The case $\a \in [1,\infty)$, $\b \in [1,\infty) \backslash \N$:}}
\\[1mm] 
Then an application of \eqref{C.17}--\eqref{C.20} yields 
\begin{align}
& \psi_{0,+,\a,\b} (z,\dott,x_0) \underset{x \downarrow -1}{\sim} y_{1,\a,\b,1}(z,\dott)   \no \\
& \quad \underset{x \downarrow -1}{\sim} \f{\Gamma(1+\a) \Gamma(\b)}
{\Gamma(a_{\a,\b,\sigma_{\a,\b}(z)}) \Gamma(a_{\a,\b,-\sigma_{\a,\b}(z)})}\bigg\{
2^{\b} y_{2,\a,\b,-1}(z,\dott)     \\
& \hspace*{1.2cm} + \f{[\Gamma(-\b)/\Gamma(\b)] \Gamma(a_{\a,\b,\sigma_{\a,\b}(z)}) 
\Gamma(a_{\a,\b,-\sigma_{\a,\b}(z)})}{\Gamma(a_{\a,-\b,\sigma_{\a,\b}(z)})
\Gamma(a_{\a,-\b,-\sigma_{\a,\b}(z)})} \underbrace{y_{1,\a,\b,-1}(z,\dott)}_{\sim \varphi_{\a,\b}(z,\dott)}\bigg\},    
\no \\
& \hspace*{7cm} \a \in [1,\infty), \; \b \in [1,\infty) \backslash \N,    \no
\end{align}
implying 
\begin{align}
m_{\a,\b}(z) \sim \f{\Gamma([1+\a+\b+\sigma_{\a,\b}(z)]/2) \Gamma([1+\a+\b-\sigma_{\a,\b}(z)]/2)}
{\Gamma([1+\a-\b+\sigma_{\a,\b}(z)]/2) \Gamma([1+\a-\b-\sigma_{\a,\b}(z)]/2)},&      \\
\a \in [1,\infty), \; \b \in [1,\infty) \backslash \N.&     \no 
\end{align}
\noindent 
{\boldmath $\mathbf{(V)}$ {\bf The case $\a \in [1,\infty)$, $\b = k \in \bbN$:}} \\[1mm] 
Using \eqref{B.26} we see that
\begin{align} \lb{7.15}
    &\psi_{0,+,\a,k}(z, \dott, x_0) \underset{x \downarrow -1}{\sim} y_{1,\a,k,1}(z, \dott)
    \no \\
& \quad \underset{x \downarrow -1} \sim
\f{\G(1+\a)}{\G(a_{\a,-k,\sigma_{\a,k}(z)})\G(a_{\a,-k,-\sigma_{\a,k}(z)})k!}\bigg\{y_{2,\a,k,-1}(z,\dott)
\\\no
&\hspace*{1.2cm}+\big[\psi(a_{\a,k,\sigma_{\a,k}(z)})+\psi(a_{\a,k,-\sigma_{\a,k}(z)})-\psi(k+1)+\g_E\big]\underbrace{y_{1,\a,k,-1}(z,\dott)}_{\sim \varphi_{\a,\b}(z,\dott)}\bigg\}, 
\\\no
& \hspace*{8.45cm}\a \in [1,\infty),\; \b = k \in \bbN,
\end{align}
implying
\begin{align}
    m_{\a,k}(z) &\sim \psi([1+\a+k+\sigma_{\a,k}(z)]/2])+\psi([1+\a+k-\sigma_{\a,k}(z)]/2]),
    \\\no
    & \hspace*{5.8cm}\a \in [1,\infty),\; \b = k \in \bbN,
\end{align}
where we have omitted $\psi(k+1)-\g_E$ which has no spectral significance.
\\[2mm]
\noindent 
{\boldmath $\mathbf{(VI)}$ {\bf The case $\a \in (-\infty,-1]$, $\b = k \in \bbN$:}} \\[1mm] 
After setting the parameters $\a, \, k, \, z$ in \eqref{7.15} to be $-\a,\, k, \,z+(1+k)\a$ and multiplying by $(1-x)^{-\a}$, we can use \eqref{A.21}, together with \eqref{A.29}, \eqref{A.30} to conclude 
\begin{align}
    &\psi_{0,+,\a,k}(z, \dott, x_0) \underset{x \downarrow -1}{\sim} y_{2,\a,k,1}(z, \dott)
    \no \\\lb{7.17}
& \quad \underset{x \downarrow -1} \sim
\f{\G(1-\a)}{\G(a_{-\a,-k,\sigma_{\a,k}(z)})\G(a_{-\a,-k,-\sigma_{\a,k}(z)})k!}\bigg\{y_{2,\a,k,-1}(z,\dott)
\\\no
&\hspace*{1.2cm}+\big[\psi(a_{-\a,k,\sigma_{\a,k}(z)})+\psi(a_{-\a,k,-\sigma_{\a,k}(z)})-\psi(k+1)+\g_E\big]\underbrace{y_{1,\a,k,-1}(z,\dott)}_{\sim \varphi_{\a,\b}(z,\dott)}\bigg\}, 
\\\no
& \hspace*{8.4cm}\a \in (-\infty,-1],\; \b = k \in \bbN,
\end{align}
implying
\begin{align}
    m_{\a,k}(z) &\sim \psi([1-\a+k+\sigma_{\a,k}(z)]/2])+\psi([1-\a+k-\sigma_{\a,k}(z)]/2]),
    \\\no
    & \hspace*{5.25cm}\a \in (-\infty,-1],\; \b = k \in \bbN.
\end{align}
\\[2mm]
\noindent 
{\boldmath $\mathbf{(VII)}$ {\bf The case $\a \in [1,\infty)$, $\b = -k, \; k\in \bbN$:}} \\[1mm] 
After setting the parameter $z$ in \eqref{7.15} to be $z-(1+\a)k$, multiplying by $(1+x)^k$ and using \eqref{A.26}, \eqref{A.27} we see that
\begin{align}
    &\psi_{0,+,\a,-k}(z, \dott, x_0) \underset{x \downarrow -1}{\sim} \overbrace{(1+\dott)^k y_{1,\a,k,1}(z-(1+\a)k, \dott)}^{= \, 2^k y_{1,\a,-k,1}(z,\dott)}
    \no \\
& \quad \underset{x \downarrow -1} \sim
\f{\G(1+\a)}{\G(a_{\a,-k,\sigma_{\a,-k}(z)})\G(a_{\a,-k,-\sigma_{\a,-k}(z)})k!}\bigg\{y_{1,\a,-k,-1}(z,\dott)
\\\no
&\hspace*{1.2cm}+\big[\psi(a_{\a,k,\sigma_{\a,-k}(z)})+\psi(a_{\a,k,-\sigma_{\a,-k}(z)})-\psi(k+1)+\g_E\big]\underbrace{y_{2,\a,-k,-1}(z,\dott)}_{\sim \varphi_{\a,\b}(z,\dott)}\bigg\}, 
\\\no
& \hspace*{8.4cm}\a \in [1,\infty),\; \b = -k, \, k \in \bbN,
\end{align}
implying
\begin{align}
    m_{\a,-k}(z) &\sim \psi([1+\a+k+\sigma_{\a,-k}(z)]/2])+\psi([1+\a+k-\sigma_{\a,-k}(z)]/2]),
    \\\no
    & \hspace*{5.5cm}\a \in [1,\infty),\; \b = -k, \; k \in \bbN.
\end{align}
\\[2mm]
\noindent 
{\boldmath $\mathbf{(VIII)}$ {\bf The case $\a \in (-\infty,-1]$, $\b = -k, \; k \in \bbN$:}} \\[1mm] 
In this case we set $z$ in \eqref{7.17} to be equal to $z-(1+\a)k$, multiply with $(1+x)^k$ and use \eqref{A.26}, \eqref{A.27} to see
\begin{align}
    &\psi_{0,+,\a,-k}(z, \dott, x_0) \underset{x \downarrow -1}{\sim} \overbrace{(1+\dott)^k y_{2,\a,k,1}(z-(1+\a)k, \dott)}^{= \, 2^k y_{2,\a,-k,1}(z,\dott)}
    \no \\
& \ \underset{x \downarrow -1} \sim
\f{\G(1-\a)}{\G(a_{-\a,-k,\sigma_{\a,-k}(z)})\G(a_{-\a,-k,-\sigma_{\a,-k}(z)})k!}\bigg\{y_{1,\a,-k,-1}(z,\dott)
\\\no
&\hspace*{0.8cm}+\big[\psi(a_{-\a,k,\sigma_{\a,-k}(z)})+\psi(a_{-\a,k,-\sigma_{\a,-k}(z)})-\psi(k+1)+\g_E\big]\underbrace{y_{2,\a,-k,-1}(z,\dott)}_{\sim \varphi_{\a,\b}(z,\dott)}\bigg\}, 
\\\no
& \hspace*{7.95cm}\a \in (-\infty,-1],\; \b = -k, \, k \in \bbN,
\end{align}
implying
\begin{align}
    m_{\a,-k}(z) &\sim \psi([1-\a+k+\sigma_{\a,-k}(z)]/2])+\psi([1-\a+k-\sigma_{\a,-k}(z)]/2]),
    \\\no
    & \hspace*{4.95cm}\a \in (-\infty,-1],\; \b = -k, \; k \in \bbN.
\end{align}
We summarize the results of this section as follows:
\begin{theorem} \lb{t7.1} 
Let $\a, \b \in (-\infty,-1] \cup [1,\infty)$. Then the Weyl--Titchmarsh--Kodaira 
$m$-function $m_{\a,\b}$ associated with $T_{\a,\b}$ $($the unique $L^2((-1,1); r_{\a,\b} dx)$-realization 
of $\tau_{\a,\b}$$)$ is of the following form$:$ \\[1mm]
{\boldmath $\mathbf{(I)}$ {\bf In the case $\a \in (-\infty,-1]$, 
$\b \in (-\infty,-1] \backslash (- \N)$:}}
\begin{align}
\begin{split}
m_{\a,\b}(z) &\sim \f{\Gamma([1-\a-\b+\sigma_{\a,\b}(z)]/2) \Gamma([1-\a-\b-\sigma_{\a,\b}(z)]/2)}
{\Gamma([1-\a+\b+\sigma_{\a,\b}(z)]/2) \Gamma([1-\a+\b-\sigma_{\a,\b}(z)]/2)},\\
\sigma(T_{\a,\b}) &= \{(n-\a-\b)(n+1)\}_{n\in\bbN_0}. 
\end{split}
\end{align}
{\boldmath $\mathbf{(II)}$ {\bf In the case $\a \in [1,\infty)$, 
$\b \in (-\infty, -1] \backslash (-\N)$:}} 
\begin{align}
\begin{split}
m_{\a,\b}(z) &\sim \f{\Gamma([1+\a-\b+\sigma_{\a,\b}(z)]/2) \Gamma([1+\a-\b-\sigma_{\a,\b}(z)]/2)}
{\Gamma([1+\a+\b+\sigma_{\a,\b}(z)]/2) \Gamma([1+\a+\b-\sigma_{\a,\b}(z)]/2)},\\
\sigma(T_{\a,\b}) &= \{(n-\b)(n+1+\a)\}_{n\in\bbN_0}. 
\end{split}
\end{align}
{\boldmath $\mathbf{(III)}$ {\bf In the case $\a \in (-\infty,-1]$, 
$\b \in [1,\infty) \backslash \N$:}} 
\begin{align}
\begin{split}
m_{\a,\b}(z) &\sim \f{\Gamma([1-\a+\b+\sigma_{\a,\b}(z)]/2) \Gamma([1-\a+\b-\sigma_{\a,\b}(z)]/2)}
{\Gamma([1-\a-\b+\sigma_{\a,\b}(z)]/2) \Gamma([1-\a-\b-\sigma_{\a,\b}(z)]/2)},\\
\sigma(T_{\a,\b}) &= \{(n-\a)(n+1+\b)\}_{n\in\bbN_0}. 
\end{split}
\end{align}
{\boldmath $\mathbf{(IV)}$ {\bf In the case $\a \in [1,\infty)$, 
$\b \in [1,\infty) \backslash \N$:}} 
\begin{align} 
\begin{split}
m_{\a,\b}(z) &\sim \f{\Gamma([1+\a+\b+\sigma_{\a,\b}(z)]/2) \Gamma([1+\a+\b-\sigma_{\a,\b}(z)]/2)}
{\Gamma([1+\a-\b+\sigma_{\a,\b}(z)]/2) \Gamma([1+\a-\b-\sigma_{\a,\b}(z)]/2)},\\
\sigma(T_{\a,\b}) &= \{n(n+1+\a+\b)\}_{n\in\bbN_0}. 
\end{split}
\end{align}
{\boldmath $\mathbf{(V)}$ {\bf In the case $\a \in [1,\infty)$, 
$\b = k \in \bbN$:}}
\begin{align}
\begin{split}
m_{\a,k}(z) &\sim \psi([1+\a+k+\sigma_{\a,k}(z)]/2])+\psi([1+\a+k-\sigma_{\a,k}(z)]/2]),\\
\sigma(T_{\a,k}) &= \{n(n+1+\a+k)\}_{n\in\bbN_0}.
\end{split}
\end{align}
{\boldmath $\mathbf{(VI)}$ {\bf In the case $\a \in (-\infty,-1]$, 
$\b = k \in\bbN$:}} 
\begin{align}
\begin{split}
m_{\a,k}(z) &\sim \psi([1-\a+k+\sigma_{\a,k}(z)]/2])+\psi([1-\a+k-\sigma_{\a,k}(z)]/2]),\\
\sigma(T_{\a,k}) &= \{(n-\a)(n+1+k)\}_{n\in\bbN_0}. 
\end{split}
\end{align}
{\boldmath $\mathbf{(VII)}$ {\bf In the case $\a \in [1,\infty)$,
$\b = -k, \; k \in \N$:}} 
\begin{align}
\begin{split}
m_{\a,-k}(z) &\sim \psi([1+\a+k+\sigma_{\a,-k}(z)]/2])+\psi([1+\a+k-\sigma_{\a,-k}(z)]/2]),\\
\sigma(T_{\a,-k}) &= \{(n+k)(n+1+\a)\}_{n\in\bbN_0}. 
\end{split}
\end{align}
{\boldmath $\mathbf{(VIII)}$ {\bf In the case $\a \in (-\infty,-1]$, 
$\b = -k, \; k \in \bbN$:}} 
\begin{align} 
\begin{split}
m_{\a,-k}(z) &\sim \psi([1-\a+k+\sigma_{\a,-k}(z)]/2])+\psi([1-\a+k-\sigma_{\a,-k}(z)]/2]),\\
\sigma(T_{\a,-k}) &= \{(n-\a+k)(n+1)\}_{n\in\bbN_0}. 
\end{split}
\end{align}
The $($necessarily simple\,$)$ poles of $m_{\a,\b}$ occur precisely at the $($necessarily simple\,$)$ eigenvalues of $T_{\a,\b}$.
\end{theorem}

\section{More on the Nevanlinna--Herglotz property of $m$-functions \\ and on Jacobi Polynomials} \lb{s8} 

In this section we prove the Nevanlinna--Herglotz property for some of the $m$-functions discussed in 
Sections \ref{s5} and \ref{s6} from scratch, in particular, we will prove  the Nevanlinna--Herglotz property in connection with Jacobi polynomials. For brevity we will focus on the non-degenerate case $\b \neq 0$ only. 

More precisely, we will study the function
\begin{align} \lb{9.1}
& \hatt m_{\a,\b}(z)     \no \\
& \quad =\f{-\Gamma(1-\b)}{2^{1+\a+\b}\b\Gamma(1+\b)}
\f{\Gamma([1+\a + \b +\sigma_{\a,\b}(z)]/2)
\Gamma([1+\a+\b-\sigma_{\a,\b}(z)]/2)}{\Gamma([1+\a-\b+\sigma_{\a,\b}(z)]/2)
\Gamma([1+\a-\b-\sigma_{\a,\b}(z)]/2)},     \no \\
& \hspace*{3.6cm} \a \in \bbR, \; \b \in \bbR \backslash \bbZ, \;  z \in \bbC \backslash \{n(n+1+\a+\b)\}_{n \in \bbN_0},
\end{align}
where we again employed the abbreviation \eqref{5.5}, $\sigma_{\a,\b}(z) = \big[(1+\a+\b)^2 + 4z\big]^{1/2}$, 
$\alpha, \beta \in \bbR, \; z \in \bbC$.  

With the notation used in Sections \ref{s5} and \ref{s6}, one has the following connections between $\hatt m_{\a,\b}$ and Jacobi $m$-functions discussed in Sections \ref{s5} and \ref{s6}, 
\begin{align}
    \hatt m_{\a,\b}(z) = 
    \begin{cases}
    m_{0,0,\a,\b}(z), &\a \in (0,1), \ \b \in (0,1),
    \\
    m_{\frac{\pi}{2}, \frac{\pi}{2}, \a, \b}(z), &\a \in (-1,0), \ \b \in (-1,0),
    \\
    m_{0,\frac{\pi}{2},\a,\b}(z), &\a \in (-1,0), \ \b \in (0,1),
    \\
    m_{\frac{\pi}{2}, 0, \a, \b}(z), &\a \in (0,1), \ \b \in (-1,0),
    \\
    m_{0,\a,\b}(z), &\a  \geq 1, \ \b \in (0,1),
    \\
    m_{\frac{\pi}{2},\a,\b}(z), &\a \geq 1, \ \b \in (-1,0).
    \end{cases}      \lb{9.3} 
\end{align}

We recall that a Nevanlinna--Herglotz (in short, an N--H) function $m(\dott)$ is defined as analytic in 
$\bbC_+$ such that $\Im(m(z)) > 0$, $z\in \bbC_+$. In addition, $m$ is extended to $\bbC_-$ by reflection, that is, one
defines
\begin{equation}
m(z)=\overline{m(\overline z)},
\quad z\in\bbC_- \lb{9.4}
\end{equation}
(which, in general, does not describe the analytic continuation of $m$ to $\bbC_-$). If $m$ and $n$ are N--H functions, then so are $m+n$ and $m(n)$, in particular, as $\bbC_+ \ni z \mapsto -1/z$ is N--H, and hence, 
\begin{equation} 
m \, \text{ is N--H if and only if $-1/m$ is.}     \lb{9.5}
\end{equation}  
Any poles and isolated zeros of $m$ are simple and located on the real axis, the residues at poles being strictly 
negative. If $m$ is N--H, then $m$ has the Nevanlinna, respectively, Riesz--Herglotz representation, that is, there exists a nonnegative measure $d\omega$ on ${\mathbb{R}}$
satisfying
\begin{equation} \lb{9.6}
\int_{{\mathbb{R}}} \frac{d\omega (\la )}{1+\la^2}<\infty
\end{equation}
such that 
\begin{align}
\begin{split}
&m(z)=c+dz+\int_{{\mathbb{R}}} d\omega (\la) \bigg[\frac{1}{\la - z}-\frac{\la} {1+\la^2}\bigg], 
\quad z\in\bbC_+, \lb{9.7} \\[2mm]
& \, c=\Re[m(i)],\quad d=\lim_{\eta \uparrow
\infty}m(i\eta )/(i\eta ) \geq 0. 
\end{split}
\end{align}
Conversely, any function $m$ of the type \eqref{9.7} $($with the measure $\omega$ satisfying 
\eqref{9.6} and some $c \in \bbR$, $d \geq 0$$)$ is N--H. Moreover,
\begin{align}
\begin{split}\lb{9.8} 
&\bbC_+ \ni z \mapsto -e^{-i r \pi}z^r \, \text{ is N--H if and only if } \, r \in [0,1],   
\\
&\bbC_+ \ni z \mapsto e^{-i r \pi}z^r \, \text{ is N--H if and only if } \, r \in [-1,0].
\end{split}
\end{align}
If the measure $d \omega$ in the representation \eqref{9.7} is pure point, then $m$ is meromorphic and real-valued on the real line (with all residues strictly negative). In this case $m$ monotonically increases from $-\infty$ to $+\infty$ between any consecutive poles, hence there is a unique zero between any two poles, an interlacing property.  
Conversely, any meromorphic function real-valued on $\bbR$ with all its residues strictly negative, is a meromorphic N-H function corresponding to a pure point measure $d \omega$ in \eqref{9.7}. (We refer to \cite{GT00} and the extensive literature cited therein for various properties of N--H functions.) 

Using \cite[eq.~5.11.13]{OLBC10}, the behavior of $\hatt m_{\a,\b}$ at infinity reads 
\begin{align} \lb{9.9}
\begin{split}
\hatt m_{\a,\b}(z) \underset{z \rightarrow \infty}{=} \frac{-\G(1-\b)e^{\mp i\b \pi}}{2^{1+\a+\b}\b \G(1+\b)} z^\b 
\big[1+O\big(z^{-1}\big)\big],&  \\
z \in \bbC_{\pm}, \; |\arg z| \geq \varepsilon, \; \varepsilon > 0, \; \a \in \R, \; \b \in \R \backslash  \Z,& 
\end{split}
\end{align}
and hence for $\hatt m_{\a,\b}$ to be N--H, \eqref{9.8} necessarily limits $\beta$ to the interval $[-1,1]$ and due to the  
$\beta$ and Gamma factors $\Gamma(1\pm \beta)$ present in \eqref{9.9} one necessarily assumes from the outset that 
\begin{equation}
\beta \in (-1,1) \backslash \{0\}.     \lb{9.10}
\end{equation} 

Before continuing, we recall that upon combining \eqref{3.24} and \eqref{4.6},  
\begin{equation}
m_{\g,\d,\a,\b}(\dott), \; m_{\g,\a,\b}(\dott) \, \text{ are N--H if $\g, \d \in [0,\pi)$, $\a \in \bbR$, $\b \in (-1,1)$.} 
\lb{9.11}
\end{equation}
In the remainder of this section we prove the N--H property of $\hatt m_{\a,\b}$ from scratch for 
$\alpha \in (- \bbN) \cup(-1,\infty)$ and $\beta \in (-1,1)\backslash \{0\}$ and relate the situation $\alpha \in (- \bbN)$ to Jacobi polynomials. 

\begin{theorem} \lb{t9.1}
Let $\a \in \bbR$, $\b \in \bbR \backslash \bbZ$. Then $\hatt m_{\a,\b}$ as introduced in \eqref{9.1} has the Nevanlinna--Herglotz property if and only if $\a \in (-\bbN)\cup (-1,\infty)$ and $\b \in (-1,0)\cup(0,1)$.   
\end{theorem}
\begin{proof}
One recalls from \eqref{9.10} that $\beta \in (-1,1) \backslash \{0\}$ is a necessary condition for $\hatt m_{\a,\b}$ to be N--H.

Next, we recall the formula\footnote{Formula \eqref{9.12} is also recorded in \cite[p.~5]{MOS66}, however, the minus sign in the last term in the first line of the formula must be replaced by a plus sign.} in \cite[eq.~1.4.(3), p.~8]{EMOT53},
\begin{align}
& \f{\G(1+a-b) \G(1-b)}{\G(a)} \f{\G((a/2)+\zeta) \G((a/2)-\zeta)}{\G(1+(a/2)-b+\zeta) \G(1+(a/2)-b-\zeta)}    \no  \\ 
& \quad = \sum_{n\in\bbN_0} \f{(a)_n (b)_n}{(1+a-b)_n (n!)} \bigg[\f{1}{n+(a/2)+\zeta} + \f{1}{n+(a/2)-\zeta}\bigg],    \lb{9.12} \\
& \hspace*{2.5cm} a,b \in \bbC, \; \Re(b) < 1, \; \zeta \in \bbC \backslash \{\pm (\bbN_0 + (a/2))\},     \no 
\end{align}
proven in \cite[Sect.~2.4]{EMOT53} (see \eqref{A.11} for the introduction of Pochhammer's symbol $(c)_n$, 
$c \in \bbC$, $n \in \bbN_0$). Identifying 
\begin{equation}
\zeta = \sigma_{\a,\b}(z)/2, \quad a = 1+\a+\b, \quad b = 1+\b,     \lb{9.13} 
\end{equation}
in \eqref{9.12}, \eqref{9.1} for $\b \in (-1,0)$ can be rewritten in the form 
\begin{align} 
\hatt m_{\a,\b}(z) &= \f{\G(1+\a+\b)}{2^{1+\a+\b} \G(1+\b) \G(1+\a)} \sum_{n \in \bbN_0} 
\f{(1+\a+\b)_n (1+\b)_n}{(1+\a)_n (n!)}    \no \\
& \hspace*{5.05cm} \times \f{2n+1+\a+\b}{n(n+1+\a+\b) -z}    \no \\
&= \f{1}{2^{1+\a+\b} \G(1+\b)^2} \sum_{n\in\bbN_0} \f{\G(n+1+\a+\b) \G(n+1+\b)}{\G(n+1+\a) (n!)}  \no \\ 
& \hspace*{3.9cm} \times \f{2n+1+\a+\b}{n(n+1+\a+\b) - z},  \quad \b \in (-1,0),     \lb{9.14}
\end{align}
where we employed 
\begin{equation}
(\gamma)_n = \G(n+\gamma)/\G(\gamma), \quad \gamma \in \bbC, \; n \in \bbN_0,     \lb{9.15} 
\end{equation}
in the second equality in \eqref{9.14}. 

Since $\hatt m_{\a,\b}$ is meromorphic and real-valued on $\bbR$, characterizing its N--H property is now equivalent to characterizing when all residues in \eqref{9.14} are strictly negative. By inspection, this is clear from elementary Gamma function properties (see, e.g. \cite[p.~255--256]{AS72}) for $\a > -1$ and $\b \in (-1,0)$. To prove that 
$\hatt m_{\a,\b}$ is not N--H for $\a \in (-\infty,-1] \backslash (-\bbN)$ it suffices to establish some strictly positive residues in \eqref{9.14}. For this purpose suppose that 
$\a \in (-k-1,-k)$ for some $k \in \bbN$ and $\b \in (-1,0)$: Then one verifies that $\G(n+1+\b) > 0$, $n \in\bbN_0$, and choosing $n=0$ one confirms that 
\begin{equation}
\f{\G(1+\a+\b) [1+\a+\b]}{\G(1+\a)} < 0, \quad 1+\a+\b \in (-k,-k+1).     \lb{9.16}
\end{equation}
Similarly, if $n=k$, 
\begin{equation}
\f{\G(k+1+\a+\b) [2k+1+\a+\b]}{\G(k+1+\a)} < 0, \quad 1+\a+\b \in (-k-1,-k).     \lb{9.17} 
\end{equation}
Since $k \in \bbN$ is arbitrary, \eqref{9.16} and \eqref{9.17} prove the non N--H property of $\hatt m_{\a,\b}$ for 
$\a \in (-\infty,-1] \backslash (-\bbN)$ and $\b \in (-1,0)$.

To treat the case $\b \in (0,1)$ one replaces $\hatt m_{\a,\b}$ by $-1/\hatt m_{\a,\b}$ to obtain 
\begin{align}  
& -1/\hatt m_{\a,\b}(z)       \no \\ 
& \quad =\f{2^{1+\a+\b}\b\Gamma(1+\b)}{\Gamma(1-\b)}
\f{\Gamma([1+\a-\b+\sigma_{\a,\b}(z)]/2) \Gamma([1+\a-\b-\sigma_{\a,\b}(z)]/2)}
{\Gamma([1+\a + \b +\sigma_{\a,\b}(z)]/2) \Gamma([1+\a+\b-\sigma_{\a,\b}(z)]/2)},    \no \\
& \hspace*{2.85cm} \a \in \bbR, \; \b \in \bbR \backslash \bbZ, \; 
z \in \bbC \backslash \{(n-\b)(n+1+\a)\}_{n \in \bbN_0}.    \lb{9.18}
\end{align}
Identifying 
\begin{equation}
\zeta = \sigma_{\a,\b}(z)/2, \quad a = 1+\a-\b, \quad b = 1-\b,     \lb{9.19} 
\end{equation}
in \eqref{9.12}, \eqref{9.18} for $\b\in(0,1)$ can be rewritten as 
\begin{align} 
-1/\hatt m_{\a,\b}(z) &= \f{- 2^{1+\a+\b} \b \G(1+\a-\b)}{\G(-\b) \G(1+\a)} \sum_{n \in \bbN_0} 
\f{(1+\a-\b)_n (1-\b)_n}{(1+\a)_n (n!)}    \no \\
& \hspace*{4.9cm} \times \f{2n+1+\a-\b}{(n-\b)(n+1+\a) -z}    \no \\
&= \f{2^{1+\a+\b} \b^2}{\G(1-\b)^2} \sum_{n\in\bbN_0} \f{\G(n+1+\a-\b) \G(n+1-\b)}{\G(n+1+\a) (n!)}  \no \\ 
& \hspace*{2.8cm} \times \f{2n+1+\a-\b}{(n-\b)(n+1+\a) - z},  \quad \b \in (0,1).     \lb{9.20}
\end{align}
The analysis for $\a \in \bbR \backslash (-\bbN)$, $\b \in (-1,0)$ now parallels the one for $\b \in (-1,0)$.

The remaining assertion for $\a \in (-\bbN)$ is obtained as follows. Suppose $\a = - k$ for some $k \in \bbN$. 
Then the fact 
\begin{equation}
1/\G(n+1-k) = 0, \quad 0 \leq n \leq k-1,    \lb{9.21} 
\end{equation}
yields that the first $k$ terms in the sums on the right-hand sides of \eqref{9.14} and \eqref{9.20} (i.e., the terms with 
$n=0,\dots,k-1$) are simply absent. The remaining N--H analysis is not affected by the missing of the first $k$ poles in \eqref{9.14} and \eqref{9.20}.  
\end{proof}

\begin{remark} \lb{r9.2}
An alternative direct proof of the N--H property of $\hatt m_{\a,\b}$ can be obtained by employing its Hadamard factorization as follows. Using the Weierstrass factorization formula for the $\G$-function (also known as Euler's product formula, cf. \cite[eq.~6.1.3]{AS72}),
\begin{align} \lb{9.22}
    \frac{1}{\G(z)} = ze^{\g_E} \prod_{n = 1}^\infty \Big[\Big(1+\frac{z}{n}\Big) e^{-\frac{z}{n}} \Big], \quad z \in \bbC 
\end{align}
(with $\gamma_{E} = 0.57721\dots$ Euler's constant), one verifies that 
\begin{equation} \lb{9.23}
\hatt m_{\a,\b}(z) = -\frac{\G(2+\a+\b)}{2^{1+\a+\b}\b\G(1+\b)\G(2+\a)} \frac{z-z_{-\b}}{z} \prod_{n=1}^{\infty} 
\bigg[1-\frac{z}{z_{n-\b}}\bigg] \bigg/\bigg[1-\frac{z}{z_n}\bigg],
\end{equation}
employing the notation
\begin{equation}
\bbR \ni t \mapsto z_t = t(t+1+\a+\b).
\end{equation}
In particular, the possible poles and zeros of $\hatt m_{\a,\b}$ are given by
\begin{align} 
    z_n = n(n+1+\a+\b), \quad n \in \N_0,     \lb{9.24}
\end{align}
and 
\begin{align} 
    z_{n-\b} = (n-\b)(n+1+\a), \quad n \in \N_0, 
\end{align}
respectively. (Cancellations between poles and zeros are possible, see the case $\a \in (- \bbN)$ in the proof of Theorem \ref{t9.1}.) To characterize the N--H property of $\hatt m_{\a,\b}$ one then employs 
\cite[Theorem 1 on p.~308]{Le80} which amounts to establishing the interlacing property of zeros and poles of 
$\hatt m_{\a,\b}$. \hfill $\diamond$
\end{remark}

We conclude with some additional comments regarding the Nevanlinna--Herglotz property of $\hatt m_{\a,\b}$ for 
$\a \in -\bbN$. For this purpose we take a closer look at the Jacobi polynomials for $\a = -k$, $k \in \bbN$. We start with the general definition given in \cite[eq.~4.12.2]{Sz75}, 
\begin{align}\lb{9.27}
\begin{split} 
   \Pn = \frac{(1+\a)_n}{n!} \mathstrut_2F_1(-n,n+\a+\b+1;1+\a;(1-x)/2),& \\
n\in \N_0, \; \a, \b \in \R, \; x \in (-1,1)&
\end{split} 
\end{align}
(see \eqref{A.11} for the introduction of Pochhammer's symbol $(c)_n$, $c \in \bbC$, $n \in \bbN_0$), where for $\a = -k$, $k\in\bbN$, $n \geq k$, an appropriate limit has to be taken. For this we can use a special case of the formula \cite[eq.~15.1.2]{AS72} to compute
\begin{align} \lb{9.28}
    P_n^{-k,\b}(x) &= \lim_{\a \rightarrow -k} \Pn = \frac{(-1)^k(n+1-k+\b)_k}{(n-k)! k!(n+1-k)_k} 
    \bigg(\frac{1-x}{2}\bigg)^k     \no  \\
    &\quad \, \times \mathstrut_2F_1(-n+k,n+1+\b;k+1;(1-x)/2), \\
    & \hspace*{2.63cm} n \geq k,\; n,k \in \bbN, \; x \in (-1,1),   \no 
\end{align}
where for $n < k$, the prefactor in \eqref{9.27} is zero whenever $\a = -k$, while the hypergeometric function remains well-defined and hence
\begin{align}
    P_n^{-k,\b}(x) = 0, \quad 0 \leq n \leq k-1, \ n \in \N_0, \ k \in \bbN, \; x \in (-1,1).
\end{align}
The Jacobi polynomials $P_n(\dott)$ satisfy (in the sense of distributions) 
\begin{align}
    \tau_{\a,\b} \Pn = z_n \Pn, \quad n\in \N_0, \; \a, \b \in \R, \; x \in (-1,1),
\end{align}
with $\tau_{\a,\b}$ given by \eqref{4.1} and $z_n$ introduced in \eqref{9.24}. In particular, $\big\{P_n^{\a,\b}(\dott)\big\}_{n\in \N_0}$ constitutes an orthogonal basis in the Hilbert space $L^2\big((-1,1);r_{\a,\b} dx\big)$ for $\a > -1$, $\b \in (-1,1)$. 

For $\a\leq-1$, the Jacobi differential expression $\tau_{\a,\b}$ is in the limit-point case at $x = 1$ (cf.~Section \ref{s3}). Hence up to a multiplicative constant, there can be at most one distributional solution $y$ of $\tau_{\a,\b}y = zy$, that lies in $L^2\big((c,1);r_{\a,\b} dx\big)$ for some $c \in (-1,1)$, and from \eqref{9.28} one infers that $P_n^{-k,\b}(x)$ with $n \geq k$ are indeed square integrable locally near $x = 1$ with respect to the weight function $r_{-k,\b}$. As $P_n^{-k,\b}(\dott)$, $n \geq k$, satisfy the Friedrichs boundary condition at $x = -1$ for $\b \in (0,1)$, and the Neumann boundary condition for $\b \in (-1,0)$, one concludes that they are distributional eigensolutions of $\tau_{-k,\b}$ corresponding to $z_n$ saisfying the corresponding boundary condition at $x = -1$ (no boundary condition at $x = 1$ is needed due to the limit-point case at $x=1$). Moreover, from \cite[eq.~15.1.2]{AS72} it also follows that the corresponding $m$-function is again $\hatt m_{\a,\b}$ with poles exactly at $z_n = n(n+1-k+\b)$, $n \geq k$. Thus, one can summarize the case $\a \in (- \bbN)$ as follows.

\begin{theorem} \lb{t9.3} 
Suppose $\a = - k \in (- \bbN)$, $\b \in (-1,1) \backslash \{0\}$, and consider the self-adjoint operator 
\begin{equation}
T_{-k,\b,Jac} = \begin{cases} T_{0,-k,\b}, & \b \in (0,1), \\
T_{\pi/2,-k,\b}, & \b \in (-1,0), \end{cases} \quad k \in \bbN 
\end{equation}
$($cf.\ Section \ref{s6}$)$. Then  
\begin{align}
    M_{-k,\b}(z) &= \begin{cases}
    m_{0,-k,\b}(z),  & \b \in (0,1),
    \\
    m_{\frac{\pi}{2},-k,\b}(z), & \b \in (-1,0), \end{cases}   \no \\
&=\f{-\Gamma(1-\b)}{2^{\b+1-k}\b\Gamma(1+\b)}     \lb{8.31} \\
& \quad \times \f{\Gamma([1-k+\b+\sigma_{-k,\b}(z)]/2)\Gamma([1-k+\b-\sigma_{-k,\b}(z)]/2)}
{\Gamma([1-k-\b+\sigma_{-k,\b}(z)]/2)\Gamma((1-k-\b-\sigma_{-k,\b}(z)]/2)},\no\\
&\hspace{5.8cm} k \in \bbN, \ \b \in (-1,1)\backslash \{0\}.    \no 
\end{align}
Moreover, the spectrum of $T_{-k,\b,Jac}$ is given by 
\begin{equation}
    \sigma(T_{-k,\b,Jac}) = \{n(n+1-k+\b)\}_{n \in \bbN, \,n \geq k}, \quad \beta \in (-1,1) \backslash \{0\},    \lb{8.32} 
\end{equation}
with corresponding eigenfunctions the Jacobi polynomials, $\big\{P_n^{-k,\b}(\dott)\big\}_{n \in \bbN, \, n \geq k}$. In particular, $\big\{P_n^{-k,\b}(\dott)\big\}_{n \in \bbN, \, n \geq k}$ forms a complete orthogonal basis in the Hilbert space $L^2\big((-1,1);r_{-k,\b} dx\big)$. 
\end{theorem}

\begin{remark} \lb{r9.4} 
The statement concerning completeness in Theorem \ref{t9.3} can also be derived directly from the formula (cf. \cite[eq.~4.22.2]{Sz75})
\begin{align}
    \binom{n}{k} P^{-k,\b}_n(x) = \binom{n+\b}{k}\Big(\frac{x-1}{2}\Big)^k P^{k,\b}_{n-k}(x), \quad n \in \bbN, \; n \geq k,
    \; \beta \in (-1,1) \backslash \{0\}, 
\end{align}
and the known completeness of the set of Jacobi polynomials $\big\{P^{k,\b}_{n-k}(\dott)\big\}_{n \in \bbN, \, n \geq k}$ in $L^2\big((-1,1);r_{\a,\b} dx\big)$, $\beta \in (-1,1) \backslash \{0\}$ (see, e.g., \cite[Corollary~5.7.5]{AAR99}, 
\cite[Theorem~4.1.5]{BW10}, \cite[Theorem~3.1.5]{Sz75}).    \hfill $\diamond$
\end{remark}

\begin{remark} \lb{r9.5} 
$(i)$ One notes that the case $\beta\in(0,1),\ \alpha\in(-\bbN)$ in Theorem \ref{t9.3} was already explicitly covered in \eqref{6.26}, but at first glance the two results seem to differ. This apparent discrepancy can be addressed through a simple change of variables. From \eqref{6.26}, one has for $\alpha=-k\in(-\bbN)$ and using the substitution $m=n+k\geq k,\ n\in\bbN_0$,
\begin{align}
\begin{split}
\sigma(T_{F,-k,\beta}) = \{(n+k)(n+1+\beta)\}_{n\in\bbN_0}=\{m(m-k+1+\beta)\}_{m\in\bbN,\, m\geq k},&  \\ 
\beta\in(0,1),&
\end{split} 
\end{align}
in agreement with \eqref{8.31}.
 
The remaining case when $\beta\in(-1,0)$ in Theorem \ref{t9.3} can similarly be verified by using Theorem \ref{t6.1} with $\gamma=\pi/2$ along with \eqref{6.18}. In particular, the $m$-function for $\gamma=\pi/2$ for this case is simply the negative reciprocal of \eqref{6.18}, which is formally \eqref{6.26} with only the $\beta$ range changed.
\\[1mm]
$(ii)$ To directly compare the $m$-functions instead, by $k$ applications of the functional equation $\Gamma(z+1)=z\Gamma(z)$
to each of the Gamma functions appearing in \eqref{8.31}, one can verify this $m$-function equals $m_{0,-k,\beta}(z)$ for $ \beta\in(0,1)$ given in \eqref{6.26}. From the previous discussion, $m_{\frac{\pi}{2},-k,\beta}(z)$ for $\beta\in(-1,0)$ is formally the same as $m_{0,-k,\beta}(z)$ with only the $\beta$ range changed; hence this case follows from the previous one. 
 \hfill $\diamond$
\end{remark}

\appendix

\section{The Hypergeometric and Jacobi Differential Equations} \lb{sA}

In this appendix we provide the connection between the hypergeometric differential equation (cf. \cite[Sect. 15.5]{AS72})
\begin{align}\lb{A.1}
\xi(1-\xi)w''(\xi)+[c-(a+b+1)\xi]w'(\xi)-abw(\xi)=0, \quad \xi \in (0,1),
\end{align}
and the Jacobi differential equation
\begin{align}\lb{A.2}
\begin{split}
\tau_{\a,\b} y(z,x) = -(1-x^2)y''(z,x)+[\a-\b+(\a+\b+2)x]y'(z,x) = z y(z,x),&   \\
\a,\b\in\bbR,\; x\in(-1,1).&
\end{split}
\end{align}

Making the substitution $\xi=(1+x)/2$ in \eqref{A.2} yields
\begin{align}
\begin{split}\lb{A.3}
\xi(1-\xi)y''(z,\xi)+[\b+1-(\a+\b+2)\xi]y'(z,\xi)+zy(z,\xi)=0,\\
\a,\b\in\bbR,\; \xi\in(0,1).
\end{split}
\end{align}
Similarly, the substitution $\xi=(1-x)/2$ in \eqref{A.2} implies
\begin{align}
\begin{split}\lb{A.4}
\xi(1-\xi)y''(z,\xi)+[\a+1-(\a+\b+2)\xi]y'(z,\xi)+zy(z,\xi)=0,\\
\a,\b\in\bbR,\; \xi\in(0,1),
\end{split}
\end{align}
which has effectively only interchanged the roles of the parameters $\a$ and $\b$.

At the endpoints of the Jacobi equation, $x=-1$ and $x=1$ respectively, the substitutions used to arrive at \eqref{A.3} and \eqref{A.4} both give $\xi=0$, hence we next consider solutions of \eqref{A.1} near $\xi=0$ (cf. \cite[Eqs.~15.5.3, 15.5.4]{AS72})
\begin{align}
& w_{1,0}(\xi) = \mathstrut_2F_1(a,b;c;\xi) = \sum_{n \in \bbN_0} \f{(a)_n (b)_n}{(c)_n} \f{\xi^n}{n!}, \quad 
a, b \in \bbC, \; c \in \bbC \backslash (-\bbN_0), \; \xi \in (0,1),   \no \\
& w_{2,0}(\xi) = \xi^{1-c}\mathstrut_2F_1(a-c+1,b-c+1;2-c;\xi),\quad a, b \in \bbC, \; (c-1) \in \bbC \backslash \bbN,   \lb{A.5} \\
& \hspace*{9.85cm} \xi \in (0,1).     \no 
\end{align}
In addition, see \cite[p.~563]{AS72}, 
\begin{equation}
\text{$w_{1,0}$ and $w_{2,0}$ are linearly independent if } \, (a-b), (c-a-b), c \in \bbC \backslash \bbZ.    
\end{equation} 
For $c = k+1 \in \bbN$ we will employ (cf. \cite[eq.~15.5.16--19]{AS72}) 
\begin{align}
 w_{1,0}(\xi) &= \mathstrut_2F_1(a,b;k+1;\xi), \quad a, b \in \bbC, \; \xi \in (0,1),     \no \\
    w_{2,0}^{\ln,k}(\xi) &= \mathstrut_2F_1(a,b;k+1;\xi)\ln(\xi)+\sum_{n=1}^\infty \frac{(a)_n(b)_n}{(k+1)_n n!}\xi^n\big[\psi(a+n)-\psi(a)      \label{A.7} 
    \\\notag
    &\quad +\psi(b+n)-\psi(b)-\psi(k+1+n)+\psi(k+1)-\psi(n+1)-\g_E\big]
    \\\notag
    &\quad -\sum_{n=1}^k \frac{(n-1)!(-k)_n}{(1-a)_k(1-b)_k} \xi^{-n}, \quad a,b \in \bbC \backslash (- \bbN_0),\; \xi \in (0,1),    \no
\end{align}
where the superscipt ``$\ln$'' indicates the presence of a logarithmic term (familiar from Frobenius theory). For the special case $c = 1$ $(k=0)$, the last sum disappears resulting in
\begin{align}
& w_{1,0}(\xi) = \mathstrut_2F_1(a,b;1;\xi), \quad a, b \in \bbC, \; \xi \in (0,1),  \no \\
& w_{2,0}^{\ln,0}(\xi) = \mathstrut_2F_1(a,b;1;\xi) \, \ln(\xi)+\sum_{n\in\bbN}\f{(a)_n(b)_n}{(n!)^2}\xi^n \lb{A.8} \\
&\quad \times [\psi(a+n)-\psi(a)+\psi(b+n)-\psi(b)-2\psi(n+1)-2\gamma_E], 
\quad a,b \in \bbC \backslash (- \bbN_0),    \no \\
& \hspace*{10.67cm} \xi \in (0,1).    \no 
\end{align}
In the case $c = 1 - k \in -\bbN_0$ or $k \in \bbN$ we can use \cite[Eq.~15.5.20,21]{AS72}
\begin{align}
    w_{1,0}(\xi) &= \xi^k \mathstrut_2F_1(a+k,b+k;1+k;\xi), \quad a, b \in \bbC, \; \xi \in (0,1),      \no \\
    w_{2,0}^{\ln,-k}(\xi) &= \xi^k \mathstrut_2F_1(a+k,b+k;1+k;\xi) \ln(\xi)        \no \\ \label{A.9}
    +\xi^k \sum_{n\in\bbN}& \xi^n \frac{(a+k)_n(b+k)_n}{(1+k)_n n!} [\psi(a+k+n)-\psi(a+n)+\psi(b+k+n) 
     \\ 
    &- \psi(b+n) - \psi(k+1+n) + \psi(k+1)-\psi(n+1)-\g_E]     \no \\
    &-\sum_{n=1}^k \frac{(n-1)!(-k)_n}{(1-a-k)_n(1-b-k)_n}\xi^{k-n}, \quad a,b \in \bbC \backslash (- \bbN_0),\; \xi \in (0,1).    \no
\end{align}
In addition,
\begin{equation}
\text{$w_{1,0}$ and $w_{2,0}^{\ln,k}$ are linearly independent if } \, a, b \in \bbC \backslash (-\bbN_0), \,k \in \bbN_0.    
\end{equation}
Here $\mathstrut_2F_1(\dott,\dott;\dott;\dott)$ denotes the hypergeometric function $($see, e.g., \cite[Ch.~15]{AS72}$)$, $\psi(\dott) = \Gamma'(\dott)/\Gamma(\dott)$ the Digamma function, $\gamma_{E} = - \psi(1) = 0.57721\dots$ represents Euler's constant, and 
\begin{equation}
(\zeta)_0 =1, \quad (\zeta)_n = \Gamma(\zeta + n)/\Gamma(\zeta), \; n \in \bbN, 
\quad \zeta \in \bbC \backslash (-\bbN_0),     \lb{A.11} 
\end{equation}
abbreviates Pochhammer's symbol (see, e.g., \cite[Ch.~6]{AS72}). 

Comparing \eqref{A.1} and \eqref{A.3}, and identifying
\begin{align}
\begin{split} 
& a = [1+ \al + \b +\sigma_{\a,\b}(z)]/2,\quad b =  [1+ \al + \b - \sigma_{\a,\b}(z)]/2,   \\
& c= 1 + \b,     \lb{A.12} \\
& \sigma_{\a,\b}(z) = \big[(1+\a+\b)^2+4z\big]^{1/2},    \\
& \xi=(1+x)/2,    
\end{split} 
\end{align}
in \eqref{A.5} and \eqref{A.8}, one obtains for the solutions of the Jacobi differential equation 
$\tau_{\a,\b} y(z,\dott) = z y(z,\dott)$ (cf.\ \eqref{A.2}) near $x=-1$, 
\begin{align}
& y_{1,\a,\b,-1}(z,x) = \mathstrut_2F_1(a_{\a,\b,\sigma_{\a,\b}(z)},a_{\a,\b,-\sigma_{\a,\b}(z)};1+\b;(1+x)/2),  \lb{A.13} \\
& \hspace*{7.5cm} \b \in \bbR \backslash (-\bbN),        \no \\
& y_{2,\a,\b,-1}(z,x) = (1+x)^{-\b} \mathstrut_2F_1(a_{\a,-\b,\sigma_{\a,\b}(z)},a_{\a,-\b,-\sigma_{\a,\b}(z)};1-\b;(1+x)/2), \no \\
& \hspace*{9.2cm} \b \in \bbR\backslash \bbN_0,   \lb{A.14} \\
& y_{2,\a,0,-1}(z,x) =\mathstrut_2F_1(a_{\a,0,\sigma_{\a,0}(z)},a_{\a,0,-\sigma_{\a,0}(z)};1;(1+x)/2) \, \ln((1+x)/2)     \no \\
&\quad +\sum_{n\in\bbN} \f{(a_{\a,0,\sigma_{\a,0}(z)})_n (a_{\a,0,-\sigma_{\a,0}(z)})_n}{2^n(n!)^2} 
(1+x)^n  \lb{A.15} \\
& \hspace*{1.4cm} \times [\psi(a_{\a,0,\sigma_{\a,0}(z)}+n) - \psi(a_{\a,0,\sigma_{\a,0}(z)}) 
+ \psi(a_{\a,0,-\sigma_{\a,0}(z)}+n)   \no \\
& \hspace*{1.9cm} -\psi(a_{\a,0,-\sigma_{\a,0}(z)}) -2\psi(n+1)-2\g_E], \quad \b = 0;    \no \\
& \hspace*{5.1cm} \a\in\bbR,\; z\in\bbC,\; x\in(-1,1),    \no 
\end{align}
where we once again used the abbreviation 
\begin{align}
a_{\mu,\nu,\pm\sigma} = [1 + \mu+\nu \pm \sigma]/2,\quad \mu, \nu, \sigma \in \bbC.    
\end{align}
In the case $\b = k \in \bbN_0$ and $\a \in [0,\infty)$ we define (cf.~\eqref{A.7})
\begin{align} \lb{A.17}
    & y_{1,\a,k,-1}(z,x) = \mathstrut_2F_1(a_{\a,k,\s_{\a,k}(z)}, a_{\a,k,-\s_{\a,k}(z)}; 1+k;(1+x)/2),
    \\\no
    & y_{2,\a,k,-1}(z,x) = \mathstrut_2F_1(a_{\a,k,\s_{\a,k}(z)}, a_{\a,k,-\s_{\a,k}(z)}; 1+k;(1+x)/2)\ln((1+x)/2)
    \\
    &\quad+\sum_{n=1}^\infty \frac{(a_{\a,k,\s_{\a,k}(z)})_n(a_{\a,k,-\s_{\a,k}(z)})_n}{2^n(1+k)_n n!}(1+x)^n
    \\\no
    &\hspace*{1.4cm} \times \, \big[\psi(a_{\a,k,\s_{\a,k}(z)}+n)-\psi(a_{\a,k,\s_{\a,k}(z)})+\psi(a_{\a,k,-\s_{\a,k}(z)}+n)
    \\\notag
    &\hspace*{1.9cm} -\,\psi(a_{\a,k,-\s_{\a,k}(z)})-\psi(1+k+n)+\psi(1+k)-\psi(n+1)-\g_E\big]
    \\\notag
    & \hspace*{2.4cm}-\sum_{n=1}^k \frac{2^n(n-1)!(-k)_n}{(1-a_{\a,k,\s_{\a,k}(z)})_k(1-a_{\a,k,-\s_{\a,k}(z)})_k} (1+x)^{-n};
    \\\no
    &\hspace*{6.65cm} \a \geq 0, \; z \in \bbC, \; x \in (-1,1),
\end{align}
while for $\b = -k \in -\bbN$ and $\a \in [0,\infty)$ one obtains from \eqref{A.9}, 
\begin{align}
    y_{2,\a,-k,-1}(z,x) &= (1+x)^k \mathstrut_2F_1(a_{\a,k,\s_{\a,k}(z)}+k, a_{\a,k,-\s_{\a,k}(z)}+k; 1+k;(1+x)/2),
    \\ \no
    y_{1,\a,-k,-1}(z,x) &= (1+x)^k \mathstrut_2F_1(a_{\a,k,\s_{\a,k}(z)}+k, a_{\a,k,-\s_{\a,k}(z)}+k; 1+k;(1+x)/2)
    \\
    \times \, \ln((1+ & x)/2)+\sum_{n=1}^\infty \frac{(a_{\a,k,\s_{\a,k}(z)}+k)_n(a_{\a,k,-\s_{\a,k}(z)}+k)_n}{2^n(1+k)_n n!}(1+x)^n  \lb{A.20}
    \\\no
    \times \, \big[\psi(&a_{\a,k,\s_{\a,k}(z)}+k+n)-\psi(a_{\a,k,\s_{\a,k}(z)}+n)+\psi(a_{\a,k,-\s_{\a,k}(z)}+k+n)
    \\\notag
    -\,\psi(&a_{\a,k,-\s_{\a,k}(z)}+n)-\psi(k+1+n)+\psi(1+k)-\psi(n+1)-\g_E\big]
    \\\notag
    &-\sum_{n=1}^k \frac{2^n(n-1)!(-k)_n}{(1-a_{\a,k,\s_{\a,k}(z)}-k)_n(1-a_{\a,k,-\s_{\a,k}(z)}-k)_n} (1+x)^{-n};
    \\\no
    &\hspace*{5.55cm} \a \geq 0, \; z \in \bbC, \; x \in (-1,1).
\end{align}
Again, $y_{1,\a,\b,-1}(z,\dott)$ and $y_{2,\a,\b,-1}(z,\dott)$ are linearly independent for $\a \in \bbR$, $\b \in \bbR \backslash \bbZ$. 
Similarly, $y_{1,\a,0,-1}(z,\dott)$ and $y_{2,\a,0,-1}(z,\dott)$ are linearly independent for $\a \in \bbR$. We note that \eqref{A.17}--\eqref{A.20} make sense for $\a \in \bbR$, but we will choose different solutions for $\a \in (-\infty,0)$ which greatly simplify the computation of the $m$-functions in Section 7:
\begin{align} 
    y_{1,\a,\pm k, -1}(z,x) &= (1-x)^{-\a}y_{1,-\a, \pm k,-1}(z+(1\pm k)\a),x),
    \no\\ \lb{A.21}
    y_{2,\a, \pm k,-1}(z,x) &= (1-x)^{-\a}y_{2,-\a,\pm k,-1}(z+(1\pm k)\a,x); 
    \\\no
    &\hspace*{2.53cm} \a \in (-\infty,0), \;  k \in \bbZ \backslash \{0\}. 
\end{align}
That \eqref{A.21} are indeed solutions to the Jacobi differential equation $\tau_{\a,\pm k}y = z y(z,\dott)$ follows from \eqref{4.3}.
\\
In precisely the same manner one obtains for solutions of $\tau_{\a,\b} y(z,\dott) = 0$ near $x=1$, 
\begin{align} 
& y_{1,\a,\b,1}(z,x) = \mathstrut_2F_1(a_{\a,\b,\sigma_{\a,\b}(z)},a_{\a,\b,-\sigma_{\a,\b}(z)};1+\a;(1-x)/2),   \lb{A.22} \\
& \hspace{7.5cm} \a \in \bbR \backslash (- \bbN),       \no \\ 
& y_{2,\a,\b,1}(z,x) = (1-x)^{-\a} \mathstrut_2F_1(a_{-\a,\b,\sigma_{\a,\b}(z)},a_{-\a,\b,-\sigma_{\a,\b}(z)};1-\a;(1-x)/2), \no \\
& \hspace*{9cm} \a \in \bbR \backslash \bbN,   \lb{A.23} \\
& y_{2,0,\b,1}(z,x) = \mathstrut_2F_1(a_{0,\b,\sigma_{0,\b}(z)},a_{0,\b,-\sigma_{0,\b}(z)};1;(1-x)/2) \, \ln((1-x)/2)    \no \\
&\quad +\sum_{n\in\bbN} \f{(a_{0,\b,\sigma_{0,\b}(z)})_n (a_{0,\b,-\sigma_{0,\b}(z)})_n}{2^n(n!)^2}
(1-x)^n    \lb{A.24} \\
& \hspace*{1.4cm} \times [\psi(a_{0,\b,\sigma_{0,\b}(z)}+n) - \psi(a_{0,\b,\sigma_{0,\b}(z)})
+\psi(a_{0,\b,-\sigma_{0,\b}(z)}+n)    \no \\ 
& \hspace*{1.9cm} -\psi(a_{0,\b,-\sigma_{0,\b}(z)}) - 2\psi(n+1)-2\g_E], \quad \a =0;     \no \\
& \hspace*{5.15cm} \b\in\bbR,\; z\in\bbC,\; x\in(-1,1).    \no
\end{align}
Again, $y_{1,\a,\b,1}(z,\dott)$ and $y_{2,\a,\b,1}(z,\dott)$ are linearly independent for $\a \in \bbR \backslash \bbZ$, $\b \in \bbR$. Similarly, $y_{1,0,\b,1}(z,\dott)$ and $y_{2,0,\b,1}(z,\dott)$ are linearly independent for $\b \in \bbR$. In the limit-point case at $x = 1$, that is, for $\a \in (-\infty,-1]\cup[1,\infty)$, we only need the principal solutions, which are $y_{1,\a,\b,1}(z,\dott)$ for $\a \in [1,\infty)$ and $y_{2,\a,\b,1}(z,\dott)$ for $\a \in (-\infty,-1]$. Thus, we see from \eqref{A.22} and \eqref{A.23} that these cases are already covered, and we do not need to define an additional solution for $\a \in \bbZ \backslash  \{0\}$. 

Since $(a_{\a,\b,\sigma_{\a,\b}(z)})_{n} (a_{\a,\b,-\sigma_{\a,\b}(z)})_{n}$, $n \in \bbN_0$,
depends polynomially on $z \in \bbC$, one infers that 
\begin{equation}
\text{for fixed $x \in (0,1)$, $y_{j,\a,\b,\pm 1}(z,x)$, $j=1,2$, are entire with respect to $z \in \bbC$.}  
\end{equation} 
Moreover $y_{j,\a,\b,\pm 1}(z,x)$ satisfy the relations (cf.~ \eqref{4.3} and \eqref{5.15})
\begin{align} \lb{A.26}
y_{1,\a,\b,-1}(z,x) &= (1+x)^{-\b} y_{2,\a, -\b,-1}(z+(1+\a)\b,x),
\\&\hspace{3.25cm} \a \in \bbR, \; \b \in \bbR \backslash \{0\},
\no\\\lb{A.27}
y_{2,\a,\b,-1}(z,x) &= (1+x)^{-\b} y_{1,\a,-\b,-1}(z+(1+\a)\b,x),
\\\no
&\hspace*{3.25cm} \a \in \bbR, \; \b \in \bbR \backslash \{0\},
\\\no
\\\lb{A.28}
y_{1,\a,\b,1}(z,x) &= (1-x)^{-\a} y_{2,-\a, \b,1}(z+(1+\b)\a,x),
\\\no
&\hspace*{3.05cm}\a \in \bbR \backslash \{0\}, \; \b \in \R,
\\\lb{A.29}
y_{2,\a,\b,1}(z,x) &= (1-x)^{-\a} y_{1,-\a,\b,1}(z+(1+\b)\a,x),
\\\no
&\hspace*{3.05cm}\a \in \bbR \backslash \{0\}, \; \b \in \R, 
\end{align}
where we used the fact
\begin{align} \lb{A.30}
    \s_{\a,\b}(z) =
    \begin{cases}
    \s_{\a,-\b}(z+(1+\a)\b),
    \\
    \s_{-\a,\b}(z+(1+\b)\a),
    \\
    \s_{-\a,-\b}(z+\a+\b).
    \end{cases}
\end{align}
\medskip

\begin{remark}\lb{rA1}
We conclude this appendix by briefly discussing Jacobi polynomials and quasi-rational eigenfunctions. The $n$th Jacobi polynomial is defined as (see \cite[eq. 18.5.7]{OLBC10})
\begin{align}
\begin{split}
    P_n^{\a, \b}(x) = \dfrac{(\a+1)_n}{n!} \mathstrut_2F_1(-n, \, n + \a + \b+1; \,
     \a + 1; \, (1-x)/2),& \lb{A.31} \\
     n\in\N_0,\; -\a\notin\N,\; -n-\a-\b-1 \notin \N, \; x \in (-1,1),&
\end{split}
\end{align}
and can be defined by continuity for all parameters $\a, \b \in \R$. One notes that $P_n^{\a,\b}(\dott)$ is a polynomial of degree at most $n$, and has strictly smaller degree if and only if $-n-\a-\b \in \{1, \dots , n \}$ (cf.~\cite[p.~64]{Sz75}). It satisfies the equation
\begin{align}
    \tau_{\a,\b} P_n^{\a, \b}(x) = \l^{\a,\b}_n  P_n^{\a, \b}(x), \quad x \in (-1,1), 
\end{align}
with 
\begin{align}
    \l^{\a,\b}_n = n(n+1+\a+\b).
\end{align}
In particular, one obtains (cf.\ \cite[p.~87]{Is05}, \cite[p.~62]{Sz75})
\begin{equation}
P_0^{\a,\b}(x)=1, \quad P_1^{\a,\b}(x) = [(2+\a+\b) x + \a - \b]/2, \, \text{ etc.,} \; x \in (-1,1). 
\end{equation}

In addition, one can verify that the Jacobi polynomials are solutions of the Jacobi differential equation \eqref{A.2} with Neumann boundary conditions at $x = 1$ (resp., $x=-1$) if $\a\in(-1,0)$ (resp., $\b\in(-1,0)$) and Friedrichs boundary conditions at $x=1$ (resp., $x=-1$) if $\a \in [0,\infty)$ (resp., $\b \in [0,\infty)$).

More generally, all quasi-rational solutions, meaning the logarithmic derivative being rational, can be derived from the Jacobi polynomials together with \eqref{4.3} and are summarized in Table \ref{table1}, which is taken from \cite{Bo19}. Here $(1-x)^{-\a} P_n^{-\a,\b}(x)$ satisfy at $x = 1$ the Friedrichs boundary condition (b.c.) for $\a \in (-\infty, 0]$ and the Neumann b.c. for $\a \in (0,1)$, while at $x = -1$ they satisfy the Friedrichs b.c. for $\b \in [0,\infty)$ and the Neumann b.c. for $\b \in (-1,0)$. For $(1+x)^{-\b} P_n^{\a,-\b}(x)$ the roles of $\a$ and $\b$ interchange compared to the last case, implying the Friedrichs b.c. at $x = 1$ for $\a \in [0,\infty)$, and the Neumann b.c. for $\a \in (-1,0)$, and at $x = -1$, the Friedrichs b.c. for $\b \in (-\infty,0]$, the Neumann b.c. for $\b \in (0,1)$. Finally $(1-x)^{-\a}(1+x)^{-\b} P_n^{-\a,-\b}(x)$ satisfy at $x = 1$ (resp., $x=-1$) the Friedrichs b.c. for $\a \in (-\infty,0]$ (resp., $\b \in (-\infty,0]$) and the Neumann b.c. for $\a \in (0,1)$ (resp., $\b\in(0,1)$). 
\hfill $\diamond$
\end{remark}

\begin{table}[h]
    \centering
    \begin{tabular}{|l| l|}
    \hline
      Eigenfunctions   & Eigenvalues  \\
    \hline 
    $P_n^{\a,\b}(x)$ & $n(n+1+\a+\b)$
    \\
    $(1-x)^{-\a} P_n^{-\a,\b}(x)$ & $n(n+1-\a+\b)-\a(1+\b)$
    \\
    $(1+x)^{-\b} P_n^{\a,-\b}(x)$ & $n(n+1+\a-\b)-\b(1+\a)$
    \\
    $(1-x)^{-\a}(1+x)^{-\b} P_n^{-\a,-\b}(x)$ & $n(n+1-\a-\b)-(\a+\b)$
    \\
    \hline
    \end{tabular}
    \vspace{10pt}
    \caption{Formal quasi-rational eigensolutions of $\tau_{\a,\b}$}
    \label{table1}
\end{table}

\section{Connection Formulas}\lb{sB}

In this appendix we provide the connection formulas utilized to find the solution behaviors in Appendix \ref{sC}, 
\eqref{C.1}--\eqref{C.12}. We express them using $w_{1,0}(\xi)$ and $w_{2,0}(\xi)$ $\big(w_{2,0}^{\ln,k}(\xi)\big)$ and their analogs $w_{1,1}(\xi)$ and $w_{2,1}(\xi)$ $\big(w_{2,1}^{\ln,k}(\xi)\big)$ at the endpoint $\xi = 1$.

We recall the relations \eqref{A.12} connecting the parameters $a,b,c$ and $\a, \b$ and note that $\xi \in (0,1)$ for the remainder of this appendix. \\[1mm] 
\noindent 
{\boldmath $\mathbf{(I)}$ {\bf The case $\a \in \R \backslash \Z$, $\b \in (-1,1) \backslash \{0\}$, that is, $c \in (0,2) \backslash \{1\}$, $a+b-c \in \R \backslash \Z$\,:}} \\[1mm] 
The two solutions $w_{1,1}(\xi)$ and $w_{2,1}(\xi)$ can be obtained from \eqref{A.5} by the change of parameters
\begin{equation} \lb{B.1}
    (a, \, b, \, c, \, \xi) \rightarrow (a, \, b, \, a+b-c+1, \, 1 - \xi),
\end{equation}
or, explicitly,
\begin{align} \label{B.2}
    w_{1,1}(\xi) &= \mathstrut_2F_1(a,b;a+b-c+1;1-\xi), 
    \\ \label{B.3}
    w_{2,1}(\xi) &= (1-\xi)^{c-a-b}\mathstrut_2F_1(c-a,c-b;c-a-b+1;1-\xi).
\end{align}
The two connection formulas are given by (cf. \cite[eq.~15.10.21--22]{OLBC10})
\begin{align} \lb{B.4}
w_{1,0}(\xi) &= \f{\Gamma(c)\Gamma(c-a-b)}{\Gamma(c-a)\Gamma(c-b)}w_{1,1}(\xi)+\f{\Gamma(c)\Gamma(a+b-c)}{\Gamma(a)\Gamma(b)}w_{2,1}(\xi),  \\
w_{2,0}(\xi) &= \frac{\G(2-c)\G(c-a-b)}{\G(1-a)\G(1-b)}w_{1,1}(\xi)+\frac{\G(2-c)\G(a+b-c)}{\G(a-c+1)\G(b-c+1)}w_{2,1}(\xi).    \lb{B.5}
\end{align}
One notes that poles occur on the right-hand sides of \eqref{B.4}, \eqref{B.5} whenever $(a+b-c) \in \bbZ$. Using \eqref{B.1} we can also express $w_{1,1}(\xi)$ or $w_{2,1}(\xi)$ as a linear combination of $w_{1,0}(\xi)$ and $w_{2,0}(\xi)$:
\begin{align}
    w_{1,1}(\xi) &= \f{\Gamma(a+b-c+1)\Gamma(1-c)}{\Gamma(a-c+1)\Gamma(b-c+1)}w_{1,0}(\xi)+\f{\Gamma(a+b-c+1)\Gamma(c-1)}{\Gamma(a)\Gamma(b)}w_{2,0}(\xi),     \lb{B.6} \\
    w_{2,1}(\xi) &= \frac{\G(1+c-a-b)\G(1-c)}{\G(1-a)\G(1-b)}w_{1,0}(\xi)+\frac{\G(1+c-a-b)\G(c-1)}{\G(c-a)\G(c-b)}w_{2,0}(\xi).     \lb{B.7} 
\end{align}

\noindent 
{\boldmath $\mathbf{(II)}$ {\bf The case $\a = 0, \ \b \in \bbR \backslash \bbZ$, that is, $c \in \bbR \backslash \bbZ$, 
$a+b = c$\,:}} \\[1mm] 
The solution  $w_{1,0}(\xi) =  \mathstrut_2F_1(a,b;a+b;\xi)$ can be expanded at $\xi = 1$ (cf. \cite[eq. 15.3.10]{AS72}):
\begin{align}
\begin{split} \lb{B.8}
\mathstrut_2F_1(a,b;a+b;\xi)&=\f{\Gamma(a+b)}{\Gamma(a)\Gamma(b)}\sum_{n\in\N_0} \f{(a)_n(b)_n}{(n!)^2} 
[2\psi(n+1)-\psi(a+n)-\psi(b+n)\\
&\hspace*{4.05cm} -\ln(1-\xi)](1-\xi)^n.
\end{split}
\end{align}
Meanwhile, two linearly independent solutions at $\xi = 1$ are given by (cf. \eqref{A.8}):
\begin{align} 
        w_{1,1}(\xi) &= \mathstrut_2F_1(a,b;1;1-\xi),     \lb{B.9} \\
        w_{2,1}^{\ln,0}(\xi) &= \mathstrut_2F_1(a,b;1;1-\xi)\ln(1-\xi)+\sum_{n\in\bbN}\f{(a)_n(b)_n}{(n!)^2}(1-\xi)^n  \lb{B.10} \\
        &\quad\, \times[\psi(a+n)-\psi(a)+\psi(b+n)-\psi(b)-2\psi(n+1)-2\gamma_E]. \no
\end{align}
The connection formula for $w_{1,1}(\xi)$ is given by \eqref{B.6} with $a+b=c$. To obtain the second connection formula we need to compare the expansion of $w_{2,1}^{\ln}(\xi)$ at $\xi = 1$, with the expansion of $\mathstrut_2F_1(a,b;a+b;\xi)$ at $\xi = 1$ using \eqref{B.8}. We obtain
\begin{align} \lb{B.11}
    w_{2,1}^{\ln,0}(\xi) =&  - [\psi(1-a)+\psi(1-b)+2\g_{E}] \frac{\G(1-a-b)}{\G(1-a)\G(1-b)} w_{1,0}(\xi) \no
    \\
    &- [\psi(a)+\psi(b)+2\g_{E}] \frac{\G(a+b-1)}{\G(a)\G(b)}w_{2,0}(\xi).
\end{align}
The reflection formula for the $\psi$-function \cite[eq.~6.3.7]{AS72}, 
\begin{align}
    \psi(1-z) - \psi(z) = \pi \cot(\pi z),
\end{align}
also implies 
\begin{align} \lb{B.13}
    w_{1,0}(\xi)&=-\frac{\G(1-a)\G(1-b)}{\G(1-a-b)\pi [\cot(\pi a) + \cot(\pi b)]}
    \no\\
    &\qquad \times\big[[\psi(a)+\psi(b)+2\g_E] w_{1,1}(\xi) + w_{2,1}^{\ln,0}(\xi)\big], \\    
    w_{2,0}(\xi) &= \frac{\G(a)\G(b)}{\G(a+b-1)\pi [\cot(\pi a) + \cot(\pi b)]}
    \no\\
    &\quad\, \times\big[[\psi(1-a)+\psi(1-b)+2\g_E] w_{1,1}(\xi) + w_{2,1}^{\ln,0}(\xi) \big].  \lb{B.14}
\end{align}
{\boldmath $\mathbf{(III)}$ {\bf The case $\a \in \R \backslash \Z, \, \b = 0$, that is, $c = 1$, 
$a + b \in \R \backslash \Z$\,:}} \\[1mm] 
This case is analogous to the previous case, with the roles of $\a$ and $\b$ interchanged. Concretely, this means that the connection formulas \eqref{B.11}--\eqref{B.14} must be changed through the renaming of \eqref{B.1} with $c \rightarrow a+b-c+1=a+b$ as $c = 1$. Since $c$ does not appear in \eqref{B.11}--\eqref{B.14} (we eliminated it through $c = a+b$), we can adopt the aforementioned formulas directly only changing the second index in the $w$'s\footnote{Formula \eqref{B.15} could have been obtained directly from \eqref{B.4} by setting $c = 1$.}
\begin{align} 
    w_{1,0}(\xi) &= \frac{\G(1-a-b)}{\G(1-a)\G(1-b)}w_{1,1}(\xi)+ \frac{\G(a+b-1)}{\G(a)\G(b)}w_{2,1}(\xi),
   \lb{B.15} \\
    w_{2,0}^{\ln,0}(\xi) &=- [\psi(1-a)+\psi(1-b)+2\g_{E}] \frac{\G(1-a-b)}{\G(1-a)\G(1-b)}w_{1,1}(\xi)
   \no \\
    & \quad \, - [\psi(a)+\psi(b)+2\g_E] \frac{\G(a+b-1)}{\G(a)\G(b)}w_{2,1}(\xi),
    \lb{B.16} \\
    w_{1,1}(\xi)&=-\frac{\G(1-a)\G(1-b)}{\G(1-a-b)\pi [\cot(\pi a) + \cot(\pi b)]}
    \no\\
    &\qquad\, \times\big[[\psi(a)+\psi(b)+2\g_E] w_{1,0}(\xi) + w_{2,0}^{\ln,0}(\xi)\big],
    \lb{B.17} \\
    w_{2,1}(\xi) &= \frac{\G(a)\G(b)}{\G(a+b-1)\pi [\cot(\pi a) + \cot(\pi b)]}
    \no\\
    &\quad\, \times\big[[\psi(1-a)+\psi(1-b)+2\g_E] w_{1,0}(\xi) + w_{2,0}^{\ln,0}(\xi)\big].  \lb{B.18} 
\end{align}
{\boldmath $\mathbf{(IV)}$ {\bf The case $\a = \b = 0$, that is, $ a + b = c = 1$:}} \\[1mm] 
For $\a = 0$ and $\b = 0$ the Jacobi differential expression \eqref{4.1} becomes the Legendre differential expression. This case was treated in detail in \cite{GLN20}, we shall only write down the connection formulas for completeness.  

The special solutions $w_{1,j}(\xi)$ and $w_{2,j}^{\ln,0}(\xi)$ for $j = 1,2$ are given by \eqref{A.8} and 
\eqref{B.9}, \eqref{B.10}. One notes that the following relations hold,
\begin{align} \lb{B.19}
    w_{1,1}(\xi) = w_{1,0}(1-\xi), \quad
    w_{2,1}^{\ln,0}(\xi) = w_{2,0}^{\ln,0}(1-\xi).
\end{align}
Using \cite[eq.~15.3.10]{AS72} together with $w_{1,0} = \mathstrut_2F_1(a,b;a+b, \xi)$ 
and Euler's reflection formula for the $\G$-function \cite[eq.~6.1.17]{AS72}, 
\begin{align} \lb{B.20}
    \G(z)\G(1-z) = \pi \csc(\pi z),
\end{align}
one obtains 
\begin{align}
    w_{1,0}(\xi) = - \pi^{-1} \sin(\pi a) \big[[\psi(a)+\psi(b)+2\g_E]w_{1,1}(\xi)+w_{2,1}^{\ln,0}(\xi)\big].
\end{align}
This formula can be also obtained as the limit $\b \to 0$ in \eqref{B.13}. From the two relations \eqref{B.19} we immediately get
\begin{align}
    w_{1,1}(\xi) &= - \pi^{-1} \sin(\pi a) \big[[\psi(a)+\psi(b)+2\g_E]w_{1,0}(\xi)+w_{2,0}^{\ln,0}(\xi)\big],
    \\
    w_{2,0}^{\ln,0}(\xi) &=  \pi^{-1} \sin(\pi a) \big\{\big[[\psi(a)+\psi(b)+2\g_E]^2-\pi^2 [\sin(\pi a)]^{-2}
\big] w_{1,1}(\xi)
    \no\\
    &\quad + [\psi(a)+\psi(b)+2\g_E] w_{2,1}^{\ln,0}(\xi)\big\},
    \\
    w_{2,1}^{\ln,0}(\xi) &=  \pi^{-1} \sin(\pi a) \big\{\big[[\psi(a)+\psi(b)+2\g_E]^2-\pi^2 [\sin(\pi a)]^{-2}\big] 
     w_{1,0}(\xi)
    \no\\
    &\quad + [\psi(a)+\psi(b)+2\g_E] w_{2,0}^{\ln,0}(\xi)\big\}.
\end{align}
{\boldmath $\mathbf{(V)}$ {\bf The case $\a \in [0,\infty), \, \b = k \in \bbN$, that is, $c = 1+k$, $a + b \geq c$\,:}} 
\\[1mm]
The fundamental system at $\xi = 0$ consists of $w_{1,0}(\xi)$ given in \eqref{A.5} and \eqref{A.7}. At the other endpoint we have the fundamental system \eqref{B.2} and \eqref{B.3}. In particular, as $a+b-c+1 = a+b-k$, we see that $w_{1,1}(\xi)$ has the form
\begin{align}
    w_{1,1}(\xi) = \mathstrut_2F_1(a,b;a+b-k;1-\xi), \quad k \in \bbN.
\end{align}
Thus, one can use \cite[eq.~15.3.12]{AS72} to read off the behavior of $w_{1,1}(\xi)$ at $\xi = 0$. Comparing with the expansion of $w_{1,0}(\xi)$ and $w_{2,0}^\ln(\xi)$ at $\xi = 0$, one concludes that
\begin{align} \label{B.26}
    w_{1,1}(\xi) &= -\frac{(-1)^k \G(a+b-k)}{\G(a-k)\G(b-k)k!} 
    \\
    &\quad \quad  \times \notag \big[[\psi(a)+\psi(b)-\psi(k+1)+\g_E]w_{1,0}(\xi)+w_{2,0}^{\ln,k}(\xi)\big].
\end{align}
As this is the only connection formula that we need, we just mention that the other ones can be computed in a similar fashion.

\section{Behavior of $y_{j,\a,\b,\mp 1}(z,x)$, $j=1,2$, near $x=\pm 1$} \lb{sC}

In this appendix we focus on solutions of the Jacobi differential equation. 

We start with the limiting behavior of $y_{j,\a,\b,-1}(z,x)$ as $x\uparrow1$ and apply the connection formulas from Appendix \ref{sB}: 
\begin{align}
& y_{1,\a,\b,-1}(z,x)\underset{x\uparrow1}{=}\f{\Gamma(1+\b)\Gamma(-\a)}
{\Gamma(a_{-\a,\b,\sigma_{\a,\b}(z)})\Gamma(a_{-\a,\b,-\sigma_{\a,\b}(z)})}[1+\Oh(1-x)] \no \\
& \hspace*{2.7cm} +\f{2^{\a} \Gamma(1+\b)\Gamma(\a)}{\Gamma(a_{\a,\b,\sigma_{\a,\b}(z)})
\Gamma(a_{\a,\b,-\sigma_{\a,\b}(z)})} (1-x)^{-\a}[1+\Oh(1-x)], \no \\
& \hspace*{7.7cm} \a \in \bbR \backslash \bbZ,\; \b\in(-1,1),     \lb{C.1} \\  
&y_{1,\a,\b,-1}(z,x)\underset{x\uparrow1}{=}\f{\Gamma(1+\b)\Gamma(-\a)}{\Gamma(a_{-\a,\b,\sigma_{\a,\b}(z)})\Gamma(a_{-\a,\b,-\sigma_{\a,\b}(z)})}[1+\Oh(1-x)]   \no \\
& \hspace*{2.7cm} -\f{(-1)^{-\a} 2^{\a} \Gamma(1+\b)}{(-\a)!\Gamma(a_{\a,\b,\sigma_{\a,\b}(z)})
\Gamma(a_{\a,\b,-\sigma_{\a,\b}(z)})} (1-x)^{-\a}     \lb{C.2} \\
& \hspace*{2.7cm} \quad \times[\ln((1-x)/2)+\gamma_E-\psi(1-\a)+\psi(a_{-\a,\b,\sigma_{\a,\b}(z)})    \no \\
& \hspace*{2.7cm} +\psi(a_{-\a,\b,-\sigma_{\a,\b}(z)})][1+\Oh(1-x)],\quad \a\in (-\bbN),\; \b\in(-1,1),   \no \\
& y_{1,0,\b,-1}(z,x)\underset{x\uparrow1}{=}-\f{\Gamma(1+\b)}{\Gamma(a_{0,\b,\sigma_{0,\b}(z)})\Gamma(a_{0,\b,-\sigma_{0,\b}(z)})}[\ln((1-x)/2)+2\gamma_E     \no \\
& \hspace*{2.3cm} \quad +\psi(a_{0,\b,\sigma_{0,\b}(z)})+\psi(a_{0,\b,-\sigma_{0,\b}(z)})][1+\Oh(1-x)],   \lb{C.3} \\
& \hspace*{7.3cm} \a = 0, \; \b\in(-1,1),  \no \\
& y_{1,\a,\b,-1}(z,x)\underset{x\uparrow1}{=}\f{2^{\a} \Gamma(1+\b)\Gamma(\a)}
{\Gamma(a_{\a,\b,\sigma_{\a,\b}(z)})\Gamma(a_{\a,\b,-\sigma_{\a,\b}(z)})} (1-x)^{-\a}[1+\Oh(1-x)]    \no \\
& \hspace*{2.7cm}-\f{(-1)^{\a}\Gamma(1+\b)}{\a!\Gamma(a_{-\a,\b,\sigma_{\a,\b}(z)})\Gamma(a_{-\a,\b,-\sigma_{\a,\b}(z)})}[\ln((1-x)/2)+\gamma_E  \lb{C.4}\\
& \hspace*{2.7cm} -\psi(\a+1)+\psi(a_{\a,\b,\sigma_{\a,\b}(z)})+\psi(a_{\a,\b,-\sigma_{\a,\b}(z)})]    \no \\
& \hspace*{2.7cm} \quad\times[1+\Oh(1-x)],\quad \a\in\bbN,\; \b\in(-1,1),   \no \\
& y_{2,\a,\b,-1}(z,x)\underset{x\uparrow1}{=}\f{2^{-\b}\Gamma(1-\b)\Gamma(-\a)}
{\Gamma(a_{-\a,-\b,\sigma_{\a,\b}(z)})\Gamma(a_{-\a,-\b,-\sigma_{\a,\b}(z)})}[1+\Oh(1-x)]     \no \\
& \hspace*{2.7cm} +\f{2^{\a-\b}\Gamma(1-\b)\Gamma(\a)}{\Gamma(a_{\a,-\b,\sigma_{\a,\b}(z)})
\Gamma(a_{\a,-\b,-\sigma_{\a,\b}(z)})} (1-x)^{-\a}[1+\Oh(1-x)],     \no \\
& \hspace*{7cm} \a \in \bbR \backslash \bbZ,\; \b\in(-1,1)\backslash\{0\},      \lb{C.5} \\
& y_{2,\a,\b,-1}(z,x)\underset{x\uparrow1}{=}\f{2^{-\b}\Gamma(1-\b)\Gamma(-\a)}
{\Gamma(a_{-\a,-\b,\sigma_{\a,\b}(z)})\Gamma(a_{-\a,-\b,-\sigma_{\a,\b}(z)})}[1+\Oh(1-x)]     \no \\
& \hspace*{2.7cm} -\f{(-1)^{-\a}2^{\a-\b}\Gamma(1-\b)}{(-\a)!\Gamma(a_{\a,-\b,\sigma_{\a,\b}(z)})
\Gamma(a_{\a,-\b,-\sigma_{\a,\b}(z)})} (1-x)^{-\a}[\ln((1-x)/2)    \no \\
& \hspace*{2.7cm} +\gamma_E-\psi(1-\a)+\psi(a_{-\a,-\b,\sigma_{\a,\b}(z)})+\psi(a_{-\a,-\b,-\sigma_{\a,\b}(z)})]     \no \\
& \hspace*{2.7cm} \quad \times[1+\Oh(1-x)],\quad \a\in (-\bbN),\; \b\in(-1,1)\backslash\{0\},     \lb{C.6} \\
& y_{2,0,\b,-1}(z,x)\underset{x\uparrow1}{=}-\f{2^{-\b}\Gamma(1-\b)}{\Gamma(a_{0,-\b,\sigma_{0,\b}(z)})
\Gamma(a_{0,-\b,-\sigma_{0,\b}(z)})}[\ln((1-x)/2)+2\gamma_E    \no \\
& \hspace*{2.3cm}\quad +\psi(a_{0,-\b,\sigma_{0,\b}(z)}) + \psi(a_{0,-\b,-\sigma_{0,\b}(z)})][1+\Oh(1-x)],      \lb{C.7} \\
& \hspace*{7.1cm} \a = 0, \; \b\in(-1,1)\backslash\{0\},     \no \\ 
& y_{2,\a,\b,-1}(z,x)\underset{x\uparrow1}{=}\f{2^{\a-\b}\Gamma(1-\b)\Gamma(\a)}
{\Gamma(a_{\a,-\b,\sigma_{\a,\b}(z)})\Gamma(a_{\a,-\b,-\sigma_{\a,\b}(z)})} (1-x)^{-\a}[1+\Oh(1-x)]      \no \\
& \hspace*{2.7cm} -\f{(-1)^{\a}2^{-\b}\Gamma(1-\b)}{\a!\Gamma(a_{-\a,-\b,\sigma_{\a,\b}(z)})
\Gamma(a_{-\a,-\b,-\sigma_{\a,\b}(z)})}[\ln((1-x)/2)+\gamma_E 
\no \\
& \hspace*{2.7cm}-\psi(\a+1)+\psi(a_{\a,-\b,\sigma_{\a,\b}(z)})+\psi(a_{\a,-\b,-\sigma_{\a,\b}(z)})]    \lb{C.8} \\
& \hspace*{2.7cm}\quad \times[1+\Oh(1-x)],\quad \a\in\bbN,\; \b\in(-1,1)\backslash\{0\},  \no \\ 
& y_{2,\a,0,-1}(z,x)\underset{x\uparrow1}{=}\f{-[2\gamma_E+\psi(a_{-\a,0,\sigma_{\a,0}(z)}) 
+\psi(a_{-\a,0,-\sigma_{\a,0}(z)})] \Gamma(-\a)}{\Gamma(a_{-\a,0,\sigma_{\a,0}(z)}) 
\Gamma(a_{-\a,0,-\sigma_{\a,0}(z)})}     \no \\
& \hspace{2.7cm}\quad \times[1+\Oh(1-x)]    \no \\
& \hspace*{2.7cm} -\f{[2\gamma_E+\psi(a_{\a,0,\sigma_{\a,0}(z)})+\psi(a_{\a,0,-\sigma_{\a,0}(z)})]2^{\a}\Gamma(\a)}{\Gamma(a_{\a,0,\sigma_{\a,0}(z)})\Gamma(a_{\a,0,-\sigma_{\a,0}(z)})} (1-x)^{-\a}   \no \\
& \hspace*{2.7cm}\quad \times[1+\Oh(1-x)],\quad \a \in \bbR \backslash \bbZ, \; \b = 0,  \lb{C.9} \\
& y_{2,\a,0,-1}(z,x)\underset{x\uparrow1}{=}\f{-[2\gamma_E+\psi(a_{-\a,0,\sigma_{\a,0}(z)}) 
+\psi(a_{-\a,0,-\sigma_{\a,0}(z)})]\Gamma(-\a)}{\Gamma(a_{-\a,0,\sigma_{\a,0}(z)})
\Gamma(a_{-\a,0,-\sigma_{\a,0}(z)})}    \no \\
& \hspace*{2.7cm}\quad  \times[1+\Oh(1-x)] -\bigg(\f{\Gamma(a_{-\a,0,\sigma_{\a,0}(z)})\Gamma(a_{-\a,0,-\sigma_{\a,0}(z)})}{\Gamma(1-\a)}    \no \\
& \hspace*{2.7cm}\quad\quad -\f{(-1)^{-\a}[2\gamma_E+\psi(a_{-\a,0,\sigma_{\a,0}(z)})+\psi(a_{-\a,0,-\sigma_{\a,0}(z)})]}{\Gamma(a_{\a,0,\sigma_{\a,0}(z)})\Gamma(a_{\a,0,-\sigma_{\a,0}(z)})}    \no \\
& \hspace*{2.7cm}\quad \times[\ln((1-x)/2)+\gamma_E-\psi(1-\a)+\psi(a_{-\a,0,\sigma_{\a,0}(z)})     \lb{C.10} \\
& \hspace*{2.7cm}\quad\quad +\psi(a_{-\a,0,-\sigma_{\a,0}(z)})]\bigg) 2^{\a} (1-x)^{-\a}[1+\Oh(1-x)], \no \\
& \hspace*{8.2cm}\a\in (-\bbN), \; \b = 0,   \no \\
& y_{2,0,0,-1}(z,x)\underset{x\uparrow1}{=}-\bigg(\Gamma([1+\sigma_{0,0}(z)]/2)
\Gamma([1-\sigma_{0,0}(z)]/2)-[\ln((1-x)/2)+2\gamma_E     \no \\
& \hspace*{2.3cm}\quad +\psi([1+\sigma_{0,0}(z)]/2)+\psi([1-\sigma_{0,0}(z)]/2)]    \lb{C.11} \\
& \hspace*{2.7cm}\quad \times\f{\psi([1+\sigma_{0,0}(z)]/2)+\psi([1-\sigma_{0,0}(z)]/2)+\gamma_E}
{\Gamma([1+\sigma_{0,0}(z)]/2)\Gamma([1-\sigma_{0,0}(z)]/2)}\bigg)   \no \\
& \hspace*{2.7cm}\quad \times[1+\Oh(1-x)],\quad \a=\b=0,    \no \\
& y_{2,\a,0,-1}(z,x)\underset{x\uparrow1}{=}\f{-(2\gamma_E+\psi(a_{\a,0,\sigma_{\a,0}(z)}) 
+ \psi(a_{\a,0,-\sigma_{\a,0}(z)}))2^{\a}\Gamma(\a)}{\Gamma(a_{\a,0,\sigma_{\a,0}(z)})
\Gamma(a_{\a,0,-\sigma_{\a,0}(z)})}     \no \\
& \hspace*{2.7cm}\quad \times (1-x)^{-\a}[1+\Oh(1-x)] -\bigg(\f{\Gamma(a_{\a,0,\sigma_{\a,0}(z)})\Gamma(a_{\a,0,-\sigma_{\a,0}(z)})}{\Gamma(\a+1)}   \no \\
& \hspace*{2.7cm} -\f{(-1)^\a [2\gamma_E+\psi(a_{\a,0,\sigma_{\a,0}(z)})+\psi(a_{\a,0,-\sigma_{\a,0}(z)})]}{\Gamma(a_{-\a,0,\sigma_{\a,0}(z)})\Gamma(a_{-\a,0,-\sigma_{\a,0}(z)})}     \lb{C.12} \\
& \hspace*{2.7cm}\quad \times[\ln((1-x)/2)+\gamma_E-\psi(\a+1)+\psi(a_{\a,0,\sigma_{\a,0}(z)})     \no \\
& \hspace*{2.7cm}\quad +\psi(a_{\a,0,-\sigma_{\a,0}(z)})]\bigg)[1+\Oh(1-x)],\quad \a\in\bbN,\; \b = 0,    \no \\
& \hspace*{8.1cm} z \in \bbC, \; x \in (-1,1).   \no 
\end{align}

We continue this appendix with a list of generalized boundary values for the solutions $y_{j,\a,\b,-1}(z,x),\ j=1,2$ at $x= \mp 1$. At the endpoint $x=-1$ one obtains  for $z\in\bbC$, 
\begin{align}\lb{C.13}
\begin{split}
\wti y_{1,\a,\b,-1}(z,-1)&=\begin{cases}
1, & \b\in(-1,0),\\
0, & \b=0,\\
0, & \b\in(0,1),
\end{cases}\\
\wti y_{1,\a,\b,-1}^{\, \prime}(z,-1)&=\begin{cases}
0, & \b\in(-1,0),\\
1, & \b=0,\\
1, & \b\in(0,1),
\end{cases}\\
\wti y_{2,\a,\b,-1}(z,-1)&=\begin{cases}
0, & \b\in(-1,0),\\
-2^{\a+1}, & \b=0,\\
\b 2^{\a+1}, & \b\in(0,1),
\end{cases}\\
\wti y_{2,\a,\b,-1}^{\, \prime}(z,-1)&=\begin{cases}
-\b 2^{\a+1}, & \b\in(-1,0),\\
0, & \b=0,\\
0, & \b\in(0,1);
\end{cases}
\end{split}
\quad \a\in\bbR.
\end{align}

To obtain generalized boundary values at the endpoint $x=1$, one employs the connection formulas \eqref{C.1}--\eqref{C.12}), to find for $z\in\bbC$, 
\begin{align}
\no \begin{split}
\wti y_{1,\a,\b,-1}(z,1)&=\begin{cases}
\dfrac{\Gamma(1+\b)\Gamma(-\a)}{\Gamma(a_{-\a,\b,\sigma_{\a,\b}(z)})\Gamma(a_{-\a,\b,-\sigma_{\a,\b}(z)})}, & \a\in(-1,0),\\[3mm]
\dfrac{- 2^{1+\a+\b}\Gamma(1+\a)\Gamma(1+\b)}{\Gamma(a_{\a,\b,\sigma_{\a,\b}(z)})\Gamma(a_{\a,\b,-\sigma_{\a,\b}(z)})}, & \a\in[0,1),
\end{cases}
\end{split}\\[1mm]
\no \begin{split}
\wti y_{1,\a,\b,-1}^{\, \prime}(z,1)&=\begin{cases}
\dfrac{2^{1+\a+\b}\Gamma(1+\a)\Gamma(1+\b)}{\Gamma(a_{\a,\b,\sigma_{\a,\b}(z)})\Gamma(a_{\a,\b,-\sigma_{\a,\b}(z)})}, & \a\in(-1,0),\\[3mm]
\dfrac{-\Gamma(1+\b)}{\Gamma(a_{0,\b,\sigma_{0,\b}(z)})\Gamma(a_{0,\b,-\sigma_{0,\b}(z)})}[2\gamma_E\\
\quad+\psi(a_{0,\b,\sigma_{0,\b}(z)})+\psi(a_{0,\b,-\sigma_{0,\b}(z)})], & \a=0,\\[3mm]
\dfrac{\Gamma(1+\b)\Gamma(-\a)}{\Gamma(a_{-\a,\b,\sigma_{\a,\b}(z)})\Gamma(a_{-\a,\b,-\sigma_{\a,\b}(z)})}, & \a\in(0,1);
\end{cases}
\end{split}\\[1mm]
&\lb{C.14} \hspace{6.5cm} \b\in(-1,1),\\[1mm]
\no\begin{split}
\wti y_{2,\a,\b,-1}(z,1)&=\begin{cases}
\dfrac{2^{-\b}\Gamma(1-\b)\Gamma(-\a)}{\Gamma(a_{-\a,-\b,\sigma_{\a,\b}(z)})\Gamma(a_{-\a,-\b,-\sigma_{\a,\b}(z)})}, & \a\in(-1,0),\\[3mm]
\dfrac{-2^{\a+1}\Gamma(1+\a)\Gamma(1-\b)}{\Gamma(a_{\a,-\b,\sigma_{\a,\b}(z)})\Gamma(a_{\a,-\b,-\sigma_{\a,\b}(z)})}, & \a\in[0,1),
\end{cases}
\end{split}\\[1mm]
\no \begin{split}
\wti y_{2,\a,\b,-1}^{\, \prime}(z,1)&=\begin{cases}
\dfrac{2^{\a+1}\Gamma(1+\a)\Gamma(1-\b)}{\Gamma(a_{\a,-\b,\sigma_{\a,\b}(z)})\Gamma(a_{\a,-\b,-\sigma_{\a,\b}(z)})}, & \a\in(-1,0),\\[3mm]
\dfrac{-2^{-\b}\Gamma(1-\b)}{\Gamma(a_{0,-\b,\sigma_{0,\b}(z)})\Gamma(a_{0,-\b,-\sigma_{0,\b}(z)})}[2\gamma_E\\
\quad+\psi(a_{0,-\b,\sigma_{0,\b}(z)})+\psi(a_{0,-\b,-\sigma_{0,\b}(z)})], & \a=0,\\[3mm]
\dfrac{2^{-\b}\Gamma(1-\b)\Gamma(-\a)}{\Gamma(a_{-\a,-\b,\sigma_{\a,\b}(z)})\Gamma(a_{-\a,-\b,-\sigma_{\a,\b}(z)})}, & \a\in(0,1);
\end{cases}
\end{split}  \\[1mm]
& \hspace{6.25cm} \b\in(-1,1)\backslash\{0\},   \lb{C.15} \\[1mm]
\no \begin{split}
\wti y_{2,\a,0,-1}(z,1)&=\begin{cases}
\dfrac{-[2\gamma_E+\psi(a_{-\a,0,\sigma_{\a,0}(z)})+\psi(a_{-\a,0,-\sigma_{\a,0}(z)})] 
\Gamma(-\a)}{\Gamma(a_{-\a,0,\sigma_{\a,0}(z)})\Gamma(a_{-\a,0,-\sigma_{\a,0}(z)})},\\[1mm]
\hspace{7.45cm}\a\in(-1,0),\\[3mm]
\dfrac{[2\gamma_E+\psi(a_{\a,0,\sigma_{\a,0}(z)})+\psi(a_{\a,0,-\sigma_{\a,0}(z)})]\Gamma(1+\a)}{2^{-\a-1}\Gamma(a_{\a,0,\sigma_{\a,0}(z)})\Gamma(a_{\a,0,-\sigma_{\a,0}(z)})},\\[1mm]
\hspace{7.5cm}\a\in[0,1),
\end{cases}
\end{split}\\[1mm]
\begin{split} 
\wti y_{2,\a,0,-1}^{\, \prime}(z,1)&=\begin{cases}
\dfrac{[2\gamma_E+\psi(a_{\a,0,\sigma_{\a,0}(z)})+\psi(a_{\a,0,-\sigma_{\a,0}(z)})]\Gamma(1+\a)}{-2^{-\a-1}\Gamma(a_{\a,0,\sigma_{\a,0}(z)})\Gamma(a_{\a,0,-\sigma_{\a,0}(z)})},\\[1mm]
\hspace{7.5cm}\a\in(-1,0),\\[1mm]
-\Gamma([1+\sigma_{0,0}(z)]/2)\Gamma([1-\sigma_{0,0}(z)]/2)\\
+\dfrac{[2\gamma_E
+\psi([1+\sigma_{0,0}(z)]/2)+\psi([1-\sigma_{0,0}(z)]/2)]^2}{\Gamma([1+\sigma_{0,0}(z)]/2)
\Gamma([1-\sigma_{0,0}(z)]/2)},\quad \a=0,\\[3mm]
\dfrac{-[2\gamma_E+\psi(a_{-\a,0,\sigma_{\a,0}(z)})+\psi(a_{-\a,0,-\sigma_{\a,0}(z)})]\Gamma(-\a)}
{\Gamma(a_{-\a,0,\sigma_{\a,0}(z)})\Gamma(a_{-\a,0,-\sigma_{\a,0}(z)})},\\[1mm]
\hspace{7.5cm}\a\in(0,1);   
\end{cases}   
\end{split} \no \\[1mm] 
&\hspace{8cm} \b = 0.     \lb{C.16}   
\end{align}

Finally, we turn to the limiting behavior of $y_{j,\a,\b,1}(z,x)$ as $x\downarrow -1$ in some cases, applying the connection formulas \eqref{B.6}, \eqref{B.7} from Appendix \ref{sB}: 
\begin{align}
 y_{1,\a,\b,1}(z,x) &= \mathstrut_2F_1(a_{\a,\b,\sigma_{\a,\b}(z)},a_{\a,\b,-\sigma_{\a,\b}(z)};1+\a;(1-x)/2)  \no \\
& = \f{\Gamma(1+\a) \Gamma(-\b)}{\Gamma(a_{\a,-\b,\sigma_{\a,\b}(z)}) 
\Gamma(a_{\a,-\b,-\sigma_{\a,\b}(z)})} \no \\
&\qquad \times \mathstrut_2F_1(a_{\a,\b,\sigma_{\a,\b}(z)},a_{\a,\b,-\sigma_{\a,\b}(z)};1+\b;(1+x)/2)
\no \\
& \quad + \f{2^{\b} \Gamma(1+\a) \Gamma(\b)}{\Gamma(a_{\a,\b,\sigma_{\a,\b}(z)}) 
\Gamma(a_{\a,\b,-\sigma_{\a,\b}(z)})} (1+x)^{-\b}   \no \\
& \qquad \times \mathstrut_2F_1(a_{\a,-\b,\sigma_{\a,\b}(z)},a_{\a,-\b,-\sigma_{\a,\b}(z)};1-\b;(1+x)/2),   \lb{C.17} \\
& \hspace*{4.95cm} \a \in \R \backslash (- \N), \; \b \in \R \backslash \Z,   \no \\
y_{2,\a,\b,1}(z,x) &= (1-x)^{-\a} \mathstrut_2F_1(a_{-\a,\b,\sigma_{\a,\b}(z)},a_{-\a,\b,-\sigma_{\a,\b}(z)};1-\a;(1-x)/2) \no \\
&  = \f{2^{-\a} \Gamma(1-\a) \Gamma(-\b)}{\Gamma(a_{-\a,-\b,\sigma_{\a,\b}(z)}) 
\Gamma(a_{-\a,-\b,-\sigma_{\a,\b}(z)})}    \no \\
& \qquad \times \mathstrut_2F_1(a_{\a,\b,\sigma_{\a,\b}(z)},a_{\a,\b,-\sigma_{\a,\b}(z)};1+\b;(1+x)/2)
\no \\
& \quad + \f{2^{\b-\a} \Gamma(1-\a) \Gamma(\b)}{\Gamma(a_{-\a,\b,\sigma_{\a,\b}(z)}) 
\Gamma(a_{-\a,\b,-\sigma_{\a,\b}(z)})} (1+x)^{-\b}    \no \\
& \qquad \times \mathstrut_2F_1(a_{\a,-\b,\sigma_{\a,\b}(z)},a_{\a,-\b,-\sigma_{\a,\b}(z)};1-\b;(1+x)/2),  
\lb{C.18} \\
& \hspace*{5.7cm} \a \in \R \backslash \N, \; \b \in \R \backslash \Z;    \no \\
& \hspace*{5.8cm} z \in \bbC, \; x \in (-1,1).    \no 
\end{align}
In particular,
\begin{align}
& y_{1,\a,\b,1}(z,x) \underset{x \downarrow -1}{=} 
\f{\Gamma(1+\a) \Gamma(-\b)}{\Gamma(a_{\a,-\b,\sigma_{\a,\b}(z)}) 
\Gamma(a_{\a,-\b,-\sigma_{\a,\b}(z)})} [1 + \Oh(1+x)]    \no \\
& \hspace*{2.9cm} + \f{2^{\b} \Gamma(1+\a) \Gamma(\b)}{\Gamma(a_{\a,\b,\sigma_{\a,\b}(z)}) 
\Gamma(a_{\a,\b,-\sigma_{\a,\b}(z)})} (1+x)^{-\b} [1 + \Oh(1+x)],  \no \\
&  \hspace*{7.2cm} \a \in \R \backslash (- \N), \; \b \in \R \backslash \Z,   \lb{C.19} \\
& y_{2,\a,\b,1}(z,x) \underset{x \downarrow -1}{=} \f{2^{-\a} 
\Gamma(1-\a) \Gamma(-\b)}{\Gamma(a_{-\a,-\b,\sigma_{\a,\b}(z)}) 
\Gamma(a_{-\a,-\b,-\sigma_{\a,\b}(z)})} [1 + \Oh(1+x)]    \no \\
& \hspace*{2.9cm} + \f{2^{\b-\a} \Gamma(1-\a) \Gamma(\b)}{\Gamma(a_{-\a,\b,\sigma_{\a,\b}(z)}) 
\Gamma(a_{-\a,\b,-\sigma_{\a,\b}(z)})} (1+x)^{-\b}  [1 + \Oh(1+x)],    \no \\
& \hspace*{7.8cm} \a \in \R \backslash \N, \; \b \in \R \backslash \Z.    \lb{C.20} 
\end{align}

\section{Laguerre \texorpdfstring{$m$}{}-function as a limit of the Jacobi \texorpdfstring{$m$}{}-function} 
\lb{sD}

The Laguerre polynomials $L_n^\b(t)$ given by (see. \cite[Ch.~V]{Sz75})
\begin{align}
    L_n^\b(t): = \binom{n+b}{n} \ _1F_1(-n, \b+1; t), \quad t \in (0,\infty), \; n \in \mathbb{N}_0, \; \b \in (-1,\infty),
\end{align}
can be obtained as a limit of the Jacobi polynomials (cf.~\cite[eq.~18.7.22]{OLBC10}, \cite[eq.~5.3.4]{Sz75})
\begin{align} \lb{10.4.2}
    L^\b_n(t) = \lim_{\alpha \uparrow \infty} (-1)^n P_n^{\a,\b}((2t/\a) -1),   
\end{align}
where the convergence is locally uniform in the complex $t$-plane.
This is one example of the limiting relations in the Askey scheme \cite{AA85}.\footnote{The Askey scheme has been expanded by Koekoek and Swarttouw \cite{KS98} (see also \cite{TL01}) to include basic orthogonal polynomials.} Similarly, for the Laguerre differential expression given by 
\begin{align}
& \tau_{Lag}^\b = -t^{-\b}e^t \f{d}{dt}t^{\b + 1}e^{-t} \f{d}{dt}, \quad t \in (0,\infty), \; \b \in (-1,\infty),
\\\no
& p(t) = p_\b(t) = t^{\b+1} e^{-t}, \quad q(t) = 0, \quad r(t) = r_\b(t) = t^\b e^{-t},
\end{align}
one obtains 
\begin{align} \lb{10.4.4}
    \lim_{\alpha \uparrow \infty} \f{1}{\a}\tau_{Jac}^{\a,\b} = \tau_{Lag}^\b, \ \mbox{ with } \ x = (2t/\a) - 1.
\end{align}
Here convergence means local uniform convergence of the coefficients in the complex $t$-plane. Additionally, the eigenvalues $\l_n^{Lag}$ of the Laguerre polynomials are independent of the parameter $\b$ and can be obtained via
\begin{align} \lb{10.4.5}
    \l_n^{Lag} = n = \lim_{\a\uparrow \infty} \l_n^{\a,\b}/\a  , \quad n \in \mathbb{N}_0. 
\end{align}
For the weight functions we have the limiting relation
\begin{align}
    r_\b(t)= t^\b e^{-t} = \lim_{\alpha \uparrow \infty} 2^{-\a} (\a/2)^\b r_{Jac}^{\a,\b}((2t/a)-1).
\end{align}
The computation of the $m$-function for various boundary conditions has been performed in \cite[Sect.~6]{GLN20}, we note, however, that our parameter $\b$ is smaller by $1$. The $m$-function for the self-adjoint extension having the Laguerre polynomials as eigenfunctions is equal to the $m$-function corresponding to the Friedrichs extension in the case $\b \in (0,1)$, and holds for $\b \in (-1,1) \backslash \{0\}$
\begin{align} \lb{10.4.7}
    m_{Lag}^\b(w) = \f{\G(-\b)\G(-w)}{\G(\b+1)\G(-\b-w)}, \quad \b \in (-1,1)\backslash \{0\}, \; w \in \bbC \backslash \bbN_0.
\end{align}
Moreover from \eqref{10.4.4} or \eqref{10.4.5} we see that the correct scaling of the new spectral parameter $w$ as $\alpha \uparrow \infty$ should be $w = z/\a$, where $z$ is the spectral parameter of the Jacobi problem.

A fundamental system for the Laguerre problem is given in terms of confluent hypergeometric functions (see \cite[Ch.~6.3]{GLN20}):
\begin{align}
\begin{aligned}
    \widehat y_{1,\b}(w,t) &= {}_1F_1(-w, \b + 1; t), 
    \\
    \widehat y_{2,\b}(w,t) &= t^{-\b} {}_1F_1(-\b-w, 1-\b; t),
\end{aligned} \quad \b \in (-1,1)\backslash \{0\}, \; w \in \C, \; t \in (0,\infty).
\end{align}
One can check that $W_{Lag}(\widehat y_{2,\b}(w,\cdot), \widehat y_{1,\b}(w,\cdot)) = \b$. Generalizing \eqref{10.4.2} we can obtain this fundamental system as a limit of the Jacobi fundamental system found in \eqref{A.13} and \eqref{A.14}. For this we need a special case of \cite[eq.~16.8.10]{OLBC10}:
\begin{align} \lb{10.4.9}
    {}_1F_1(a,c;t) =  \lim_{b \to \infty} {}_2F_1(a,b,c;t/b),
\end{align}
and the observation that
\begin{align} \lb{10.4.10}
    [\a+\b+1-\sigma(\a t)]/2 & \underset{\alpha \uparrow \infty}{=} -t + O\big(\a^{-1}\big), 
    \\\no
    [\a+\b+1+\sigma(\a t)]/2 & \underset{\alpha \uparrow \infty}{=} \a +\b+1+t+O\big(\a^{-1}\big),
\end{align}
as $\alpha \uparrow \infty$. Together \eqref{10.4.9} and \eqref{10.4.10} imply that
\begin{align}
     \widehat y_{1,\b}(w,t) &= \lim_{\alpha \uparrow \infty} y_{1,\a,\b,-1}(\a t,(2t/\a)-1),
     \\\no
     \widehat y_{2,\b}(w,t) &= \lim_{\alpha \uparrow \infty} (2/\a)^\b y_{2,\a,\b,-1}(\a t,(2t/\a)-1); \quad w \in \bbC, \; t \in (0,\infty).
\end{align}
Taking into account the normalization chosen in \cite[Sect.~6]{GLN20} and in our paper, the following relation between the two $m$-functions \eqref{9.1} and \eqref{10.4.7} holds:
\begin{align} \lb{10.4.12}
    m_{Lag}^\b(w) = \lim_{\alpha \uparrow \infty} 2^{\a+1}(2/\a)^\b \hatt m_{\a,\b}(\a w).
\end{align}
Using \cite[eq.~6.1.47]{AS72} with \eqref{10.4.10} one can indeed check that the limit \eqref{10.4.12} holds locally uniformly for $w \in \C \backslash  \N_0$.

\section{Special Cases} \lb{sE}

\subsection{The Case of Gegenbauer or Ultraspherical Polynomials}
\hfill

In this section, we look at some of the special cases of the Jacobi differential expression in more detail beginning with the Gegenbauer (or ultraspherical) case (see, e.g., \cite[Ch. 22]{AS72}, \cite[Ch.~18]{OLBC10}, \cite[Ch. IV]{Sz75}). This can be realized by choosing the parameters $\a=\b=\mu-1/2$ so that
\begin{align}
p(x)=(1-x^2)^{\mu+1/2},\quad r(x)=(1-x^2)^{\mu-1/2},\quad q(x)=0,\quad x\in(-1,1),\quad \mu\in\bbR.
\end{align}
Thus, one considers the differential expression
\begin{align}
\tau_\mu=-(1-x^2)^{1/2-\mu}(d/dx)\big((1-x^2)^{\mu+1/2}\big)(d/dx),\quad x\in(-1,1),\ \mu\in\bbR,
\end{align}
noting  at the endpoints $x=\pm1$, $\tau_\mu$ is regular for $\mu\in(-1/2,1/2)$, in the limit circle case for $\mu\in[1/2,3/2)$, and in the limit point case for $\mu\in\bbR\backslash(-1/2,3/2)$. Furthermore, choosing $\mu=1/2$ one arrives at the Legendre equation once again. After the appropriate change in parameter, \eqref{5.4} describes the form of the $m$-function for the new parameter $\mu\in(-1/2,3/2)$, whereas Theorem \ref{t7.1} describes $\mu\in\bbR\backslash(-1/2,3/2)$. 

\subsection{The Case of Chebyshev Polynomials}
\hfill

The Chebyshev cases of the first and second kind are two more special cases. 

The Chebyshev case of the first kind is realized by choosing $\mu=0$ in the Gegenbauer case, or $\a=\b=-1/2$ in the Jacobi case (see, e.g., \cite[Ch. 22]{AS72}, \cite[Ch.~18]{OLBC10}, \cite[Ch. IV]{Sz75}). Thus, one considers the differential expression
\begin{align}
\tau_0 = - (1-x^2)^{1/2}(d/dx)\big((1-x^2)^{1/2}\big)(d/dx),\quad x\in(-1,1),
\end{align}
which is regular at $x=\pm1$. By \eqref{5.4} one can write
\begin{align}
\begin{split}
m_{\g,\d,-1/2,-1/2}(z)&=\big[-2^{-1/2}\cos(\g)(\cos(\d)\wti y_{1}(z,1)+\sin(\d)\wti y_{1}^{\, \prime}(z,1))\\
&\quad-\sin(\g)(\cos(\d)\wti y_{2}(z,1)+\sin(\d)\wti y_{2}^{\, \prime}(z,1))\big]\\
&\quad\times\big[-2^{-1/2}\sin(\g)(\cos(\d)\wti y_{1}(z,1)+\sin(\d)\wti y_{1}^{\, \prime}(z,1))\\
&\quad\quad+\cos(\g)(\cos(\d)\wti y_{2}(z,1)+\sin(\d)\wti y_{2}^{\, \prime}(z,1))\big]^{-1},\\
&\quad\hspace{1.85cm} \g,\d\in[0,\pi),\ z\in\rho(T_{\g,\d,-1/2,-1/2}),
\end{split}
\end{align}
with
\begin{align}
\wti y_{1}(z,1)&=
\f{\Gamma(1/2)\Gamma(1/2)}{\Gamma((1/2)+z^{1/2})\Gamma((1/2)-z^{1/2})}=\f{\pi}{\Gamma((1/2)+z^{1/2})\Gamma((1/2)-z^{1/2})},    \no \\[1mm]
\begin{split} 
\wti y_{1}^{\, \prime}(z,1)&=
\f{\Gamma(1/2)\Gamma(1/2)}{\Gamma(z^{1/2})\Gamma(-z^{1/2})}=\f{\pi}{\Gamma(z^{1/2})\Gamma(-z^{1/2})},   
\\[3mm]
\wti y_{2}(z,1)&=
\f{2^{1/2}\Gamma(3/2)\Gamma(1/2)}{\Gamma(1+z^{1/2})\Gamma(1-z^{1/2})}=\f{2^{-1/2}\pi}{\Gamma(1+z^{1/2})\Gamma(1-z^{1/2})},
\end{split} \\[1mm]
\wti y_{2}^{\, \prime}(z,1)&=
\f{2^{1/2}\Gamma(3/2)\Gamma(1/2)}{\Gamma((1/2)+z^{1/2})\Gamma((1/2)-z^{1/2})}=\f{2^{-1/2}\pi}{\Gamma((1/2)+z^{1/2})\Gamma((1/2)-z^{1/2})}.    \no 
\end{align}
Here we have used (cf. \cite[Eqs. 6.1.8, 6.1.9]{AS72}) $\Gamma(1/2)=\pi^{1/2}$, $\Gamma(3/2)=2^{-1}\pi^{1/2}$, 
and the fact that
\begin{align}
\sigma_{-1/2,-1/2}(z)=2 z^{1/2}.
\end{align}

\begin{example}[Neumann boundary conditions]\hfill

Considering the Neumann boundary conditions, one obtains
\begin{align}
m_{\pi/2,\pi/2,-1/2,-1/2}(z)&=2^{1/2}\f{\wti y_{2}^{\, \prime}(z,1)}{\wti y_{1}^{\, \prime}(z,1)}    \\
&=\f{\Gamma(z^{1/2})\Gamma(-z^{1/2})}{\Gamma((1/2)+z^{1/2})\Gamma((1/2)-z^{1/2})},\ z\in\rho(T_{\pi/2,\pi/2,-1/2,-1/2}),   \no 
\end{align}
which has simple poles at
\begin{align}
z_n=n^2,\quad n\in\bbN_0,
\end{align}
as expected since the Chebyshev polynomials of the first kind, $T_n(\dott)$, satisfy Neumann boundary conditions because of the relation $($cf. \cite[eq. 22.5.31]{AS72}$)$
\begin{align}
T_n(x)=\f{n! \pi^{1/2}}{\Gamma(n+(1/2))}P_n^{-1/2,-1/2}(x), \quad n \in \bbN_0, \; x \in (-1,1).
\end{align}
\end{example}

\medskip

The Chebyshev case of the second kind is realized by choosing $\mu=1$ in the Gegenbauer case, or $\a=\b=1/2$ in the Jacobi case (see, e.g., \cite[Ch. 22]{AS72}, \cite[Ch.~18]{OLBC10}, \cite[Ch. IV]{Sz75}). Thus, one obtains the differential expression
\begin{align}
\tau_1=-(1-x^2)^{-1/2}(d/dx)\big((1-x^2)^{3/2}\big)(d/dx),\quad x\in(-1,1),
\end{align}
which is in the limit circle case at $x=\pm1$. Thus by \eqref{5.4} one can write
\begin{align}
\begin{split}
m_{\g,\d,1/2,1/2}(z)&=\big[2^{1/2}\sin(\g)(\cos(\d)\wti y_{1}(z,1)+\sin(\d)\wti y_{1}^{\, \prime}(z,1))\\
&\quad +\cos(\g)(\cos(\d)\wti y_{2}(z,1)+\sin(\d)\wti y_{2}^{\, \prime}(z,1))\big]\\
&\quad\times\big[2^{1/2}\cos(\g)(\cos(\d)\wti y_{1}(z,1)+\sin(\d)\wti y_{1}^{\, \prime}(z,1))\\
&\quad\quad +\sin(\g)(\cos(\d)\wti y_{2}(z,1)+\sin(\d)\wti y_{2}^{\, \prime}(z,1))\big]^{-1},\\
&\quad\hspace{2.25cm} \g,\d\in[0,\pi),\ z\in\rho(T_{\g,\d,1/2,1/2}),
\end{split}
\end{align}
with
\begin{align}
\begin{split}
\wti y_{1}(z,1)&=
\f{-4\Gamma(3/2)\Gamma(3/2)}{\Gamma(1+(1+z)^{1/2})\Gamma(1-(1+z)^{1/2})}     \\
&=\f{-\pi}{\Gamma(1+(1+z)^{1/2})\Gamma(1-(1+z)^{1/2})},\\[1mm]
\wti y_{1}^{\, \prime}(z,1)&=
\f{\Gamma(3/2)\Gamma(-1/2)}{\Gamma((1/2)+(1+z)^{1/2})\Gamma((1/2)-(1+z)^{1/2})}    \\
&=\f{-\pi}{\Gamma((1/2)+(1+z)^{1/2})\Gamma((1/2)-(1+z)^{1/2})},\\[3mm]
\wti y_{2}(z,1)&=
\f{-2^{3/2}\Gamma(3/2)\Gamma(1/2)}{\Gamma((1/2)+(1+z)^{1/2})\Gamma((1/2)-(1+z)^{1/2})}    \\
&=\f{-2^{1/2}\pi}{\Gamma((1/2)+(1+z)^{1/2})\Gamma((1/2)-(1+z)^{1/2})},\\[1mm]
\wti y_{2}^{\, \prime}(z,1)&=
\f{2^{-1/2}\Gamma(3/2)\Gamma(1/2)}{\Gamma((1+z)^{1/2})\Gamma(-(1+z)^{1/2})}     
=\f{2^{-3/2}\pi}{\Gamma((1+z)^{1/2})\Gamma(-(1+z)^{1/2})}.
\end{split}
\end{align}
Here we have used (cf. \cite[Eqs. 6.1.8, 6.1.9]{AS72}) $\Gamma(-1/2)=-2\pi^{1/2}$, $\Gamma(1/2)=\pi^{1/2}$, $\Gamma(3/2)=2^{-1}\pi^{1/2}$, and the fact that
\begin{align}
\sigma_{1/2,1/2}(z)=2 (1+z)^{1/2}.
\end{align}

\begin{example}[The Friedrichs extension]\hfill

Considering the Friedrichs extension, one obtains
\begin{align}
m_{0,0,1/2,1/2}(z)&=2^{-1/2}\f{\wti y_{2}(z,1)}{\wti y_{1}(z,1)}     \\
&=\f{\Gamma(1+(1+z)^{1/2})\Gamma(1-(1+z)^{1/2})}{\Gamma((1/2)+(1+z)^{1/2})\Gamma((1/2)-(1+z)^{1/2})},\ z\in\rho(T_{0,0,1/2,1/2}),     \no 
\end{align}
which has simple poles at
\begin{align}
z_n=n(n+2),\quad n\in\bbN_0,
\end{align}
as expected since the Chebyshev polynomials of the second kind, $U_n(\dott)$, satisfy Friedrichs boundary conditions as they satisfy the relation $($cf. \cite[eq. 22.5.32]{AS72}$)$
\begin{align}
U_n(x)=\f{(n+1)!\sqrt{\pi}}{2\Gamma(n+(3/2))}P_n^{1/2,1/2}(x), \quad n \in \bbN_0, \; x \in (-1,1).
\end{align}
\end{example}

\subsection{The Case of Zernike Polynomials}
\hfill

The Zernike polynomials are a sequence of orthogonal polynomials on the unit disk that play important roles in various branches of optics and optometry (see, e.g., \cite{LF11}, \cite{MD07}). The even and odd Zernike polynomials are defined, respectively, as
\begin{align}
\begin{split}
Z_n^\ell(\rho,\varphi)&=R_n^\ell(\rho)\cos(\ell\varphi),\\
Z_n^{-\ell}(\rho,\varphi)&=R_n^\ell(\rho)\sin(\ell\varphi); \quad \ell,n\in\bbN_0,\ n\geq \ell\geq 0,\ 0\leq\rho\leq1,
\end{split}
\end{align}
with the radial polynomials, $R_n^\ell(\rho)$, defined by hypergeometric functions as
\begin{align}
R_n^\ell(\rho)=\begin{cases}
\f{(-1)^{(n-\ell)/2}((n+\ell)/2)!}{\ell!((n-\ell)/2)!\rho^{-\ell}} F\big(1+[(n+\ell)/2],-(n-\ell)/2;1+\ell;\rho^{2}\big),\\
\hfill n-\ell \text{ even},\\
0,\hfill n-\ell \text{ odd}.
\end{cases}
\end{align}
They are related to the Jacobi polynomials by (cf. \cite[eq. (14)]{LF11})
\begin{align}\lb{10.5.21}
R_n^\ell(\rho)=(-1)^{(n-\ell)/2}\rho^\ell P_{(n-\ell)/2}^{\ell,0}\big(1-\rho^2\big), \quad \ell,n\in\bbN_0,\ n\geq \ell\geq 0,\ 0\leq\rho\leq1.
\end{align}
This implies that letting $\a=\ell\in\bbN_0,\ \b=0$ in the Jacobi expression yields a representation of the radial Zernike differential expression as
\begin{align}
\tau_{\ell} = - (1-x)^{-\ell}(d/dx) \big((1+x)(1-x)^{\ell+1}\big) (d/dx),\quad \ell\in\bbN_0,\ x \in (-1,1),
\end{align}
which is always in the limit circle case at $x=-1$, but only in the limit circle case at $x=1$ when $\ell=0$ and in the limit point case otherwise, that is, for $\ell\in\bbN$. In particular, when $\ell=0$, this actually coincides with the Legendre case discussed in Example \ref{e5.3}. The case when $\ell\in\bbN$ was covered in Section \ref{s6}. Replacing $n$ by $(n-\ell)/2$ and $\alpha$ by $\ell$ in \eqref{6.9},
$n,\ell\in\bbN_0$, $n\geq \ell\geq 1$, one obtains for the Friedrichs extension
\begin{align}
\begin{split}
m_{F,\ell,0}(z)=-2^{-\ell-1} [2\gamma_E+\psi([\ell+1+\sigma(z)]/2)+\psi([\ell+1-\sigma(z)]/2)],&   \\
\ell\in\bbN_0,\ z\in\rho(T_{F,\ell,0}),&   
\end{split}
\end{align}
with simple poles at
\begin{align}
z_{n,\ell}=\frac{n-\ell}{2}\left(\frac{n+\ell}{2}+1\right),\quad n,\ell\in\bbN_0,\ n\geq \ell\geq 1,
\end{align}
as expected from relation \eqref{10.5.21}.

\medskip

\noindent 
{\bf Acknowledgments.} We are indebted to Dale Frymark, Conni Liaw, Andrei Martinez-Finkelshtein, Roger Nichols, and Brian Simanek  for very helpful discussions. 
 
 
\end{document}